\documentclass[11pt]{article}
\usepackage[utf8]{inputenc}
\usepackage[top=2.3cm, bottom=2.3cm,left=2.5cm, right=2.5cm]{geometry}
\usepackage[maxbibnames=99,backend=biber,
style=numeric-comp,giveninits=true,date=year,
doi=false,isbn=false,url=false
]{biblatex}
\AtEveryBibitem{%
  \clearfield{number}}
\DeclareFieldFormat{pages}{#1}
\renewbibmacro{in:}{%
  \ifentrytype{article}
    {}
    {\bibstring{in}%
     \printunit{\intitlepunct}}}
\DeclareFieldFormat[article,inbook,incollection,inproceedings,patent,thesis,unpublished]{title}{\mkbibemph{#1\isdot}}

\DeclareFieldFormat{journaltitle}{#1\isdot}

\usepackage{float}
\usepackage{graphicx}
\usepackage{outlines}
\usepackage{amsfonts}
\usepackage{amsmath}
\usepackage{amsthm} 
\usepackage{amssymb}
\usepackage{xfrac} 
\usepackage{nicefrac}
\usepackage{array}
\usepackage{tabulary}
\usepackage{comment}
\usepackage{pdfpages}
\usepackage{url}
\usepackage{comment}
\usepackage{pdfpages}
\usepackage{lscape}
\usepackage{parskip}
\usepackage{tikz-cd}
\usetikzlibrary{calc,shapes.geometric} 
\usepackage[font={footnotesize}]{caption}
\usepackage[font={footnotesize}]{subcaption}
\usepackage{titlesec}
\usepackage[hidelinks]{hyperref}
\usepackage{listings}
\usepackage{mdframed}
\usepackage{xcolor}
\usepackage{mathtools}
\usepackage{booktabs}
\usepackage[plainruled]{algorithm2e}
\usepackage{setspace}
\usepackage{tocloft} 
\usepackage{csquotes} 
\usepackage{epigraph}
\usepackage{indentfirst}
\usepackage{prettyref}
\newcommand{\pref}{\prettyref}
\newrefformat{fig}{Figure~\ref{#1}}
\newrefformat{tab}{Table~\ref{#1}}
\newrefformat{def}{Definition~\ref{#1}}
\newrefformat{prop}{Proposition~\ref{#1}}
\newrefformat{cor}{Corollary~\ref{#1}}
\newrefformat{rem}{Remark~\ref{#1}}
\newrefformat{obs}{Observation~\ref{#1}}
\newrefformat{ex}{Example~\ref{#1}}
\newrefformat{lem}{Lemma~\ref{#1}}
\newrefformat{sec}{\S\ref{#1}}
\newrefformat{subsec}{\S\ref{#1}}
\usepackage[colorinlistoftodos, textsize=small]{todonotes} 
\usepackage{enumitem}
\setlist[itemize]{topsep=0.3em,leftmargin=1.9em}
\setlist[enumerate]{topsep=0.2em,leftmargin=2.0em}
\newcommand\Item[1][]{%
  \ifx\relax#1\relax  \item \else \item[#1] \fi
  \abovedisplayskip=0pt\abovedisplayshortskip=0pt~\vspace*{-\baselineskip}}

\newcommand{\NN}{\mathbb{N}} 
\newcommand{\ZZ}{\mathbb{Z}} 
\newcommand{\QQ}{\mathbb{Q}} 
\newcommand{\RR}{\mathbb{R}} 
\newcommand{\CC}{\mathbb{C}} 
\newcommand{\hCC}{\widehat{\mathbb{C}}} 
\newcommand{\hCA}{\widehat{\mathbb{A}}} 
\newcommand{\DD}{\mathbb{D}} 
\newcommand{\HH}{\mathbb{H}} 
\newcommand*{\medcup}{\mathbin{\scalebox{1.2}{\ensuremath{\cup}}}} 

\newcommand{\shortminus}{\scalebox{0.7}[1.0]{\( - \)}}
\newcommand{\shortplus}{\raisebox{.1565475463\height}{\scalebox{0.7}{\( + \)}}}
\providecommand{\keywords}[1]{\noindent\small\textbf{\textit{Keywords---}} #1} 
\theoremstyle{plain}
\newtheorem{theorem}{Theorem}[section]
\newtheorem{lemma}[theorem]{Lemma}
\newtheorem{proposition}[theorem]{Proposition}
\newtheorem{corollary}[theorem]{Corollary}

\newtheorem{innercustomgeneric}{\customgenericname}
\providecommand{\customgenericname}{}
\newcommand{\newcustomtheorem}[2]{%
  \newenvironment{#1}[1]
  {%
   \renewcommand\customgenericname{#2}%
   \renewcommand\theinnercustomgeneric{##1}%
   \innercustomgeneric
  }
  {\endinnercustomgeneric}
}
\newcustomtheorem{customthm}{Theorem}
\newcustomtheorem{customcor}{Corollary}
\theoremstyle{definition}
\newtheorem{definition}[theorem]{Definition}
\newtheorem*{definition*}{Definition}
\newtheorem{remark}[theorem]{Remark}
\newtheorem*{remark*}{Remark}

\newtheorem*{observation*}{Observation}

\newtheorem{example}[theorem]{Example}
\newcustomtheorem{customdef}{Definition}

\newcustomtheorem{customproof}{Proof}
\usepackage{fancyhdr}
\numberwithin{equation}{section}
\makeatletter
    \def\@algocf@capt@plainruled{above}
    \renewcommand{\algocf@caption@plainruled}{%
    \vskip\AlCapSkip%
    \box\algocf@capbox%
    \vskip 8\algoheightrule}
    \titleformat{\section}{\Large\bfseries}{\thesection}{0.5em}{}
\makeatother
\DontPrintSemicolon

\begin{document}
\pagenumbering{arabic}
\title{\vspace{-5.5mm}\LARGE\scshape Dynamics of projectable functions:\\ Towards an atlas of wandering domains\\ for a family of Newton maps\vspace{-2.00mm}}
\author{Robert Florido\thanks{e-mail: \href{mailto:robert.florido@ub.edu}{\textit{robert.florido@ub.edu}}} $^1$ \ $\cdot$ \ Núria Fagella\thanks{This work is supported by the (a) Spanish State Research Agency, through the Severo Ochoa and Mar\'ia de Maeztu Program for Centers and Units of Excellence in R\&D (CEX2020-001084-M), and PID2020-118281GB-C32; (b) Generalitat de Catalunya through the grants 2017SGR1374 and ICREA Academia 2020.} $^{1,2}$}
\date{\vspace{-1.5mm}
{\small $^1$Departament de Matem\`atiques i Inform\`atica, Universitat de Barcelona, Barcelona, Spain}\\%
{\small $^2$Centre de Recerca Matem\`atica, Barcelona, Spain}\\[2ex]%
September 4, 2023} 
\maketitle
\vspace{-1.25mm}
\begin{abstract}
\vspace{-0.75mm}
We present a one-parameter family $F_\lambda$ of transcendental entire functions with zeros, whose Newton's method yields wandering domains, coexisting with the basins of the roots of $F_\lambda$. Wandering domains for Newton maps of zero-free functions have been built before by, e.g., Buff and R\"uckert \cite{Buff2006} based on the lifting method. This procedure is suited to our Newton maps as members of the class of projectable functions (or maps of the cylinder), i.e. transcendental meromorphic functions $f(z)$ in the complex plane that are semiconjugate, via the exponential, to some map $g(w)$, which may have at most a countable number of essential singularities. 

In this paper we make a systematic study of the general relation (dynamical and otherwise) between $f$ and $g$, and inspect the extension of the logarithmic lifting method of periodic Fatou components to our context, especially for those $g$ of finite-type. We apply these results to characterize the entire functions with zeros whose Newton's method projects to some map $g$ which is defined at both $0$ and $\infty$. The family $F_\lambda$ is the simplest in this class, and its parameter space shows open sets of $\lambda$-values in which the Newton map exhibits wandering or Baker domains, in both cases regions of initial conditions where Newton's root-finding method fails.
\end{abstract}
\vspace{-2.00mm}
\vspace{2.0mm}
%
\section{Introduction}
\label{sec:1_Intro}

The iteration theory of meromorphic functions has been a primary focus of recent research in complex dynamics, investigating the possible extension of celebrated theorems in rational dynamics and the occurrence of new phenomena (see \cite{Bergweiler1993} for a comprehensive survey on the field). \textit{Transcendental meromorphic functions} $f:\CC\to\hCC$ are holomorphic except for isolated poles on the complex plane that may accumulate at the essential singularity at $\infty$. They naturally arise, for example, from the popular Newton's root-finding method applied to entire functions. In order to study the long-term behavior of arbitrary points under iteration, we split the \textit{Riemann sphere} $\hCC:=\CC\cup\{\infty\}$ into two completely invariant sets: the \textit{Fatou set} $\mathcal{F}(f)$ as the maximal open set on which the family of iterates $\{f^n\}_{n\in\NN}$ is defined and normal (or equicontinuous); and the \textit{Julia set} $\mathcal{J}(f):=\hCC\backslash \mathcal{F}(f)$, its chaotic complement. If $U$ is a \textit{Fatou component} of $f$, i.e. a connected component of $\mathcal{F}(f)$, then $f^n(U)$ is contained in a Fatou component $U_n$ for each $n\in\NN$, and $U_1\backslash f(U)$ contains at most two points \cite{Herring1998}. If $U_n\neq U_m$ for all $n\neq m$, then $U$ is called a \textit{wandering component} (or \textit{wandering domain}); otherwise $U$ is eventually $p$-\textit{periodic}, where $p\geq 1$ is the smallest such that $U_{k+p}=U_k$ for some $k\in\NN$.

\vspace{-1mm}
It is well-known that Newton's method $N_F(z):=z-\frac{F(z)}{F'(z)}$, where $F:\CC\to\CC$ is an entire function, may fail to converge to the roots of $F$, i.e. to the (attracting) fixed points of $N_F$. This happens not only if the initial condition $z_0$ is chosen in the Julia set, but also if some iterate of $z_0$ falls into a periodic cycle of Fatou components not containing the roots of $F$ as showcased in \cite{Curry1983}, or even into a chain of wandering domains. Although there are several conditions that rule out the existence of wandering domains in this context (see e.g. \cite{Baranski2020} and \cite[\S 4]{Ruckert2008}), explicit Newton's methods with wandering domains were given in \cite{Bergweiler1994} and \cite{Buff2006}. However, these examples were associated to zero-free functions $F$, so that there were indeed no roots to be found. In the present work we display the first, to our knowledge, explicit families of Newton maps which show that wandering domains and attracting invariant basins can, and often do, coexist. Our construction is based on the logarithmic lifting method due to Herman \cite{Herman1985}. More precisely, we shall build these wandering domains by lifting (via an exponential map) certain periodic Fatou components of a function $g$, which is semiconjugate through the exponential to a Newton's method. This leads to the following class of meromorphic functions in which the use of such technique makes sense. Denote by $\exp_\tau(z):=e^{\sfrac{2\pi i z}{\tau}}$ the exponential of period $\tau\in\CC^*:=\CC\backslash\{0\}$, and by $S_0(\tau):=\left\{z:-\frac{1}{2}<\operatorname{Re}\frac{z}{\tau}\leq\frac{1}{2}\right\}$ its \textit{(fundamental) period strip}.

\vspace{-0.4mm}
\begin{definition}[Projectable functions]
    \label{def:A_proj}
    Let $f:\CC\to\hCC$ be a transcendental meromorphic function. We say that $f$ is \textit{projectable} via $\exp_\tau$ if there exists a function $g$, its \textit{exponential projection}, satisfying
    \vspace{-2.25mm}
    \begin{equation}\label{eq:Projectable}
        g \circ \exp_\tau = \exp_\tau \circ f
        \vspace{-0.8mm}
    \end{equation}
    whenever defined, where $\tau\in\CC^*$.
\end{definition}

\vspace{-0.4mm}
By the preliminary change of variable $z\mapsto \tau z$, we assume without loss of generality that $\tau=1$. Given $g$, the function $f$ is unique up to an integer, and it is called a \textit{logarithmic lift} of $g$. Here we point out that $\exp_1$ induces a conformal isomorphism from the \textit{(extended) cylinder} $\hCA := \CC/\ZZ \cup \{\pm i\infty \}$ onto $\hCC$, which sends the upper end $+i\infty$ to $0$, and the lower end $-i\infty$ to $\infty$. Thus we may say that projectable functions $f$ quotient down to \textit{meromorphic maps of the cylinder}, whose domain of definition will be specified later. This naturally raises several questions about the nature of $f$ and the dynamical relationship with its projection via $\exp_1$. 

We wish to address such questions from a general point of view, and therefore we start by identifying the structure of projectable functions via the exponential. For this purpose, in \pref{sec:2_Pseudo} we transfer to the complex plane the notion of (simply and doubly) \textit{pseudoperiodic maps} in the sense of Arnol'd \cite{Arnold1999} (see \pref{def:Pseudoperiodic}). This is in line with the work of Brady \cite{Brady1970} who studied doubly pseudoperiodic functions, as a generalization of the Weierstrass $\zeta$-function. Pseudoperiodic maps turn out to be the sum of a linear and a periodic map, and characterize projectable functions as in the following theorem.
Recall that the set of all periods of a non-constant periodic meromorphic function on $\CC$ forms either a $1$-dimensional lattice $\tau_1\ZZ$ (the \textit{simply periodic} ones, e.g. $e^{\sfrac{2\pi i z}{\tau_{\scriptstyle 1}}}$), or a $2$-dimensional lattice $\tau_1\ZZ+\tau_2\ZZ$ (the \textit{doubly periodic} or \textit{elliptic} ones, e.g. the Weierstrass $\wp$-function), where $\tau_1,\tau_2\in\CC^*$ have non-real ratio, say $\sfrac{\tau_2}{\tau_1}\in\HH^+:=\{z:\operatorname{Im}{z}>0\}$ (see e.g. \cite{Ahlfors1979}).

\vspace{-0.5mm}
\begin{customthm}{A}[Form of projectable functions]
    \label{thm:A_FormProj}
    The class of projectable functions $f$ via $\exp_1$ coincides with the class of non-affine pseudoperiodic maps such that $f(z+1) = f(z) + \ell$ for all $z$, with $\ell\in\ZZ$. They can be written uniquely as
    \vspace{-0.15mm}
    \begin{equation}
        \label{eq:ProjPeriodic}
        f(z) = \ell z + \Phi(e^{2\pi i z}),
        \vspace{0.25mm}
    \end{equation}
    where $\Phi$ is a non-constant meromorphic function in $\CC^*$. Furthermore, $f$ is doubly pseudoperiodic (i.e. we also have $f(z+\tau)=f(z)+\eta_\tau$ for some $\tau\in\HH^+$ and $\eta_\tau\in\CC$) if and only if
    \vspace{-0.5mm}
    \begin{equation}
        \label{eq:ProjUnique}
        \Phi(e^{2\pi i z}) =  \frac{\ell\tau-\eta_{\footnotesize\tau}}{2\pi i}\Big(\zeta{(z)}-2\hspace{0.3mm}\zeta{(\sfrac{1}{2})}\hspace{0.4mm} z\Big)+E(z),
        \vspace{-1.5mm}
    \end{equation}
    where $\zeta$ is the Weierstrass $\zeta$-function with respect to $\ZZ+\tau\ZZ$, and $E$ is a doubly periodic function with periods $1$ and $\tau$. In particular, $f$ is also projectable via $\exp_\tau$ when $\eta_\tau=L\tau$ for some $L\in\ZZ$.
\end{customthm}

Hence, any projectable function via $\exp_1$ can be written as the sum of a linear map and a periodic function $\Phi\circ\exp_1$ as above, which is either simply periodic, doubly periodic, or a linear combination of those, with $1$ as a period. We refer to \cite{Hawkins2022} for an example of the dynamics of a doubly pseudoperiodic function, where two different directions of projection exist. In the case of entire projectable functions, the dynamics of their exponential projections $g$, which are holomorphic self-maps of $\CC^*$, has been widely studied, especially when both $0$ and $\infty$ are essential singularities of $g$ (see e.g. \cite{Bergweiler1995j,Keen1988,Kotus1987,Makienko1992,Marti-Pete2016}), following the early work of R\"adstrom \cite{Radstrom1953}. In other words, these projectable functions project down to holomorphic branched coverings of $\CC^*$, which may not be well-defined at $0$ or $\infty$ (i.e. the ends of the cylinder). However, in the non-entire case, the map $g$ is not going to be defined at all points of $\CC^*$.

In this regard, given a projectable function $f$, in \pref{sec:3_ExpProjection} we start the study of its exponential projection $g$ by showing that poles of $f$ correspond via $\exp_1$ to essential singularities of $g$ in $\CC^*$ (see \pref{prop:PolesToEssential}). Thus $g$ belongs to \textit{Bolsch's class} $\mathbf{K}$ \cite{Bolsch1997} of meromorphic functions with countably many essential singularities, which is the smallest class that contains all transcendental meromorphic maps and is closed under composition. We denote by $\mathcal{D}(g):=\hCC\backslash\mathcal{E}(g)$ the domain of definition of $g$, where $\mathcal{E}(g)$ is the set of essential singularities. We can show that both $0$ and $\infty$ (the omitted values of the exponential) are in $\mathcal{D}(g)$ if and only if $f$ belongs to the following class of (non-entire) projectable functions via $\exp_1$ (see \pref{prop:Regular0inf}).

\begin{definition}[Class $\mathbf{R}_\ell$]
    \label{def:classR}
    Denote by $\mathbf{R}_\ell$, $\ell\in\ZZ$, the class of meromorphic functions of the form
    \vspace{-0.75mm}
    \begin{equation}
        f(z)=\ell z + R(e^{2\pi i z}),
        \vspace{-0.75mm}
    \end{equation}
    where $R$ is a non-constant rational map such that $R(0)\neq \infty$ and $R(\infty)\neq \infty$. To be precise,
    \vspace{-1mm}
    \begin{equation}
        \label{eq:defRclassRl}
        R(w) = \frac{a_n w^n + \dots + a_0}{b_m w^m + \dots + b_0}
        \vspace{-1mm}
    \end{equation}
    is the ratio of two coprime polynomials with $m\geq \max{\{n,1\}}$, and $a_n, b_m, b_0\in\CC^*$, i.e. $R^{-1}(\infty)\neq \emptyset$.
\end{definition}

Given $f\in\mathbf{R}_\ell$, its exponential projection $g:\hCC\hspace{0.2mm}\backslash \exp_1 (R^{-1}(\infty))\to\hCC$ is written as
\vspace{-0.75mm}
\begin{equation}
    g(w)=w^\ell e^{2\pi i R(w)},
    \vspace{-0.75mm}
\end{equation}
for which $0$ and $\infty$ are fixed points if $\ell>0$, a cycle of period $2$ if $\ell<0$, or omitted values if $\ell=0$. Clearly, each essential singularity of $g$ has as many preimages as the degree of $R$, counted with multiplicity, which we call \textit{(essential) poles} of $g$ (see \pref{def:BolschClass}). Observe that $g$ would be a transcendental meromorphic map, i.e. $\#\mathcal{E}(g)=1$, as long as $f\in\mathbf{R}_\ell$ with exactly one pole in the period strip $S_0$ of $\exp_1$. In general, this is also possible just for some projectable entire functions (see \pref{rem:EntireCase}). In our pursuit of wandering domains for Newton's methods with fixed points, the class $\mathbf{R}_\ell$ is going to be central as will become clear later.

As an example, Newton's method of $\sin{\pi z}$ is in class $\mathbf{R}_1$, with $R(w)=\frac{i}{\pi}\frac{w-1}{w+1}$, while the Arnol'd standard map lies outside of $\mathbf{R}_\ell$ since $\{0,\infty\}\subset R^{-1}(\infty)$ (see \pref{ex:standardWeierstrass}, and \cite{Fagella1999}). Any $f\in\mathbf{R}_\ell$ may seen as a map defined on the extended cylinder $\hCA$, which is holomorphic outside of the canonical projection of $f^{-1}(\infty)$ on $\CC/\ZZ$. 
These functions correspond to pseudoperiodic analogues of the (periodic) maps $R\circ \exp_1$ which were studied in \cite{Baranski1995} (concerning the dimension of Julia sets), although here, as in \cite[\S~6]{Kotus2008}, we allow the set of singularities of the inverse function to intersect the Julia set.

\textit{Singular values} play a pivotal role in complex dynamics. For a given $g\in\mathbf{K}$, these are points $v$ in the range of $g$ for which some branch of its inverse $g^{-1}$ fails to be defined in any neighborhood of $v$. They are either \textit{critical values}, \textit{asymptotic values}, or limit points of those. The critical value set $\mathcal{CV}(g)$ consists of images of \textit{critical points} of $g$ which, in Bolsch's class, correspond not only to points $c\in\mathcal{D}(g)$ such that $g'(c)=0$, but also to multiple preimages of an essential singularity. The asymptotic value set $\mathcal{AV}(g)$ corresponds to those $v\in\hCC$ for which there is an \textit{asymptotic path} $\gamma:[0,1)\to\mathcal{D}(g)$ such that $\gamma(t)\to \hat{w}\in\mathcal{E}(g)$ and $g(\gamma(t))\to v$ as $t\to 1$. We denote the set of singular values of $g$ by \vspace{-0.2mm}
\begin{equation}
    \label{eq:Intro_SingularValues}
    \mathcal{S}(g):=\overline{\mathcal{CV}(g)\cup \mathcal{AV}(g)}, \vspace{-0.1mm}
\end{equation}
where the closure is taken in $\hCC$. Then, $g:\hCC\hspace{0.3mm} \backslash\hspace{-0.1mm}  \left( \mathcal{E}(g)\medcup g^{-1}(\mathcal{S}(g)) \right) \to \hCC\hspace{0.2mm} \backslash\hspace{0.2mm} \mathcal{S}(g)$ is a covering map. The relevance of $\mathcal{S}(g)$ becomes clear in close relation to the different types of periodic Fatou components.

The well-known classification of periodic Fatou components for meromorphic functions also hold for maps in Bolsch's class $\mathbf{K}$, and beyond (see \cite{Baker2001}, and references therein). Recall that a periodic point $w_0\in\mathcal{D}(g)$, as well as the cycle to which it belongs, is called \textit{attracting}, \textit{indifferent}, or \textit{repelling} if the modulus of its multiplier (i.e. $|(g^p)'(w_0)|$) is less than, equal to, or greater than $1$, respectively, where $p\geq 1$ is the smallest natural such that $g^p(w_0)=w_0$. In particular, $w_0$ is said to be \textit{parabolic} if the multiplier is $e^{2\pi i \rho}$ with $\rho\in\QQ$, while it is of \textit{Siegel} type if $\rho\notin\QQ$ and a local linearization is possible. Given that $\mathcal{F}(g)=\mathcal{F}(g^p)$, it is enough to classify an invariant Fatou component of $g\in\mathbf{K}$: it can be either a basin of attraction of an attracting or parabolic fixed point, a \textit{Siegel disk} or \textit{Herman ring} on which $g$ is conformally conjugate to an irrational rotation of a disk or annulus, respectively (called \textit{rotation domains}), or a \textit{Baker domain} on which the iterates of $g$ tend to an essential singularity of $g$ (see \cite[Thm.~C]{Baker2001}). It is known that any cycle of basins of attraction (of an attracting or parabolic cycle) must contain at least one singular value, and all boundary components of a cycle of Siegel disks or Herman rings are in the closure of forward orbits of values in $\mathcal{S}(g)$ \cite[Lem.~10]{Baker2001}.

In our context, we emphasize that a projectable function $f(z)$ has, in general, infinitely many poles and singular values accumulating at $\infty$, that is, $f$ lies outside of the so-called Eremenko-Lyubich class $\mathcal{B}$ \cite{Eremenko1992}, as $\mathcal{S}(f)\cap\CC$ is not bounded. This is always the case if $f$ is not $1$-periodic, due to pseudoperiodicty: $f(z+k)=f(z)+\ell k$ for all $k\in\ZZ$. The crucial point in our discussion is that, by the global change of coordinates $w(z):=e^{2\pi i z/\tau}$, we transfer the analysis to a Bolsch's function $g(w)$ with simpler dynamics. In particular, this occurs when $g$ is a \textit{finite-type} map (i.e. $\#\mathcal{S}(g)<\infty$), since $g$ has no wandering components nor Baker domains (see e.g. \cite{Baker2001, Eremenko1992}). The correspondence between critical and asymptotic values of $f$ and $g$ (\pref{rem:projCV} and \pref{prop:projAV}; see also \pref{prop:projC} and \pref{fig:3_projAV}) leads to the following result.

\begin{customthm}{B}[Projections of finite-type]
    \label{thm:B_FiniteType}
    Let $f$ be a projectable function via $\exp_1$, written as $f(z)=\ell z + \Phi(e^{2\pi i z})$ for some $\ell\in\ZZ$ and $\Phi$ meromorphic in $\CC^*$, $g$ its exponential projection, and $\widetilde{S}_0:=\left\{z:-\frac{\ell}{2}<\operatorname{Re}z\leq\frac{\ell}{2}\right\}$. Then $g$ is of finite-type if and only if one of the following holds:
    \begin{enumerate}[label = (\roman*)] 
        \vspace{-1.5mm}
        \item (Non-1-periodic case) $\ell\neq 0$ with $\#\big( \mathcal{S}(f) \cap \widetilde{S}_0\big)<\infty$ and, in addition, $\eta_\tau\in\QQ$ in the case that $f$ is doubly pseudoperiodic with $f(z+\tau)=f(z)+\eta_\tau$ for some $\tau\in\HH^+$.
        \vspace{-1.0mm}
        \item (1-periodic case) $\ell=0$ with $\#\exp_1\hspace{-0.2mm}\left( \mathcal{S}(\Phi)\backslash\{\infty\}\right)\hspace{-0.2mm}<\hspace{-0.2mm}\infty$.
        \vspace{-1.0mm}
    \end{enumerate}
    Furthermore, $\mathcal{S}(g)\cap\CC^*=\exp_1\hspace{-0.2mm}\left( \mathcal{S}(f)\backslash\{\infty\}\right)$ in case \textit{(i)}, $\mathcal{S}(g)\cap\CC^*=\exp_1\hspace{-0.2mm}\left( \mathcal{S}(\Phi)\backslash\{\infty\}\right)$ in case \textit{(ii)}, and $\{0,\infty\}\subset \mathcal{AV}(g)$ for both of them. Additionally, $0$ (resp. $\infty$) belongs to $\mathcal{CV}(g)$ if and only if $f$ is simply pseudoperiodic with $|\ell|\geq2$ and $g^{-1}(0)$ (resp. $g^{-1}(\infty)$) is outside of $\mathcal{E}(\Phi)\cup \Phi^{-1}(\infty)$.
\end{customthm}

Observe that for such a $g$ of finite-type, the corresponding projectable function $f$ may have infinitely many critical values, but none in the period strip $S_0$ of $\exp_1$, as e.g. the double standard map $f(z)=2z+1-\frac{1}{\pi}\sin{2\pi z}$, with $\mathcal{CV}(f)=2\ZZ+1$. The periodic case of this theorem includes all doubly periodic functions $f$ (since they are known to have finitely many critical values and no asymptotic values), and even functions $f$ which are not of finite-type (see \pref{ex:InfinitelyCinS1}). Note also that projections $g$ of functions in the class $\mathbf{R}_\ell$ are all of finite-type, with finitely many essential poles and critical values, and $0$ and $\infty$ as asymptotic values (\pref{cor:classRsv}), although the converse is not true (consider e.g. projections of the Arnol'd family). This, together with the control of the multipliers of the points at $0$ and $\infty$, makes $\mathbf{R}_\ell$ (with $\ell\in\ZZ^*$) an excellent class of meromorphic maps to deliver Baker and wandering domains by the lifting method (see \pref{sec:4_Lifting}).

Based on Herman's idea \cite{Herman1985} (detailed by Baker \cite[\S 5]{Baker1984} in the entire case), we produce wandering domains for a (meromorphic) projectable functions $f$, by lifting those periodic Fatou components of the projection $g$ which are not the image under $\exp_1$ of a periodic component of $\mathcal{F}(f)$. In general, this can be done by identifying the following class of points of $f$, all of which project via $\exp_1$ to periodic points of $g$ (\pref{lem:characPseudoPoint}). %
\begin{definition}[Pseudoperiodic points]
    \label{def:PseudoperiodicPoint}
    Let $f$ be a projectable function via $\exp_1$. We say that $z^*\in\CC\backslash f^{-1}(\infty)$ is a \textit{pseudoperiodic point} of \textit{type} ($p,\sigma$) of $f$ if, for some $p\geq 1$ and $\sigma\in\ZZ$, 
    \vspace{-0.95mm}
    \begin{equation}
        f^p(z^*) = z^* + \sigma.  
        \vspace{-1.05mm}
    \end{equation} 
    It is said to be of \textit{minimal type}, or \textit{$(p,\sigma)$-pseudoperiodic}, if $p\geq 1$ is the smallest natural with this property.
\end{definition}

In our case, to relate the dynamics of $f$ and $g$,  we rely on a theorem by Zheng \cite[Cor.~3.1]{Zheng2005}, based on Bergweiler's result \cite{Bergweiler1995j} in the entire setting. It essentially states that the Fatou and Julia sets of $f$ and $g$ are in correspondence via the exponential, that is,
\vspace{-0.55mm}
\begin{equation}
        \label{eq:expFatouJulia}
        \exp_1 \mathcal{F}(f) = \mathcal{F}(g)\cap \CC^*, \qquad \exp_1\hspace{-0.2mm}\left( \mathcal{J}(f) \hspace{0.3mm} \backslash \hspace{0.3mm} \{\infty\} \right) = \mathcal{J}(g)\cap \CC^*. \vspace{-0.55mm}
\end{equation}
Therefore, a Fatou component, say $U$, of $f$ projects under $\exp_1$ to a Fatou component, say $V$, of $g$, with $V\cap\CC^*=\exp_1 U$. Conversely, the component $V\subset\mathcal{F}(g)$ lifts via $\exp_1$ to $\{U+k\}_{k\in\ZZ}\subset\mathcal{F}(f)$, that is, either to infinitely many distinct Fatou components of $f$, or to only one (see \pref{lem:FillBaker}). 

Clearly, $U$ and $V$ do not need to be of the same type. On the one hand, we can build \textit{escaping wandering domains} $U$ (those for which $\infty$ is the only limit function of $\{f^n|_U^{}\}_{n\in\NN}$), which may be bounded or unbounded, by detecting appropriate pseudoperiodic points of $f$ (see \pref{cor:WDviaPseudoperiodicity}). Hence, we construct wandering domains by lifting via $\exp_1$ some periodic component $V\subset\mathcal{F}(g)$, without leaving the family of projectable functions $f$ under consideration, in contrast to the usual procedure of adding an integer to $f$ (one can think on Newton maps, depending on a parameter). Recall that the latter consists in changing the choice of the logarithmic lift of $g$, which turns a periodic component of $f$ directly into a wandering domain of $f+\sigma$, $\sigma\in\ZZ^*$.

On the other hand, we may produce Baker domains of $f$ by lifting certain periodic components $V\subset\mathcal{F}(g)$ related somehow to the points at $0$ and $\infty$ (i.e. the projection of the upper and lower ends of $\CC/\ZZ$, respectively), especially for those projections $g$ of finite-type (see \pref{thm:logFatouJulia}). Examples \ref{ex:MeroStandard} and \ref{ex:LiftParabolicBR} (see also \pref{fig:Blaschke_fHR}) illustrate the different possibilities that may occur; see \pref{subsec:PeriodicCase} for the case where $f$ is periodic.

Finally in \pref{sec:5_Newton} we apply the general theory for projectable functions developed in sections \ref{sec:3_ExpProjection} and \ref{sec:4_Lifting} to the special case of Newton's methods, our original motive. We start by characterizing Newton maps in class $\mathbf{R}_\ell$ with (attracting) fixed points, which are the natural candidates for our constructions. 
\begin{customthm}{C}[Newton's methods in class $\mathbf{R}_\ell$ with fixed points]
    \label{thm:C_AtlasBWD}
    Let $F$ be an entire function with zeros, and $\ell\in\ZZ$. Its Newton map $N_F$ is in class $\mathbf{R}_\ell$ if and only $\ell=1$ and \vspace{-0.6mm}
    \begin{equation}
        \label{eq:ThmC_ExpressionF}
        F(z) = e^{\Lambda z} \Psi(e^{2\pi i z}), \quad \mbox{ with } \quad \Psi(w) = w^{m_0} P(w) e^{Q(w)+\widetilde{Q}(\sfrac{1}{w})}, \vspace{-0.3mm}
    \end{equation}
    where $\Lambda\in\CC$, $m_0\in\ZZ$, and $P$, $Q$, $\widetilde{Q}$ are polynomials with $P(0)\neq 0$ and $P^{-1}(0)\cap\CC^*\neq \emptyset$. In addition, $\Lambda \neq -2\pi i (m_0+\deg{P})$ if $Q$ is constant, and  $\Lambda \neq -2\pi i m_0$ if $\widetilde{Q}$ is constant.
\end{customthm}

This provides uniparametric families of projectable Newton maps $N_\Lambda\in\mathbf{R}_1$ of entire functions satisfying $F(z+1)=e^{\Lambda} F(z)$ for all $z$, which take the form \vspace{-1.25mm}
\begin{equation}
    \label{eq:NewtonProjR1}
    N_\Lambda(z) = z+R_\Lambda(e^{2\pi i z}), \quad \mbox{ where } \quad R_\Lambda(w) = -\frac{\Psi(w)}{\Lambda\Psi(w)+2\pi i w \Psi'(w)}.        \vspace{-0.75mm}
\end{equation} 
As a Newton map in class $\mathbf{R}_1$, the points at $0$ and $\infty$ are not poles of the rational map $R_\Lambda$ (see \pref{lem:5_rationalRnewton}), and $N_\Lambda$ has $\tilde{p}$ distinct fixed points (roots of $F$) and finitely many poles in a period strip of $\exp_1$; indeed \vspace{-0.5mm}
\begin{equation}
    \#\exp_1 \big(N_\Lambda^{-1}(\infty)\big) = \#\hspace{0.2mm} R_\Lambda^{-1}(\infty) = \tilde{p}+\deg{Q}+\deg{\widetilde{Q}}. \vspace{-0.25mm}
\end{equation}

For convenience we consider the parameter $\lambda:=\Lambda + \pi i (2m_0+\deg{P})$. Applying our results in \pref{sec:4_Lifting} into this framework, we obtain Baker and wandering domains for families of Newton maps for different values of $\lambda$, which coexist with the attracting invariant basins of $N_\lambda$, as desired (see \pref{cor:C_BakerWandering}).
The simplest cases are Newton's methods with exactly one superattracting fixed point and a simple pole in each period strip of $\exp_1$. It can be seen that those $N_\lambda$ are conjugate to a member of the following family (see \pref{prop:ConjugationNewtonPole}).
\begin{definition}[Pseudotrigonometric family $\mathbf{N}_\lambda$]
    \label{def:classN}
    The \textit{pseudotrigonometric family} $\mathbf{N}_\lambda$ consists of the Newton maps of $F_\lambda(z) = e^{\lambda z} \sin{\pi z}$, $\lambda\in\CC\backslash\{\pm \pi i\}$, which are of the form
    \vspace{-0.5mm}
    \begin{equation}
        N_{\lambda}(z) = z + M_\lambda\big( e^{2\pi i z} \big), \quad \mbox{ where } \quad M_\lambda(w) = -\frac{\hspace{2mm} w-1}{\left(\lambda+\pi i \right)w-\left(\lambda-\pi i \right)}.
    \end{equation}
    \vspace{-5mm}
\end{definition}

The name refers to the fact that $N_0(z)=z-\frac{1}{\pi}\tan{\pi z}$ is a pseudoperiodic analogue of the tangent map (of period $1$). 
The exponential projection $g_\lambda(w)$ of $N_\lambda$ has a unique essential singularity at $B_\lambda:=\frac{\lambda-\pi i}{\lambda+\pi i}$, and only one free critical point at $C_\lambda:=B_\lambda^2$. For each value of $\lambda$, $B_\lambda$ may be placed at $\infty$ via $M_\lambda$ (see \pref{rem:logarithmicSing}), i.e. $g_\lambda$ is conjugate to a transcendental meromorphic map with two finite asymptotic values (as the tangent map), a unique free critical point and a superattracting fixed point at the origin (as the quadratic map), which degenerates to an entire map for $\lambda = \pm \pi i$. Despite its simplicity, this family turns out to exhibit a wide variety of interesting Newton dynamics which can be reflected in a one-dimensional parameter slice.

To this end, we inspect the set $\widetilde{\mathcal{M}}$ of parameters $\lambda$ in which the free critical point of $g_\lambda$ does not converge to $1$ (the non-white region in \pref{fig:BakerWDf_fixed}). This is equivalent to study the values of $\lambda$ for which the pseudotrigometric Newton's method $N_\lambda$ fails to converge to a root of $F_\lambda$ in some open set of initial conditions. This unveils components of $\mathcal{\widetilde{M}}$ in which the free critical point $C_\lambda$ is attracted to a periodic cycle other than $1$, whose immediate basin of attraction may lead to Baker and wandering domains of $N_\lambda$ by the lifting method, which in turn coexist with the infinitely many basins of the roots of $F_\lambda$ (see Examples \ref{ex:BakerWD_SuperCoexistence} and \ref{ex:WD_SuperCoexistence}). Of special interest in the parameter space for $g_\lambda$, and hence for $N_\lambda$ (see also Figures \ref{fig:MandelFrot2} and \ref{fig:MandelPlogD}), are the connected components of $\widetilde{\mathcal{M}}$ leading to wandering domains of different nature for our family of Newton maps (see \pref{rem:SubhypWD}). These and many other questions related to $\mathcal{\widetilde{M}}$ will be addressed in a future paper.

\vspace{-0.1mm}
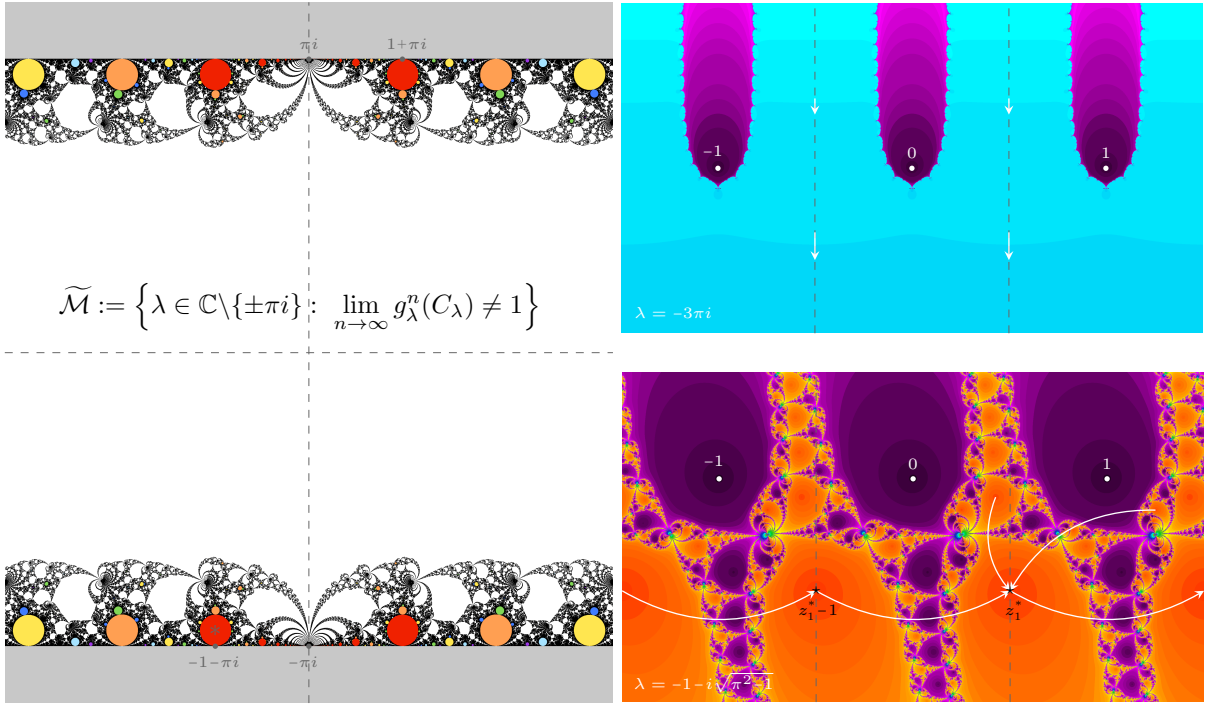
\begin{figure}[H]
    \hspace{2.6mm}\hfill
    \begin{minipage}{0.4865\textwidth}
        \begin{tikzpicture}
    \node[anchor=south west,inner sep=0] (image) at (0,0) {\includegraphics[width=\linewidth]{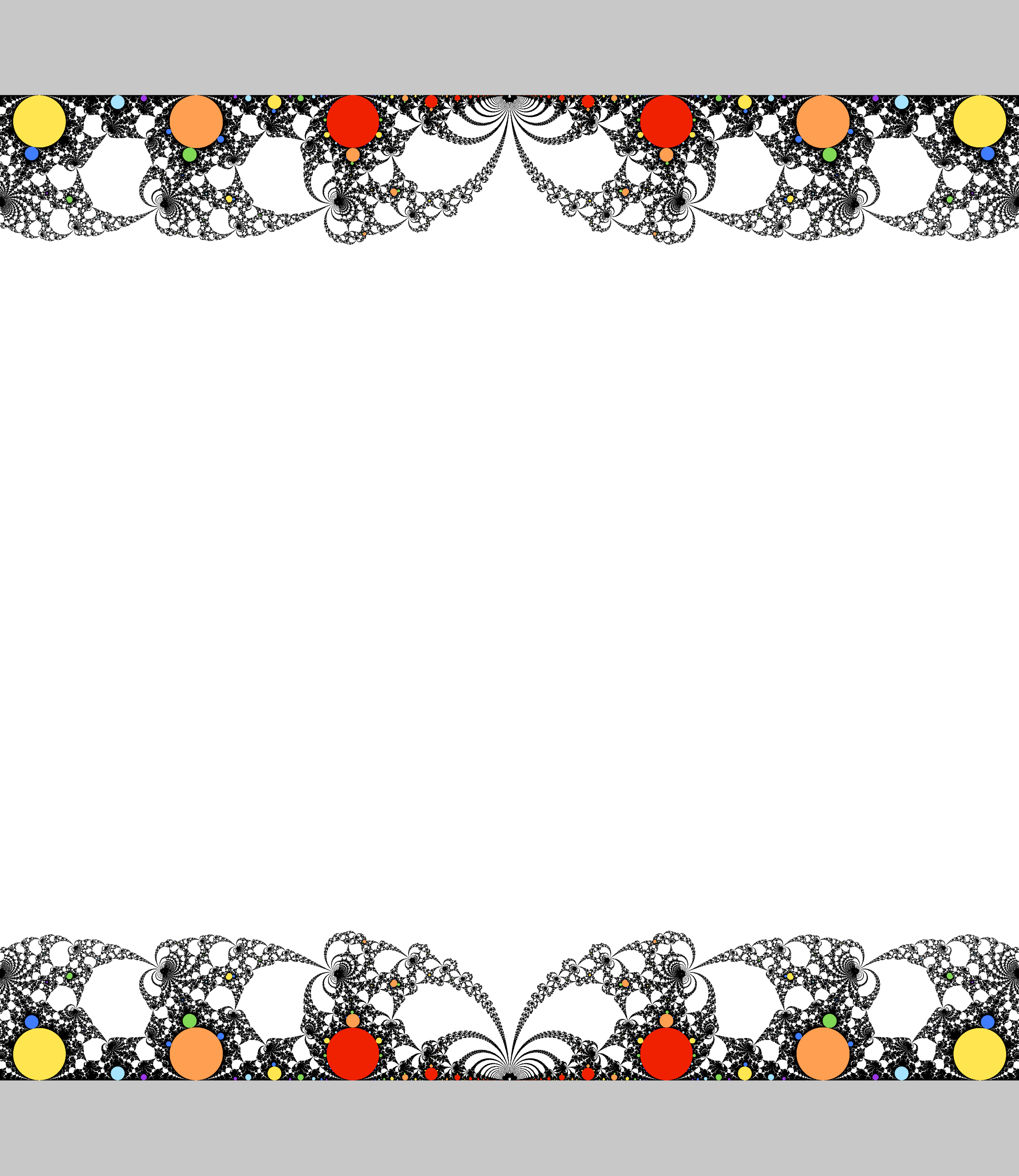}};
        \begin{scope}[x={(image.south east)},y={(image.north west)}]
        \coordinate (a) at ($ ({(0+3.25)/6.5},{(-pi+3.75)/7.5}) $);
        \coordinate (a2) at ($ ({(0+3.25)/6.5},{(pi+3.75)/7.5}) $);
        \coordinate (b) at ($ ({(-1+3.25)/6.5},{(-pi+3.75)/7.5}) $);
        \coordinate (b2) at ($ ({(1+3.25)/6.5},{(pi+3.75)/7.5}) $);
        \coordinate (c) at ($ ({(-1+3.25)/6.5},{(-sqrt(pi*pi-1)+3.75)/7.5}) $);
        
        \draw[line width=0.1pt, color = darkgray!80, dashed] (0.5,0.00) -- (0.5,1);
        \draw[line width=0.1pt, color = darkgray!80, dashed] (0,0.5) -- (1,0.5);
        \fill[darkgray!80] (a) circle (1pt) node[below]{\tiny\hspace{-0.6em}$\shortminus\pi i$};
        \fill[darkgray!80] (a2) circle (1pt) node[above]{\tiny$\pi i$};
        \fill[darkgray!80] (b) circle (1pt) node[below]{\tiny\hspace{-0.16em}$\shortminus 1 \hspace{0.1mm} \shortminus \hspace{0.1mm} \pi i$};
        \fill[darkgray!80] (b2) circle (1pt) node[above]{\tiny\hspace{0.45em}$1\hspace{0.1mm} \shortplus \hspace{0.1mm}\pi i$};
        \draw (c) node[color=darkgray!80] {\footnotesize$\ast$};
        \draw (0.485,0.565) node[draw=none,fill=none] {\small $\widetilde{\mathcal{M}}:= \Big\{ \lambda\in\CC\backslash\{\pm\pi i\}: \ \lim\limits_{n\to\infty} g_\lambda^n(C_\lambda) \neq 1 \Big\}$};
        \end{scope}
        \end{tikzpicture}
    \end{minipage}
    \hfill
    \begin{minipage}{0.49\textwidth}
        \begin{tikzpicture}
        \node[anchor=south west,inner sep=0] (image) at (0,0) {\includegraphics[width=0.95\linewidth]{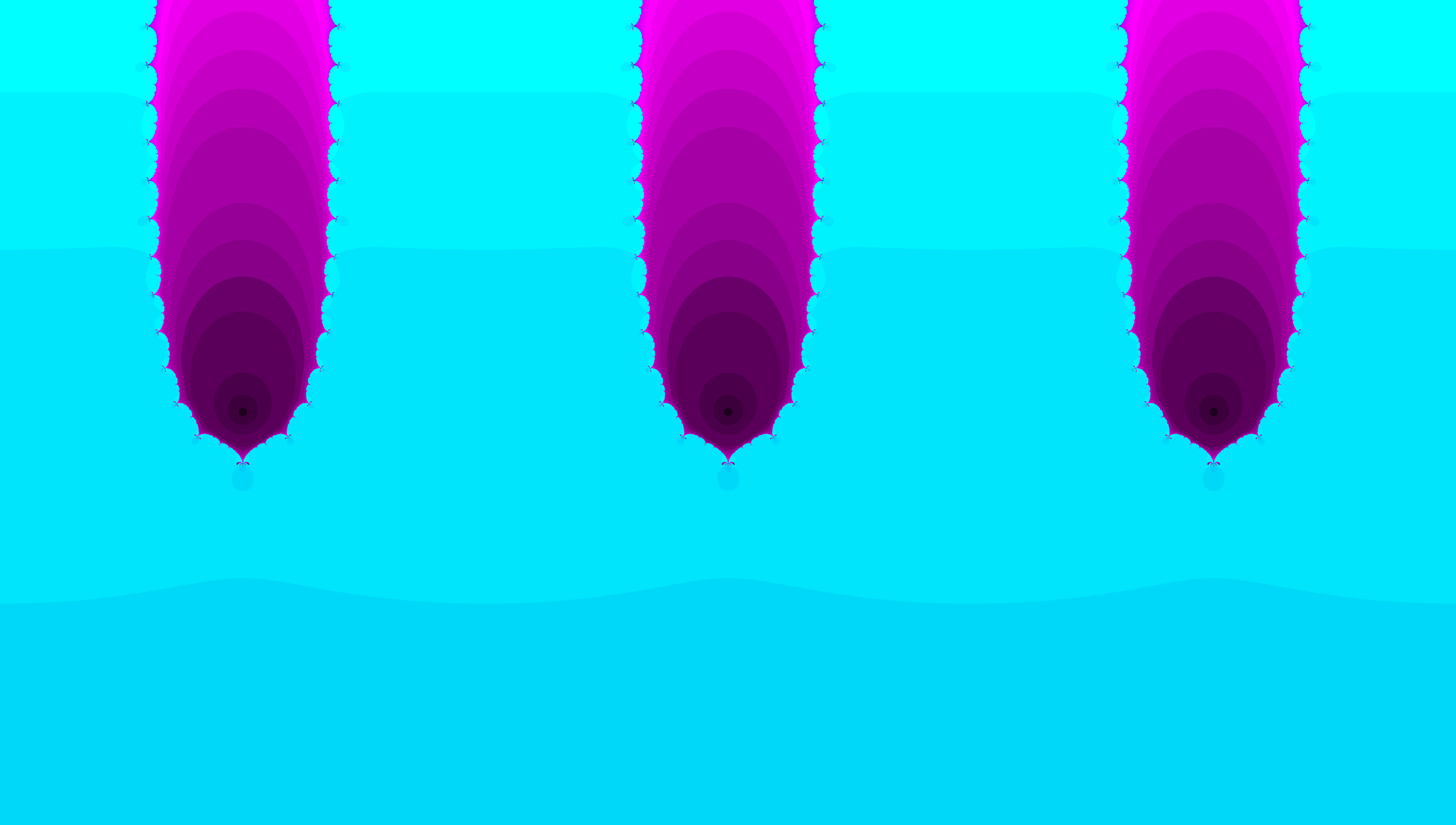}};
        \begin{scope}[x={(image.south east)},y={(image.north west)}]
        \coordinate (a) at ($ ({(0+1.5)/3},{(0+0.85)/1.7}) $);
        \coordinate (a2) at ($ ({(1+1.5)/3},{(0+0.85)/1.7}) $);
        \coordinate (a3) at ($ ({(-1+1.5)/3},{(0+0.85)/1.7}) $);
        \coordinate (b) at ($ ({(0+1.5)/3},{(-0.11+0.85)/1.7}) $);
        \coordinate (b2) at ($ ({(1+1.5)/3},{(-0.11+0.85)/1.7}) $);
        \coordinate (b3) at ($ ({(-1+1.5)/3},{(-0.11+0.85)/1.7}) $);
        \fill[white!80] (a) circle (1pt) node[above]{\tiny$0$};
        \fill[white!80] (a2) circle (1pt) node[above]{\tiny$1$};
        \fill[white!80] (a3) circle (1pt) node[above]{\tiny\hspace{-0.55em}$\shortminus 1$};
        \draw (0.09,0.06) node[draw=none,fill=none, color=white] {\tiny $\lambda=\shortminus 3\pi i$}; 
        \draw[line width=0.1pt, color = darkgray!80, dashed] ($({1/3},0)$) -- ($({1/3},1)$);
        \draw[line width=0.1pt, color = darkgray!80, dashed] ($({2/3},0)$) -- ($({2/3},1)$);
        \draw[line width=0.5pt,color=white,-stealth] ($ ({(-0.5+1.5)/3},{(0.36+0.85)/1.7}) $) to ($ ({(-0.5+1.5)/3},{(0.276484+0.85)/1.7}) $);
        \draw[line width=0.5pt,color=white,-stealth] ($ ({(0.5+1.5)/3},{(0.36+0.85)/1.7}) $) to ($ ({(0.5+1.5)/3},{(0.276484+0.85)/1.7}) $);
        \draw[line width=0.5pt,color=white,-stealth] ($ ({(+0.5+1.5)/3},{(-0.33+0.85)/1.7}) $) to ($ ({(+0.5+1.5)/3},{(-0.473163+0.85)/1.7}) $);
        \draw[line width=0.5pt,color=white,-stealth] ($ ({(-0.5+1.5)/3},{(-0.33+0.85)/1.7}) $) to ($ ({(-0.5+1.5)/3},{(-0.473163+0.85)/1.7}) $);
        \end{scope}
    \end{tikzpicture}
    \\ \\
        \begin{tikzpicture}
        \node[anchor=south west,inner sep=0] (image) at (0,0) {\includegraphics[width=0.95\linewidth]{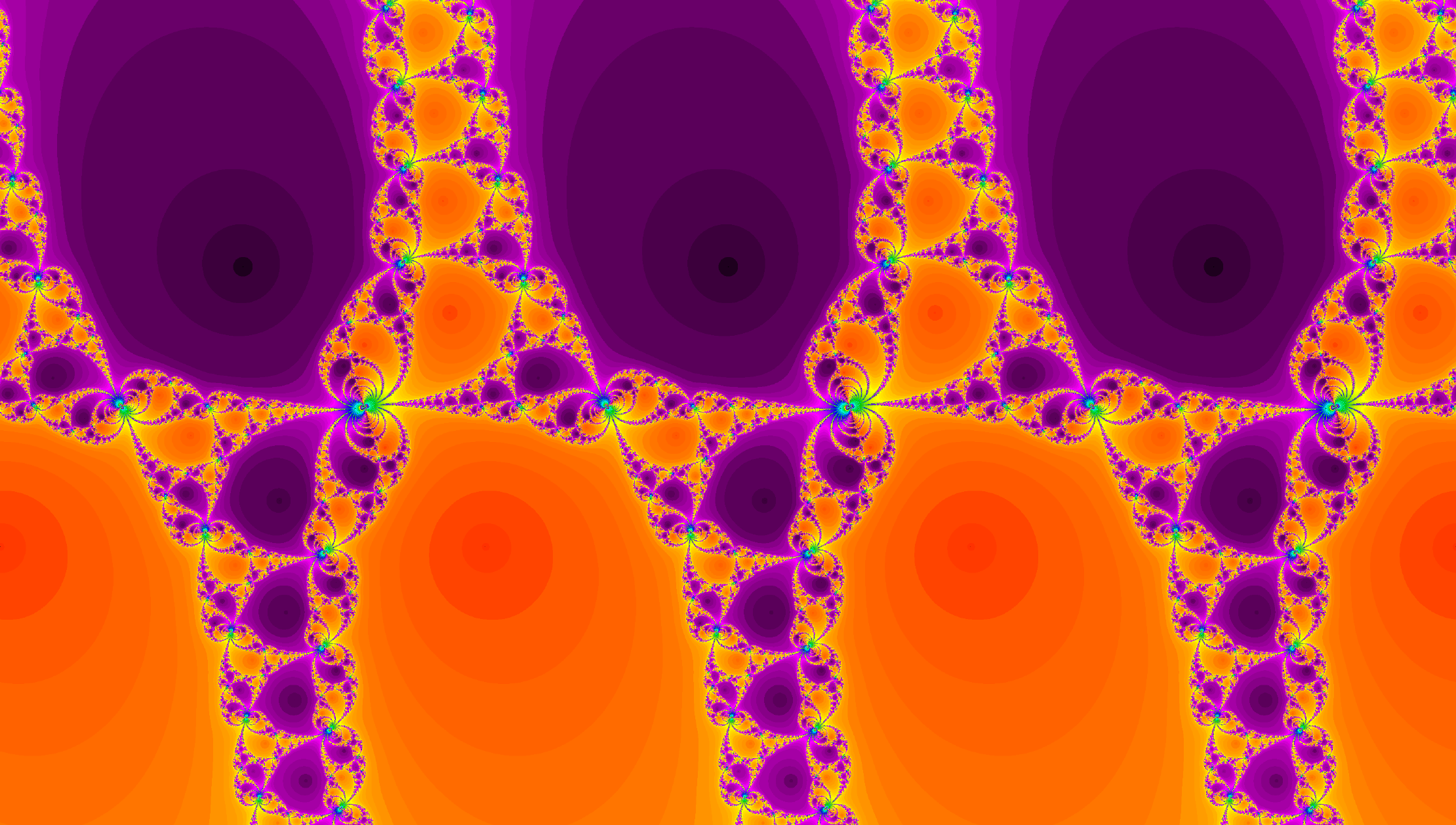}};
        \begin{scope}[x={(image.south east)},y={(image.north west)}]
            \coordinate (a) at ($ ({(0+1.5)/3},{(0+1.15)/1.7}) $);
            \coordinate (a2) at ($ ({(1+1.5)/3},{(0+1.15)/1.7}) $);
            \coordinate (a3) at ($ ({(-1+1.5)/3},{(0+1.15)/1.7}) $);
            \coordinate (b) at ($ ({(-0.5+1.5)/3},{(-0.57667+1.15)/1.7}) $);
            \coordinate (b2) at ($ ({(0.5+1.5)/3},{(-0.57667+1.15)/1.7}) $);
            \draw[line width=0.1pt, color = darkgray!80, dashed] ($({1/3},0)$) -- ($({1/3},1)$);
            \draw[line width=0.1pt, color = darkgray!80, dashed] ($({2/3},0)$) -- ($({2/3},1)$);
            \fill[white!80] (a) circle (1pt) node[above]{\tiny$0$};
            \fill[white!80] (a2) circle (1pt) node[above]{\tiny$1$};
            \fill[white!80] (a3) circle (1pt) node[above]{\tiny\hspace{-0.55em}$\shortminus 1$};
            \draw (0.14,0.06) node[draw=none,fill=none,color=white] {\tiny$\lambda=\shortminus 1 \hspace{0.2mm}\shortminus \hspace{0.2mm}i\sqrt{\pi^2\shortminus1}$};
            \draw[line width=0.5pt,color=white,-stealth] ($ ({(-0.574615+1+1.5)/3},{(-0.094429+1.15)/1.7}) $) to[bend right] (b2);
            \draw[line width=0.5pt,color=white,-stealth] ($ ({(0.24955+1+1.5)/3},{(-0.161072+1.15)/1.7}) $) to[bend right] (b2);
            \draw[line width=0.5pt,color=white,-stealth] ($ ({(-0.5+1.5)/3},{(-0.57667+1.15)/1.7}) $) to[bend right] (b2);
            \draw[line width=0.5pt,color=white,-stealth] ($ ({(-1.5+1.5)/3},{(-0.57667+1.15)/1.7}) $) to[bend right] (b);
            \draw[line width=0.5pt,color=white,-stealth] (b2) to[bend right] ($ ({(1.5+1.5)/3},{(-0.57667+1.15)/1.7}) $);
            \draw (b2) node[color=black,scale=1] {\tiny$\star$} node[color=black,below]{\tiny\hspace{0.41em}$z_{\scalebox{0.7}{1}}^{\raisebox{-0.15\height}{\scalebox{0.7}{*}}}$};
            \draw (b) node[color=black,scale=1] {\tiny$\star$} node[color=black,below]{\tiny\hspace{0.2em}$z_{\scalebox{0.7}{1}}^{\raisebox{-0.15\height}{\scalebox{0.7}{*}}}\hspace{0.1mm}\shortminus\hspace{0.1mm} 1$};
        \end{scope}
        \end{tikzpicture}
    \end{minipage}
    \hfill\vspace{-0.5mm}
    \caption{\textit{Left (parameter space of $g_\lambda$ or the pseudotrigonometric family $N_\lambda$)}: The set $\widetilde{\mathcal{M}}$ for which the free critical point $C_\lambda$ of $g_\lambda$ fails to converge to $1$. The color of each pixel $\lambda$ indicates the period $p$ of the cycle attracting $C_\lambda$ under iteration: red if $p=1$ and $\lim\limits_{n\to\infty} g_\lambda^n(C_\lambda) \notin\{0, \infty\}$ (gray otherwise, i.e. when $|\operatorname{Im}\lambda|>\pi$), orange if $p=2$, yellow if $p=3$, green if $p=4$, light blue if $p=5$, dark blue if $p=6$, purple if $p=7$, and black if higher; see also \pref{fig:MandelPlogD}. Range: $[-3.75,3.75]\times[-3.25,3.25]$.
    \textit{Right-top (dynamical plane of $N_{\lambda}$ for $\lambda=-3\pi i$)}: The superattracting basins of $k\in\ZZ$ (in purple) coexist with a simply-connected Baker domain (in blue); see \pref{ex:BakerWD_SuperCoexistence}. Range: $[-1.5,1.5]\times[-0.85,0.85]$. \textit{Right-bottom (dynamical plane of $N_{\lambda}$ for $\lambda=-1-\sqrt{\pi^2-1}$)}: The superattracting basins of $k\in\ZZ$ (in purple) coexist with a chain of simply-connected wandering domains (in orange) containing a pseudoperiodic point $z_1^*$ of type $(1,1)$; see \pref{ex:WD_SuperCoexistence}. Range: $[-1.5,1.5]\times[-1.15,0.55]$. The brightness of the blue, orange and purple colors indicate the speed of convergence to the fixed points of $g_\lambda$ at $\infty$, $e^{2\pi i z_1^*}$ and $1$, respectively (lighter if it requires more iterates). The dashed lines refer to the coordinate axes in the $\lambda$-plane, and to $\{\operatorname{Re}z=\pm\sfrac{1}{2}\}$ in the $z$-planes.} 
    \label{fig:BakerWDf_fixed}
\end{figure}

\vspace{-2.25mm}

\textbf{\normalsize Outline of the paper.} In \pref{sec:2_Pseudo} we derive the general form of pseudoperiodic maps to prove \pref{thm:A_FormProj} on the class of projectable functions $f$. In \pref{sec:3_ExpProjection} we analyze the projection of poles and singular values of $f$ via $\exp_1$ (see Propositions \ref{prop:PolesToEssential}, \ref{prop:projC} and \ref{prop:projAV}), to identify exponential projections $g$ of finite-type in Bolsch's class (\pref{thm:B_FiniteType}), including those from the class $\mathbf{R}_\ell$ (\pref{cor:classRsv}). The fundamentals of the lifting method in our setting are detailed in \pref{sec:4_Lifting}, starting with the notion of pseudoperiodic points. \pref{thm:logFatouJulia} is the keystone for our purposes, which delivers different types of Baker and wandering domains of $f$ in the non-periodic case (see Examples \ref{ex:MeroStandard} and \ref{ex:LiftParabolicBR}). In \pref{sec:5_Newton} we characterize the Newton maps in class $\mathbf{R}_\ell$ with fixed points (\pref{thm:C_AtlasBWD}), whose attracting basins, under the conditions of \pref{cor:C_BakerWandering}, coexist with Baker or wandering domains. We explore the one-parameter family $\mathbf{N}_\lambda$ of Newton maps, as the simplest one in this class, to unveil the atlas of wandering domains in \pref{fig:BakerWDf_fixed}, and conclude with some observations on the components of $\widetilde{\mathcal{M}}$.
%
\section{Pseudoperiodic maps and Proof of Thm. \ref{thm:A_FormProj}}
\label{sec:2_Pseudo}

The semiconjugacy relation $g\circ\exp_1=\exp_1\circ f$ determines the class of projectable functions (see \pref{def:A_proj}), as well as its iterates, by means of $g^n\circ\exp_1=\exp_1\circ f^n$, $n\in\NN$, whenever defined. The following lemma, which is fundamental for our discussion, can be easily proved by induction.

\begin{lemma}[Pseudoperiodicity]
\label{lem:ProjCond}
    Let $f$ be a projectable function via $\exp_1$. Then
    \begin{equation}
        \label{eq:projRelation}\vspace{-0.5mm}
        f(z+1) = f(z) + \ell
    \end{equation}
    for some $\ell\in\ZZ$, and every $z\in\CC$. Moreover, for any $n\in\NN$ and $k\in\ZZ$,
    \begin{equation}
        \label{eq:projCondition}\vspace{-0.5mm}
        f^n(z+k) = f^n(z) + \ell^n k. \vspace{-0.5mm}
    \end{equation}
\end{lemma}

Hence, projectable functions correspond to meromorphic functions satisfying relation \pref{eq:projRelation}, which are occasionally called \textit{$\ell$-pseudoperiodic}, or \textit{periodic modulo an integer} $\ell$ (of period $1$), in the sense that $f(z+1)-f(z)=\ell$. In order to further characterize projectable functions, it is convenient to identify them as a special case of the more general class of \textit{pseudoperiodic functions}.

\begin{definition}[Pseudoperiodic functions] \label{def:Pseudoperiodic}
    Let $f:\CC\to\hCC$ be a meromorphic function, and $\Lambda$ a lattice in $\CC$. We say that $f$ is \textit{pseudoperiodic} with respect to $\Lambda$ if, for each $\tau\in \Lambda$, there exists a constant $\eta_\tau\in\CC$ (a \textit{pseudoperiod} of $f$) such that \vspace{-0.75mm} 
    \begin{equation} \label{eq:PseudoPeriodicity}
        f(z+\tau) = f(z) + \eta_\tau, \vspace{-0.75mm} 
    \end{equation}
    for all $z\in\CC$. It is said to be \textit{simply} or \textit{doubly} \textit{pseudoperiodic} if all its pseudoperiods are of the form $m\eta_{\footnotesize \tau_1}$ or $m\eta_{\footnotesize\tau_1}+n\eta_{\footnotesize\tau_2}$, respectively, for some $\tau_1,\tau_2\in\CC^*$ with $\operatorname{Im}\left[\sfrac{\tau_2}{\tau_1}\right]>0$, where $m,n\in\ZZ$.
\end{definition}

\begin{example}[Standard map and Weierstrass $\zeta$-function]
    \label{ex:standardWeierstrass}
    The extension of the Arnol'd \textit{standard family} of circle maps \cite{Arnold2009} to the complex plane, given by the entire function
    \begin{equation} \label{eq:ArnoldFamily}\vspace{-0.5mm}
    f_{\scriptsize{\alpha,\beta}}(z)=z+\alpha-\frac{\beta}{2\pi} \sin{2\pi z}
    \end{equation}
    with parameters $\alpha\in\RR$ and $\beta>0$, is an example of a simply pseudoperiodic function such that $\eta_{\footnotesize\tau_1}=\tau_1=1$, and thus projectable via $\exp_1$. The archetype of a doubly pseudoperiodic function is the so-called \textit{Weierstrass $\zeta$-function} with respect to $\Lambda:=\tau_1\ZZ+\tau_2\ZZ$, $\sfrac{\tau_2}{\tau_1}\in\HH^+$, given by
\begin{equation} \label{eq:WeierstrassZeta}\vspace{-0.5mm}
    \zeta(z) = \frac{1}{z} + \sum\limits_{\tau\in\Lambda^*} \left( \frac{1}{z-\tau} + \frac{1}{\tau} + \frac{z}{\tau^2} \right),
\end{equation}
where the sum runs over all non-zero lattice points. Its pseudoperiods, $\eta_{\tau_1}=2\zeta\left(\sfrac{\tau_1}{2}\right)$ and $\eta_{\tau_2}=2\zeta\left(\sfrac{\tau_2}{2}\right)$, satisfy the \textit{Legendre relation}: $\eta_{\tau_1}\tau_2-\eta_{\tau_2}\tau_1 = 2\pi i$. It may be written as $\zeta(z)=\frac{\eta_{\tau_1}}{\tau_1} z + \varphi_1(z)$, where $\varphi_1$ is a $\tau_1$-periodic map such that $\varphi_1(z+\tau_2)=\varphi(z)-\frac{2\pi i}{\tau_1}$ (see explicit form in \cite[\S 18]{Lang1987}).
\end{example}

In analogy to the classical theory of periodic functions, and following the description of doubly pseudoperiodic functions by Brady \cite{Brady1970}, we show that any pseudoperiodic meromorphic function on $\CC$ is either simply or doubly pseudoperiodic, and they can be written in a unique manner.

\begin{proposition}[Form of pseudoperiodic functions] \label{prop:PseudoForm}
    Let $f$ be a pseudoperiodic function. Then $f$ can be uniquely expressed either as the simply pseudoperiodic function with respect to $\tau_1 \ZZ$, $\tau_1\in\CC^*$,
    \begin{equation}
        f(z) = a z+\varphi(z)
    \end{equation}
    with pseudoperiod $\eta_{\tau_1}=a\tau_1$, where $a\in\CC$ and $\varphi$ is a simply periodic function of period $\tau_1$; or as the doubly pseudoperiodic function with respect to $\Lambda:=\tau_1\ZZ+\tau_2\ZZ$, $\sfrac{\tau_2}{\tau_1}\in\HH^+$,
    \begin{equation} 
        \label{eq:DoublyPseudo}
        f(z) = az+b\zeta(z)+E(z)
    \end{equation}
    with pseudoperiods $\eta_{\tau_1}=a\tau_1+2b\zeta(\sfrac{\tau_1}{2})$ and $\eta_{\tau_2}=a\tau_2+2b\zeta(\sfrac{\tau_2}{2})$, where $a,b\in\CC$, $\zeta$ is the Weierstrass $\zeta$-function with respect to $\Lambda$, and $E$ is a doubly periodic function of periods $\tau_1$ and $\tau_2$.
\end{proposition}

\begin{proof}
    We distinguish cases in terms of the number of independent pseudoperiods of $f$. First, suppose that $f$ is simply pseudoperiodic with respect to $\tau_1\ZZ$, where $\tau_1$ is taken as the non-zero complex number of smallest modulus that satisfies the pseudoperiodicity condition \pref{eq:PseudoPeriodicity}. Let $\eta_{\tau_1}$ be the associated pseudoperiod, and consider the function $\varphi(z):=f(z)-az$, where $a=\frac{\eta_{\tau_1}}{\tau_1}$. Then,
    \begin{equation*}
        \varphi(z+\tau_1) = f(z+\tau_1)-az-\eta_{\tau_1} = \varphi(z)
    \end{equation*} 
    for all $z\in\CC$, so that $\varphi$ is simply periodic (with period $\tau_1$). To verify uniqueness, we may assume that $f(z) = \Tilde{a} z +\widetilde{\varphi}(z)$ for some $\Tilde{a}\in\CC$ and a non-constant $\tau_1$-periodic function $\widetilde{\varphi}$. It follows that $(a-\Tilde{a})\tau = 0$, and so $\widetilde{\varphi}=\varphi$, that is, the representation of $f$ in such form is unique.
    
    \noindent Now suppose that $f$ is doubly pseudoperiodic with respect to $\Lambda$, and choose $\tau_1$ and $\tau_2$ as the two non-zero complex numbers of smallest modulus, with $\operatorname{Im}[\sfrac{\tau_2}{\tau_1}]>0$, such that $\Lambda=\tau_1\ZZ + \tau_2\ZZ$. Denote by $\eta_{\tau_1}$ and $\eta_{\tau_2}$ the corresponding pseudoperiods satisfying \pref{eq:PseudoPeriodicity}. This case follows directly from \cite[Thm.~4.1.4]{Brady1970}, which leads to expression \pref{eq:DoublyPseudo} in the same fashion, by using Legendre's relation between the pseudoperiods of the Weierstrass $\zeta$-function and the periods of the elliptic function $E$.
    
    \noindent Finally, observe that if, in the latter case, $f$ had an additional pseudoperiod $\eta_{\tau_3}$ associated to some $\tau_3\in\CC^*\backslash\Lambda$, then, for some $m,n\in\ZZ$, the point $\tau_3-m\tau_1-n\tau_2$  would be a non-zero complex number of smaller modulus than $\tau_1$ or $\tau_2$ satisfying relation \pref{eq:PseudoPeriodicity}, contrary to our construction. Hence, such a point must lie in some vertex of $\Lambda$, and $\eta_{\tau_3}=m\eta_{\tau_1}+n\eta_{\tau_2}$, given that $\eta_{\tau}$ is $\ZZ$-linear in $\tau$. 
    $\hfill\square$
\end{proof}

\begin{remark}[Representation by exponentials]
    \label{rem:Fourier}
    The nonlinear term of any projectable function is a periodic function, which may be expanded as the quotient of two convergent Fourier series. Indeed, for any $1$-periodic meromorphic map $\varphi(z)$, defined in a region $\Omega$ invariant under translation by $\pm 1$, there exists a unique function $\Phi(w)$ which is meromorphic in $\exp_1\Omega := \{ e^{2\pi i z}: z\in\Omega \}$ such that
    \begin{equation}
        \varphi(z) = \Phi(e^{2\pi i z}).
    \end{equation}
    The poles of $\Phi(w)$ coincide with the image under $\exp_1$ of the poles of $\varphi(z)$, and are of the same multiplicity (see e.g. \cite[Thm.~4.7]{Markushevich1965}). The \textit{Laurent development} of $\Phi$ in an annulus $\{w:r<|w|<R\}$ in which $\Phi$ has no poles, where $0\leq r<R \leq \infty$, delivers the \textit{Fourier coefficients} of $\varphi$ in the horizontal strip $\{ z:\frac{-1}{\hspace{1mm}2\pi}\ln{R}<\operatorname{Im}z < \frac{-1}{\hspace{1mm}2\pi}\ln{r}\}$ via $w=e^{2\pi i z}$. Notice that $\Phi(w)=\varphi\left(\frac{1}{2\pi i} \log{w}\right)$ is well-defined since, although the logarithm has infinitely many values, they differ by integers.
\end{remark}

\begin{customproof}{of \pref{thm:A_FormProj}}
    \textup{Due to \pref{lem:ProjCond} projectable functions $f$ via $\exp_1$ naturally arise as the particular class of pseudoperiodic functions, whether simply or doubly pseudoperiodic, with an integer pseudoperiod $\ell$. In any case, by direct application of \pref{prop:PseudoForm} (see also \pref{ex:standardWeierstrass} on Weierstrass $\zeta$-function), $f$ is the sum of a linear map and a periodic function (with period $1$). Thus, \pref{rem:Fourier} asserts that $f(z)=\ell z + \Phi(e^{2\pi i z})$ for a meromorphic function $\Phi$ in $\CC^*$, which is non-constant as $f$ is transcendental. 
    Moreover, in the case that $f$ is doubly pseudoperiodic (with respect to $\ZZ+\tau\ZZ$), using the same notation as in \pref{eq:DoublyPseudo}, and given that $\eta_1=\ell$, we obtain
    \begin{equation}
        \ell = a+2b\zeta(\sfrac{1}{2}), \quad \eta_\tau = a\tau+2b\zeta(\sfrac{\tau}{2}).
    \end{equation}
    Therefore, $b=\frac{\ell\tau-\eta_\tau}{2\pi i}$ via the Legendre relation $\zeta(\sfrac{1}{2})\tau-\zeta(\sfrac{\tau}{2})=\pi i$. Taking into account that $\zeta(z)-2\zeta(\sfrac{1}{2})z$ is a $1$-periodic function (see \pref{ex:standardWeierstrass}), we obtain the expression \pref{eq:ProjUnique}. In this situation, $f$ would be also projectable via $e^{\sfrac{2\pi i z}{\tau}}$ if and only if $\eta_\tau=L\tau$ for some integer $L$.
    $\hfill\square$}
\end{customproof}
%
\section{The exponential projection and Proof of Thm. \ref{thm:B_FiniteType}}
\label{sec:3_ExpProjection}

\subsection{Essential singularities}

A projectable function $f$ has a unique essential singularity at $\infty$ due to its non-constant periodic part. The first step of our analysis is to determine the set of essential singularities of its exponential projection $g$, which we denoted by $\mathcal{E}(g)$.

\begin{proposition}[Projection of poles]
    \label{prop:PolesToEssential}
    Let $f$ be a projectable function via $\exp_1$, and $g$ its exponential projection.
    Then
    \begin{equation}
        \mathcal{E}(g)\cap \CC^* = \exp_1\left(f^{-1}(\infty)\right).
    \end{equation}
    If $f$ is non-entire, then $f$ has infinitely many poles, and $\mathcal{E}(g)\cap \CC^*\neq\emptyset$.
\end{proposition}

\begin{proof}
    Consider an arbitrary pole $b\in\CC$ of the projectable function $f$, and a neighborhood $U_b$ of $b$ such that $U_b\cap f^{-1}(\infty) = \{ b\}$. Then, there exists some $R>0$ such that $\{z: |z|>R \}\subset f(U_b)$. Take a pair of curves $\gamma^{\pm}$ in $U_b$ towards the pole of $f$ such that $\operatorname{Im}f(z)\to\pm\infty$ as $z\to b$ along $\gamma^\pm$.
    
    \noindent On the one hand, $\exp_1(U_b)$ is a neighborhood of $B:=e^{2\pi i b}\in\CC^*$, so that the directions $\gamma^\pm$ are mapped conformally by $\exp_1$ to paths $\Gamma^\pm:=\exp_1(\gamma^\pm)$ tending to $B$. On the other hand, given that 
    \begin{equation*}
        g\circ\exp_1 (\gamma^\pm) = \exp_1 \hspace{-0.3mm}\circ \hspace{0.3mm} f(\gamma^\pm)
    \end{equation*}
    due to the semiconjugacy, and $|e^{2\pi i z}|=e^{-2\pi \operatorname{Im}z}$,
    we obtain that $g(\Gamma^\pm)$ are paths converging to $0$ and $\infty$, respectively.
    Thus, $\lim\limits_{w\to B}g(w)$ does not exist, i.e. $B$ is an essential singularity of $g$.
    \vspace{-1mm}
    
    \noindent For all other points $z\in\CC\backslash f^{-1}(\infty)$, $f$ is holomorphic and bounded in a neighborhood of $z$, hence $g(e^{2\pi i z})$ is well-defined. In the non-entire case, we have that $f^{-1}(\infty)\subset\CC$ is an infinite set by pseudoperiodicity, and so $f$ has at least one pole in every period strip of $\exp_1$ (of width $1$). Hence, there exists at least one essential singularity for $g$.
    $\hfill\square$
\end{proof}

In view of this, and considering that a meromorphic function can have at most countably many poles, the exponential projection $g$ belongs to the so-called \textit{Bolsch's class} \cite{Bolsch1997}.

\begin{definition}[Bolsch's class and essential poles]
    \label{def:BolschClass}
    Denote by $\mathbf{K}$ the \textit{Bolsch's class} formed by those functions $g$ for which there is a closed countable set $\mathcal{E}(g)\subset\hCC$ such that $g$ is non-constant and meromorphic in $\hCC\backslash \mathcal{E}(g)$, but in no larger set. We say that $\mathcal{D}(g):=\hCC\backslash \mathcal{E}(g)$ is the \textit{domain of definition} of $g$, and $B\in\mathcal{D}(g)$ is a \textit{(essential) prepole} of $g$ of order $m\geq 1$ if $B\in g^{-m}\left( \mathcal{E}(g) \right)$.
\end{definition} \vspace{-0.5mm}

This is the smallest class which includes transcendental meromorphic functions and is closed under composition. Note that in general iterates of a meromorphic (non-entire) map $f$ are no longer meromorphic in $\CC$, since every pole of $f$ becomes an essential singularity of $f^2$. In the context of Bolsch's class, as stated in \cite[Lem.~2]{Baker2001}, if $g_1$ and $g_2$ belong to $\mathbf{K}$, then $g_2\circ g_1\in\mathbf{K}$ with
\begin{equation}
    \label{eq:essSingComp}
    \mathcal{E}(g_2\circ g_1) = \mathcal{E}(g_1) \medcup g_1^{-1}\left( \mathcal{E}(g_2)\right),
\end{equation}
and by \cite[Lem.~4]{Baker2001} its set of singular values satisfies
\begin{equation}
    \label{eq:singValueComp}
    \mathcal{S}(g_2\circ g_1) \subset \mathcal{S}(g_2) \medcup g_2\left( \mathcal{S}(g_1) \hspace{0.1mm}\backslash \mathcal{E}(g_2)\hspace{0.1mm} \right).
\end{equation} 

\vspace{-2.0mm}
In transcendental dynamics, (essential) poles and prepoles are dynamically relevant since their forward orbits get eventually truncated, in contrast to poles of rational maps, for which $\infty$ is a common point on $\hCC$. Furthermore, for a general map $g\in\mathbf{K}$, $\mathcal{E}(g)$ is the closure of the set of isolated essential singularities of $g$, and $\mathcal{J}(g)$ is the closure of the set of all its essential prepoles if $g$ has at least one pole which is not an omitted value. Hence, \textit{Picard's great theorem} applies, that is, for any neighborhood $U$ of $\hat{w}\in\mathcal{E}(g)$, the function $g$ assumes in $U\backslash\mathcal{E}(g)$ every value of $\hCC$ infinitely often, with at most two exceptions, often called \textit{Picard exceptional values} (see more details in \cite{Bolsch1997}).

\begin{remark}[Doubly pseudoperiodic case]
    \label{rem:doublyPcase}
    Any projectable function via $\exp_1$ which is doubly pseudoperiodic with respect to $\ZZ+\tau\ZZ$, $\tau\in\HH^+$ (see \pref{thm:A_FormProj}), is non-entire with poles at $b+m+n\tau$ for all $m,n\in\ZZ$, given $b\in f^{-1}(\infty)$. Thus, $f$ has infinitely many poles in the period strip $S_0=\{z:-\frac{1}{2}< \operatorname{Re}z \leq \frac{1}{2}\}$ of $\exp_1$, accumulating onto both ends of $S_0$. \pref{prop:PolesToEssential} asserts that both $0$ and $\infty$ must always be essential singularities of its exponential projection $g$, and $\#\mathcal{E}(g)=\infty$.
\end{remark}

Note that poles of a simply periodic function may also accumulate at either end of the strip $S_0$.

\begin{example}[Accumulation of poles and zeros at the ends of the strip]
    Consider the $1$-periodic function
    \begin{equation}
        f(z)=\Phi(e^{2\pi i z}), \quad \mbox{ where } \quad \Phi(w) = \frac{e^{2\pi i w}-1}{e^{{2\pi i}/{w}}-1}.
    \end{equation}
    Then, $f$ has zeros at $a_{k,m}:=\frac{k}{2}-i\frac{\log m}{2\pi}$, poles at $b_{k,m}:=\frac{k}{2}+i\frac{\log m}{2\pi}$, and removable singularities at $\frac{k}{2}$, where $k,m\in\ZZ$ with $m\geq 2$ (see \cite[\S X.2.6]{Nevanlinna1970}). Note $\{a_{k,m}, b_{k,m}\}\subset S_0$ if and only if $k\in\{0,1\}$, and $0\in\mathcal{E}(g)$ since it is the limit point of $e^{2\pi i b_{0,m}}\in\mathcal{E}(g)$ as $m\to\infty$. The zeros of $f$ in $S_0$ accumulate at the lower end of the strip, hence $\infty$ is a limit point of $g^{-1}(1)$, and so $\infty\in\mathcal{E}(g)$ as well.
\end{example}

It remains to check under which conditions the points at $0$ and/or $\infty$, as omitted values of $\exp_1$, are in the domain of definition of the projection $g$, called $\mathcal{D}(g)$. Here the class $\mathbf{R}_\ell$ (see \pref{def:classR}) of non-entire projectable functions $f$ emerges in a natural way. Recall that $\exp_1(z):=e^{2\pi i z}$ induces an isomorphism from $\CC /\ZZ$ onto $\CC^*$, sending the upper (resp. lower) end to $0$ (resp. $\infty$).
\begin{proposition}[Projection of cylinder ends]
    \label{prop:Regular0inf}
    Let $f$ be a projectable function via $\exp_1$, written as $f(z) = \ell z + \Phi(e^{2\pi i z})$ for some $\ell\in\ZZ$ and $\Phi$ meromorphic in $\CC^*$, and $g$ its exponential projection. Then $0\in\mathcal{D}(g)$ (resp. $\infty\in\mathcal{D}(g)$) if and only if $\Phi(0)$ (resp. $\Phi(\infty)$) is defined and not equal to $\infty$. Furthermore,\vspace{-0.5mm}
    $$\{0,\infty\}\subset\mathcal{D}(g)\quad \iff \quad f\in\mathbf{R}_\ell.$$
\end{proposition}
\vspace{-1.5mm}
\begin{proof}
    Due to the semiconjugacy, the exponential projection $g$ is given by
    \begin{equation*}
        g(w)=w^\ell e^{2\pi i \Phi(w)},
    \end{equation*}
    whenever defined. As a map in Bolsch's class, we deduce from the relation \pref{eq:essSingComp} that $0$ (resp. $\infty$) is in the domain of definition of $g$ if and only if $0$ (resp. $\infty$) does not belong to $\mathcal{E}(\Phi)\cup\Phi^{-1}(\infty)$, that is, $0$ (resp. $\infty$) is neither an essential singularity of $\Phi$ nor a preimage of $\mathcal{E}(\exp_1)=\infty$ under $\Phi$.
    
    \noindent It follows that $\{0,\infty\}\subset\mathcal{D}(g)$ if and only if $\Phi$ is defined and not equal to infinity at both $0$ and $\infty$. As $\Phi$ is a non-constant meromorphic function in $\CC^*$, we conclude that here $\Phi$ must be a rational map such that $\{0,\infty\}\cap\Phi^{-1}(\infty)=\emptyset$, i.e. $f\in\mathbf{R}_\ell$.
   $\hfill \square$
\end{proof}

\begin{remark}[Entire case]
    \label{rem:EntireCase}
    If $f$ is a projectable entire function, then its exponential projection $g$ is in general an analytic self-map of $\CC^*$. To be precise, we have that $f(z)=\ell z+\Phi(e^{2\pi i z})$, where $\ell\in\ZZ$ and $\Phi\circ\exp_1$ must be a simply $1$-periodic entire function. Given that $f$ has no poles, and
    \begin{equation}
        \Phi^{-1}(\infty)\cap\CC^* = \exp_1\left(f^{-1}(\infty)\right)
    \end{equation} %
    by \pref{rem:Fourier}, we obtain $\Phi^{-1}(\infty)\cap\CC^*=\emptyset$. Due to Propositions \ref{prop:PolesToEssential} and \ref{prop:Regular0inf}, $\mathcal{E}(g)\cap\CC^*=\emptyset$, and both $0$ and $\infty$ are in $\mathcal{D}(g)$ if and only if $f\in\mathbf{R}_\ell$ (non-entire). Thus, $\mathcal{E}(g)\neq \emptyset$, and we have the following cases:
    \begin{enumerate}[label = (\roman*)] 
        \vspace{-0.75mm}
        \item $\#\mathcal{E}(g)=1$, say $\mathcal{E}(g)=\{\infty\}$ (up to $w\mapsto \sfrac{1}{w}$). Note that $\Phi$ is either a transcendental entire function (if $\mathcal{E}(\Phi)=\{\infty\}$), or a polynomial (if $\Phi(\infty)=\infty$). If $\ell<0$, then $g$ is a transcendental meromorphic map with only one (omitted) pole at $0$, while otherwise $g$ is a transcendental entire function.\vspace{-0.0mm}
        \item $\mathcal{E}(g)=\{0,\infty\}$. Then $g$ is a \textit{transcendental self-map} of $\CC^*$, $\Phi(w)=H(w)+\widetilde{H}(\sfrac{1}{w})$ for some non-constant entire functions $H$ and $\widetilde{H}$, and $\ell$ is equal to the winding number of $g(\Gamma)$ with respect to $0$, for any simple closed curve $\Gamma\subset\CC^*$ (oriented counterclockwise) around the origin; see more details in \cite{Marti-Pete2016,Radstrom1953}.
    \end{enumerate}
\end{remark}

\subsection{Singular values}

Our goal in this section is to relate the singular values of a projectable function $f$ with those of its projection $g$ via $\exp_1$, and in particular, to characterize the cases in which $g$ is of finite-type, i.e. $\# \mathcal{S}(g)<\infty$. Recall that $\mathcal{S}(g)$ refers to the set of singular values of $g$ defined in \pref{eq:Intro_SingularValues}, and denote by $\mathcal{C}(g)$ the set of critical points. First, we investigate the correspondence between the critical points of $f$ and $g$. By the semiconjugacy,
\begin{equation}
     \label{eq:derivativeGandF}
    g'(e^{2\pi i z}) = e^{2\pi i \left(f(z)-z\right)} f'(z)
\end{equation}
for any $z\in\CC\backslash f^{-1}(\infty)$. Note that in the general form $f(z)=\ell z + \Phi(e^{2\pi i z})$, where $\ell\in\ZZ$ and $\Phi\circ\exp_1$ is periodic (see \pref{thm:A_FormProj}),
\begin{equation}
    \label{eq:derivativeSP}
    g'(w) = w^{\ell-1} e^{2\pi i\Phi(w)}\left( \ell + 2\pi i w \hspace{0.1mm}\Phi'(w) \right).
\end{equation}

\begin{proposition}[Projection of critical points]
    \label{prop:projC}
    Let $f$ be a projectable function via $\exp_1$, and $g$ its exponential projection, written as $g(w)=w^\ell e^{2\pi i\Phi(w)}$ with $\ell\in\ZZ$ and $\Phi$ meromorphic in $\CC^*$. Then
    \begin{equation}
    \label{eq:projCstar}
        \mathcal{C}(g)\cap \CC^* = \exp_1\left( \mathcal{C}(f)\hspace{0.1mm}\backslash \hspace{0.1mm}f^{-1}(\infty) \right).
    \end{equation}
    If $0$ is not an omitted value of $f'$, then $\#\mathcal{C}(f)=\infty$, and $\mathcal{C}(g)\cap \CC^*\neq\emptyset$.
    Moreover, $0\in\mathcal{C}(g)$ (resp. $\infty\in\mathcal{C}(g)$) if and only if $g(0)$ (resp. $g(\infty)$) is defined, and either $|\ell|\geq 2$, or $\ell=0$ with $\Phi'(0)=0$ (resp. $\Phi'(\infty)=0$).
\end{proposition}

\begin{proof}
    It follows from \pref{eq:derivativeGandF} that the critical points of $g$ in $\CC^*$ are the image under $\exp_1$ of the critical points of $f$, excluding the possible multiple poles of $f$ (since they project via $\exp_1$ to points outside the domain of definition $\mathcal{D}(g)$ of $g$; see \pref{prop:PolesToEssential}). If $0$ is not an omitted value of $f'$, i.e. $f'(c)=0$ for some $c\in\CC\backslash f^{-1}(\infty)$, it is clear that $\#\mathcal{C}(f)=\infty$, as the derivative $f'$ is $1$-periodic.
    
    \noindent Given $g$ as in the statement with $0\in\mathcal{D}(g)$, i.e. $0\notin\mathcal{E}(\Phi)\cup\Phi^{-1}(\infty)$ by \pref{prop:Regular0inf}, we have the following:
    \begin{enumerate}[label = (\roman*)] 
        \vspace{-0.5mm}
        \item If $\ell\geq2$, then $0$ is a fixed critical point of $g$, as computed from the expression \pref{eq:derivativeSP}. 
        \vspace{-0.25mm}
        \item If $\ell\leq-2$, then $\{0,\infty\}$ is a critical cycle of period $2$ if $\infty\in\mathcal{D}(g)$; otherwise $0$ is a multiple (essential) pole of $g$, since $0$ is a multiple zero of $\sfrac{1}{g(w)}$ as seen from \pref{eq:derivativeSP}.
        \vspace{-0.25mm}
        \item If $\ell=0$ (periodic case), then $g'(0) = e^{2\pi i\Phi(0)} 2\pi i \Phi'(0)$ vanishes if and only if $\Phi'(0)=0$.
        \vspace{-0.5mm}
    \end{enumerate}
    \noindent Thus, if $|\ell|\geq 2$, or $\ell=0$ with $\Phi'(0)=0$, we have $0\in\mathcal{C}(g)$. The reverse implication follows from \pref{eq:derivativeSP} by considering the value of $g'$ in the remaining cases: $g'(0)=e^{2\pi i\Phi(0)}\in\CC^*$ if $\ell=1$, and $0$ is a simple preimage of $\infty$ if $\ell=-1$, that is, $0\notin \mathcal{C}(g)$. The same arguments apply to the point at $\infty$. $\hfill\square$
\end{proof}

Notice that, even if $f$ has infinitely many critical points in the fundamental period strip $S_0$ of $\exp_1$ (and so in every vertical strip of width $1$), its exponential projection $g$ may be of finite-type.

\begin{example}[Infinitely many critical points in every period strip]
    \label{ex:InfinitelyCinS1}
    Consider the $1$-periodic entire function
    \begin{equation}
        f(z)=\Phi(e^{2\pi i z}), \quad \mbox{ where } \quad \Phi(w) = w - \frac{1}{2\pi}\sin{\left(2\pi w \right)}.
    \end{equation}
    The critical points of $f$ are $c_{k,m}:=\frac{k}{2}-i\frac{\log m}{2\pi}$ with $f(c_{k,m})=e^{\pi i k}m$, where $k,m\in\ZZ$, and $m\geq 1$. Observe that $\#\left(\mathcal{C}(f)\cap S_0\right)=\infty$, and its set of critical values, $\mathcal{CV}(f)=\ZZ\backslash\{0\}$, is also infinite. However, $\mathcal{CV}(g)=\{1\}$ for $g(w)=e^{2\pi i \Phi(w)}$. Here $\infty\in\mathcal{E}(g)$, but $0\in\mathcal{C}(g)$ due to \pref{prop:projC}.
\end{example}

The previous proposition leads directly to the analogous relation between the critical values of $f$ and $g$.
\begin{remark}[Projection of critical values]
    \label{rem:projCV}
    Using the notation in \pref{prop:projC}, since $g\hspace{0.1mm}\circ\hspace{0.1mm}\exp_1=\exp_1\hspace{-0.1mm}\circ f$, the set of critical values of $g(w)=w^\ell e^{2\pi i \Phi(w)}$, given by $\mathcal{CV}(g)=g\left(\mathcal{C}(g)\right)$, is the union of
    \begin{equation}
        \exp_1\big( \mathcal{CV}(f)\hspace{0.1mm}\backslash \hspace{0.1mm}\{\infty\} \big)
    \end{equation}
    and, as long as $0$ (resp. $\infty$) is in $\mathcal{D}(g)$ (see \pref{prop:Regular0inf}), one of the following points:
    \begin{enumerate}[label = (\roman*)] 
        \vspace{-1mm}
        \item $0$ (resp. $\infty$) if $\ell\geq 2$;
        \vspace{-0.5mm}
        \item $\infty$ (resp. $0$) if $\ell\leq -2$;
        \vspace{-0.5mm}
        \item $e^{2\pi i \Phi(0)}$ (resp. $e^{2\pi i \Phi(\infty)}$) if $\ell=0$ with $\Phi'(0)=0$ \big(resp. $\Phi'(\infty)=0$\big). In this case, as $f=\Phi\circ\exp_1$ and $g=\exp_1\hspace{-0.15mm}\circ\hspace{0.25mm}\Phi$, note that $\mathcal{CV}(f)=\Phi\big(\mathcal{C}(\Phi)\cap\CC^*\big)$, and so $\mathcal{CV}(g) = \exp_1\big(\mathcal{CV}(\Phi)\backslash\{\infty\}\big)\subset\CC^*$.
        \vspace{-1mm}
    \end{enumerate}
\end{remark}

A projectable function $f$ may have no critical points (and thus no critical values) at all, such as the exponential or tangent map (with two asymptotic values), or the non-periodic entire function
\begin{equation}
    \label{eq:HerringExAV}
    f(z)=\frac{1}{2\pi i}\int_0^{2\pi i z} e^{-e^t} dt = z + \sum_{k=1}^\infty \frac{(-1)^k}{2\pi i \hspace{0.2mm} k! \hspace{0.2mm} k} \big(e^{2\pi i k z}-1\big),
\end{equation}
which has infinitely many asymptotic values (as locally omitted values), i.e. $\#\mathcal{S}(f)=\infty$ (see \cite[Ex.~4]{Herring1998}). 

In what follows we discuss the relation between asymptotic values of a given $f$ and its projection $g$, by building corresponding asymptotic paths in different cases. Recall that, given $v\in\mathcal{AV}(g)$, an asymptotic path to $\hat{w}\in\mathcal{E}(g)$ (associated to $v$) is a curve $\gamma:[0,1)\to\mathcal{D}(g)$ such that $\gamma(t)\to \hat{w}$ and $g(\gamma(t))\to v$ as $t\to 1$.

\begin{proposition}[Projection of asymptotic values]
    \label{prop:projAV}
    Let $f$ be a projectable function via $\exp_1$, written as $f(z) = \ell z + \Phi(e^{2\pi i z})$ for some $\ell\in\ZZ$ and $\Phi$ meromorphic in $\CC^*$, and $g$ its exponential projection. Then $\{0,\infty\}\subset\mathcal{AV}(g)$. Furthermore,
    \begin{enumerate}[label = (\roman*)] 
        \vspace{-1mm}
        \item If $\ell \neq 0$, then $\infty\in\mathcal{AV}(f)$, and $\mathcal{AV}(g)\hspace{0.2mm}\cap\hspace{0.2mm}\CC^*=\exp_1\left(\mathcal{AV}(f)\backslash\{\infty\}\right)$.
        \vspace{-0.5mm}
        \item If $\ell=0$, then $\Phi\big(\{0,\infty\}\backslash\mathcal{E}(\Phi) \big) \subset \mathcal{AV}(f)$, and $\mathcal{AV}(g)\hspace{0.2mm}\cap\hspace{0.2mm}\CC^*\hspace{-0.2mm}=\hspace{-0.2mm}\exp_1\left(\mathcal{AV}(\Phi)\backslash\{\infty\}\right)$.
    \end{enumerate}
\end{proposition}

\begin{proof}%
    We start by showing that both $0$ and $\infty$ are always asymptotic values of $g$ through finding corresponding asymptotic paths to some essential singularity of $g$, either to the projection of a pole of $f$ via $\exp_1$ (see \pref{prop:PolesToEssential}), or to an essential singularity of $g$ at $0$ or $\infty$ in the entire case. 
    
    \noindent On the one hand, suppose that $f$ is non-entire and consider a pole $b\in\CC$ of $f$, and a neighborhood $\widetilde{U}$ of $\infty$. Let $U$ be the connected component of $f^{-1}\big(\widetilde{U}\big)$ which contains $b$, and choose $\widetilde{U}$ small enough such that $f:U\to \widetilde{U}$ is a proper map, i.e. $U$ does not contain other poles or critical points of $f$, apart from $b$ itself. Take two straight paths $\widetilde{\gamma}^{\pm}\subset i\RR^\pm \cap \widetilde{U}$, so that their preimages, say $\gamma^\pm:[0,1)\to U\hspace{0.1mm}\backslash\hspace{0.1mm} \{ b \}$, are curves landing at $b$ as $t\to 1$. Hence, $\exp_1\hspace{-0.3mm}\circ \hspace{0.5mm}\gamma^\pm$ is a path towards $e^{2\pi i b}\in\mathcal{E}(g)\cap\CC^*$, and for all $t\in[0,1)$, \vspace{-0.25mm}
    \begin{equation}
        \label{eq:projAV_semiconj}
        g\big(\exp_1\hspace{-0.1mm}\circ\hspace{0.3mm} \gamma^\pm(t)\big) = \exp_1\big( f\circ \gamma^\pm(t)\big) = \exp_1\hspace{-0.1mm}\circ\hspace{0.3mm}\widetilde{\gamma}^\pm(t).
        \vspace{-0.25mm}
    \end{equation}
    Given that $\operatorname{Im}\widetilde{\gamma}^\pm(t)\to\pm\infty$ as $t\to 1$, we conclude from \pref{eq:projAV_semiconj} that $0$ and $\infty$ are asymptotic values of $g$ associated to the paths $\exp_1\hspace{-0.3mm}\circ \hspace{0.5mm} \gamma^+(t)$ and $\exp_1\hspace{-0.3mm}\circ \hspace{0.5mm} \gamma^-(t)$, respectively. On the other hand, assume $f$ to be entire, so that, by \pref{rem:EntireCase}, $\mathcal{E}(g)\cap\CC^*=\emptyset$ but $g$ has at least one essential singularity at $E\in\{0,\infty\}$. In fact, we have either that $E\in\mathcal{E}(\Phi)$, or $\Phi(E)=\infty$ (see \pref{prop:Regular0inf}). Then, it is clear that we can find a pair of paths $\Gamma^\pm:[0,1)\to\mathcal{D}(g)$ such that $\Gamma^\pm(t)\to E$ and $\operatorname{Im}\Phi(\Gamma^\pm(t))\to\pm \infty$ as $t\to 1$, and so, as $g(w)=w^\ell e^{2\pi i \Phi(w)}$, $\Gamma^\pm$ are the asymptotic paths that we were looking for.
    
    \noindent In order to relate the asymptotic values of $g$ in $\CC^*$ to those of $f$, we split the discussion into two cases:
    \begin{enumerate}[label = (\roman*)] 
        \vspace{-0.5mm}
        \item If $\ell\neq 0$, we first show that $\infty\in\mathcal{AV}(f)$. Consider a path $\gamma_\infty:[0,1)\to\CC\backslash f^{-1}(\infty)$ to $\infty$ which is invariant under translation by $1$, and let $t^*\in[0,1)$ be such that $\gamma_\infty(t^*)=\gamma_\infty(0)+1$ (see \pref{fig:3_projAV}). Due to pseudoperiodicity, note that
        \begin{equation}
            \label{eq:3_proofAVpseudo}
            f(\gamma_\infty(t)+k)=f(\gamma_\infty(t) )+\ell k
        \end{equation}
        for all $t\in[0,1)$. From \pref{eq:3_proofAVpseudo} we obtain that $\left\{f\big( \gamma_\infty([0,t^*])+k \big)\right\}_{k\in\NN}$ is a sequence of compact sets (on the curve $f\circ\gamma_\infty$) converging to $\infty$ as $k\to\infty$, and thus $f(\gamma_\infty(t))\to\infty$ as $t\to 1$.
        
        \vspace{-1.5mm}
        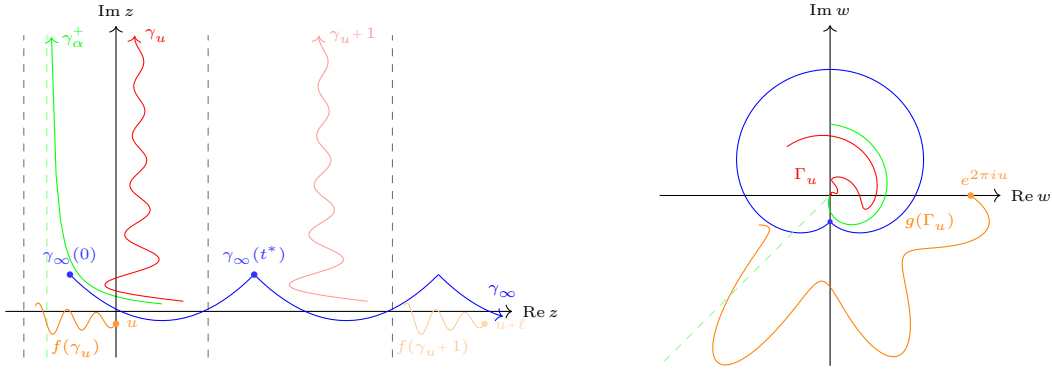
\begin{figure}[H]
        \hspace{2.6mm}\hfill
        \begin{tikzpicture}[scale=2.45]
        \draw[line width=0.1pt, color = darkgray!80, dashed] (0.5,-0.25) -- (0.5,1.5);
        \draw[line width=0.1pt, color = darkgray!80, dashed] (-0.5,-0.25) -- (-0.5,1.5);
        \draw[line width=0.1pt, color = darkgray!80, dashed] (1.5,-0.25) -- (1.5,1.5);
        \draw[line width=0.1pt, color = green!50, dashed] (-0.375,-0.25) -- (-0.375,1.5);
        \draw[->,color=green,domain=-0.25:0.35] plot (-\x,{0.038/(0.375-\x)-0.02}) node[right] {\tiny$\gamma_\alpha^{\shortplus}$};
        \draw[->] (-0.6,0) -- (2.15,0) node[right] {\tiny$\operatorname{Re}z$};
        \draw[->] (0,-0.25) -- (0,1.55) node[above] {\tiny$\operatorname{Im}z$};
        \draw[color=blue,domain=-0.25:0.75] plot (\x,{\x*\x-0.5*\x+1/80});
        \draw[color=blue,domain=0.75:1.75] plot (\x,{(\x-1)*(\x-1)-0.5*(\x-1)+1/80});
        \draw[->,color=blue,domain=1.75:2.1] plot (\x,{(\x-2)*(\x-2)-0.5*(\x-2)+1/80}) node[above=1.25mm] {\tiny$\gamma_\infty^{}$};
        \fill[blue!80] (-0.25,0.2) circle (0.5pt) node[above]{\tiny$\gamma_\infty^{}(0)$};
        \fill[blue!80] (0.75,0.2) circle (0.5pt) node[above]{\tiny$\gamma_\infty^{}(t^*)$};
        \draw[samples=65,color=orange,domain=-0.4375:0] plot (\x,{-1.05*sin(40*(\x+0.695) r)/(40*(\x+0.695)) - 0.05}) node[below left=1.05mm]{\tiny$f(\gamma_u^{})$};
        \draw[->, samples=265,color=red,domain=0.055:1.5] plot ({0.85*sin(23.5*(\x+0.047) r)/(23.5*(\x+0.047))+0.125},\x) node[right]{\tiny$\gamma_u^{}$};
        \fill[orange!80] (0,-0.0672512) circle (0.5pt) node[right]{\tiny$u$};
        \draw[samples=65,color=orange!40,domain=-0.4375:0] plot ({\x+2},{-1.05*sin(40*(\x+0.695) r)/(40*(\x+0.695)) - 0.05}) node[below left=1.0mm]{\hspace{0.5mm}\tiny$\qquad f(\gamma_u^{}\shortplus\hspace{0.5mm}1)$};
        \draw[->, samples=265,color=red!40,domain=0.055:1.5] plot ({0.85*sin(23.5*(\x+0.047) r)/(23.5*(\x+0.047))+0.125+1},\x) node[right]{\tiny$\gamma_u^{}\shortplus\hspace{0.5mm}1$};
        \fill[orange!30] (2,-0.0672512) circle (0.5pt) node[right]{\tiny$u\hspace{0.1mm}\shortplus\hspace{0.1mm}\ell$};
        \end{tikzpicture}
        \hfill
        \begin{tikzpicture}[scale=1.225]
        \draw[->] (-1.85,0) -- (1.85,0) node[right] {\tiny$\operatorname{Re}w$};
        \draw[->] (0,-1.85) -- (0,1.85) node[above] {\tiny$\operatorname{Im}w$};
        \draw[samples=165,color=blue,domain=-0.25:0.75] plot ({exp(-2*pi*(\x*\x-0.5*\x+1/80))*cos(2*pi*\x r)},{exp(-2*pi*(\x*\x-0.5*\x+1/80))*sin(2*pi*\x r)});
        \draw[samples=665,color=red,domain=0.055:1.5] plot ({exp(-2*pi*\x)*cos(2*pi*(0.85*sin(23.5*(\x+0.047) r)/(23.5*(\x+0.047))+0.125) r)},{exp(-2*pi*\x)*sin(2*pi*(0.85*sin(23.5*(\x+0.047) r)/(23.5*(\x+0.047))+0.125) r)}) node[above left=0.25mm] {\tiny$\Gamma_u^{}$};
        \draw[samples=165,color=green,domain=-0.25:0.35] plot ({exp(-2*pi*(0.038/(0.375-\x)-0.02)))*cos(-2*pi*\x r)},{exp(-2*pi*(0.038/(0.375-\x)-0.02)))*sin(-2*pi*\x r)});
        \draw[line width=0.1pt, color = green!50, dashed, domain=-1.8:0] plot (\x,\x);
        \draw[samples=265,color=orange,domain=-0.4375:0] plot ({exp(-2*pi*(-1.05*sin(40*(\x+0.695) r)/(40*(\x+0.695)) - 0.05)))*cos(2*pi*\x r)},{exp(-2*pi*(-1.05*sin(40*(\x+0.695) r)/(40*(\x+0.695)) - 0.05)))*sin(2*pi*\x r)}) node[below left=1.25mm] {\tiny$g(\Gamma_u^{})$};
        \fill[blue!80] ({exp(-2*pi*0.2)*cos(2*pi*0.75 r)},{exp(-2*pi*0.2))*sin(2*pi*0.75 r)}) circle (0.8pt);
        \fill[orange!80] ({exp(-2*pi*(-0.0672512))*cos(2*pi*0 r)},{exp(-2*pi*(-0.0672512)))*sin(2*pi*0 r)}) circle (1pt) node[above]{\quad\tiny$e^{2\pi i u}$};
        \end{tikzpicture}
        \hfill\vspace{-0.5mm}
        \caption{Correspondence between asymptotic paths of $f$ (\textit{left}) and its projection $g$ (\textit{right}) via $\exp_1$ for the proof of \pref{prop:projAV}. The green dashed line refers to $\operatorname{Re}z=\frac{\alpha}{2\pi}$, $\alpha\in(-\pi,\pi]$, and the gray ones to $\operatorname{Re}z=\frac{1}{2}+k$, $k\in\ZZ$.}
        \label{fig:3_projAV}
        \end{figure}

        \vspace{-2.0mm}
        \noindent Next we consider an asymptotic path $\gamma_u$ to $\infty$ associated to some $u\in\mathcal{AV}(f)\cap\CC$. We claim that $\operatorname{Im}\gamma_u(t)$ is unbounded. To see this, suppose that $|\operatorname{Im}\gamma_u(t)|\leq M$ for some $M\in\RR^+$ and all $t\in[0,1)$, and so $\operatorname{Re}\gamma_u(t)\to\pm\infty$ as $t\to 1$. Then, there would be points $z_k:=k+i y_k\in\gamma_u$ for some unbounded sequence of integers $k$ and $|y_k|\leq M$, and by pseudoperiodicity, \vspace{-0.4mm}
        \begin{equation}
            \label{eq:3_proofAVsequence}
            f(z_k)= f(i y_k) + \ell k.
            \vspace{-0.4mm}
        \end{equation}
        Assuming without loss of generality that $i y_k\notin f^{-1}(\infty)$, the relation \pref{eq:3_proofAVsequence} implies that $f(z_k)\to\infty$ as $|k|\to\infty$, which contradicts that $\gamma_u$ is an asymptotic path for $u\in\CC$. 
        
        \noindent Furthermore, in our situation, we assert that either $\operatorname{Im}\gamma_u(t)\to +\infty$ with $0\in\mathcal{E}(g)$, or $\operatorname{Im}\gamma_u(t)\to -\infty$ with $\infty\in\mathcal{E}(g)$, as $t\to 1$. Indeed, if this were not the case, then $\Phi(0)$ or $\Phi(\infty)$, respectively, would be defined and not equal to $\infty$ (see \pref{prop:Regular0inf}), so that the explicit expression of $f$ leads to \vspace{-0.3mm}
        \begin{equation*}\vspace{-0.3mm}
            f(\gamma_u(t))=\ell \gamma_u(t) + \Phi(e^{2\pi i \gamma_u(t)})\to \infty \vspace{-0.3mm}  
        \end{equation*}
        as $t\to 1$, in contradiction with $u\neq\infty$. Hence, we conclude that $e^{2\pi i u}\in\mathcal{AV}(g)\cap\CC^*$, since $\Gamma_u:=\exp_1\hspace{-0.3mm}\circ\hspace{0.4mm} \gamma_u$ is a corresponding asymptotic path to either $0\in\mathcal{E}(g)$ in the first case, or $\infty\in\mathcal{E}(g)$ in the second. 
        
        \noindent The reverse inclusion follows from the fact that an asymptotic path $\Gamma:[0,1)\to\mathcal{D}(g)$ associated to $v\in\mathcal{AV}(g)\cap\CC^*$, can only converge to $0$ or $\infty$ as $t\to 1$. In fact, if this were not true, then any lift of $\Gamma$ via $\exp_1$ would tend to $b\in f^{-1}(\infty)$ as $t\to 1$ (see \pref{prop:PolesToEssential}), in contradiction with $\lim\limits_{t\to 1}g(\Gamma(t))\in\CC^*$. 
        Then, every curve $\gamma\subset\exp_1^{-1}\Gamma$ has unbounded imaginary part. Given that $g(\Gamma(t))=e^{2\pi i f(\gamma(t))}$ for all $t\in[0,1)$, any such $\gamma$ is an asymptotic path of $f$ associated to some $u\in\CC$ such that $v=e^{2\pi i u}$.
        \item If $\ell=0$, i.e. $f$ is $1$-periodic with $f=\Phi\circ\exp_1$ and $g=\exp_1\hspace{-0.3mm}\circ \hspace{0.5mm}\Phi$, then the relation \pref{eq:singValueComp} implies that \vspace{-0.3mm}
        \begin{equation}
            \label{eq:ProjAV_periodicL}
            \mathcal{AV}(f) \subset \mathcal{AV}(\Phi)\hspace{0.2mm}\cup\hspace{0.2mm} \Phi\big(\{0,\infty\}\backslash\mathcal{E}(\Phi) \big), \quad \mbox{ and } \quad  \mathcal{AV}(g)\subset \{0,\infty\}\hspace{0.2mm}\cup\hspace{0.2mm} \exp_1\left( \mathcal{AV}(\Phi)\backslash\{\infty\} \right).\vspace{-0.3mm}
        \end{equation}
        First, notice that $\Phi(0)$ (resp. $\Phi(\infty)$), whenever defined, is an asymptotic value of $f(z)=\Phi(e^{2\pi i z})$ along any path in $\HH^+$ (resp. $\HH^{-}$) with unbounded imaginary part. For any $\alpha\in(-\pi,\pi]$, consider the curves $\gamma_\alpha^\pm:[0,1)\to\CC$ with
        \begin{equation*}
            \operatorname{Im}\gamma_\alpha^\pm(t)\to \pm\infty, \quad \mbox{ and } \quad  \operatorname{Re}\gamma_\alpha^\pm(t)\to \frac{\alpha}{2\pi},
        \end{equation*}
        as $t\to 1$ (see \pref{fig:3_projAV}). Since $\exp_1\circ\hspace{0.4mm}\gamma_\alpha^{+}$ and $\exp_1\circ\hspace{0.4mm}\gamma_\alpha^{-}$ land at $0$ and $\infty$, respectively, in a direction at an angle $\alpha$ with the positive real axis, we deduce that all possible asymptotic values of $\Phi$ are attainable by $f$ along $\gamma_\alpha^\pm$ for appropriate values of $\alpha$, and equality holds in \pref{eq:ProjAV_periodicL}.

        \noindent The remaining inclusion for $g(w)=e^{2\pi i \Phi(w)}$ follows from $\{0,\infty\}\subset\mathcal{AV}(g)$ (as shown at the beginning of the proof), and the fact that, for any $u\in\mathcal{AV}(\Phi)\cap\CC$, $e^{2\pi i u}$ is clearly an asymptotic value of $g$.
        $\hfill\square$
    \end{enumerate}
\end{proof}\vspace{-2mm}

Finally we proceed to prove \pref{thm:B_FiniteType} to discern those exponential projections $g$ of finite-type.

\begin{customproof}{of \pref{thm:B_FiniteType}}
    \textup{We split the class of projectable functions $f$ via $\exp_1$, written as $f(z)=\ell z + \Phi(e^{2\pi i z})$ with $\ell\in\ZZ$ and $\Phi$ meromorphic in $\CC^*$ (see \pref{thm:A_FormProj}), into three disjoint cases:
    \begin{enumerate}[label = (\roman*)] 
        \vspace{-1.05mm}
        \item[\textit{(i-a)}] If $f$ is simply pseudoperiodic with $\ell\neq 0$, i.e. $\Phi\circ\exp_1$ is a simply $1$-periodic function, then for any critical point $c$, we have $\{c+k\}_{k\in\ZZ}\subset \mathcal{C}(f)$, as $f'$ is periodic. By pseudoperiodicity, \vspace{-0.4mm}
        \begin{equation*}
            f(c+k)=f(c)+\ell k \neq f(c) \vspace{-0.4mm} 
        \end{equation*}
        for all $k\in\ZZ^*$, and so $\{f(c+k)\}_k$ is a collection of distinct critical values, except if $c$ is a pole. This, together with \pref{rem:projCV}, shows that $\#\mathcal{CV}(g)<\infty$ if and only if $f$ has finitely many critical values in the strip $\widetilde{S}_0 =\{z:-\frac{\ell}{2}<\operatorname{Re}z\leq \frac{\ell}{2} \}$, as all other critical values of $f$ differ by an integer multiple of $\ell$ from those, i.e. $\exp_1\hspace{-0.3mm}\big( \mathcal{CV}(f)\backslash\{\infty\} \big) = \exp_1\hspace{-0.3mm}\big( \mathcal{CV}(f)\cap\widetilde{S}_0 \big)$. The same argument applies to $\mathcal{AV}(g)$ in view of \pref{prop:projAV}, by considering instead the image under $\exp_1$ of the sequences $\{f(\gamma+k)\}_k$ for asymptotic paths $\gamma$ associated to any $u\in\mathcal{AV}(f)\cap\CC$.
            \vspace{-0.0mm}
        \item[\textit{(i-b)}] If $f$ is doubly pseudoperiodic with $\ell\neq 0$ and, say, $f(z+\tau)=f(z)+\eta_\tau$ for some $\tau\in\HH^+$ and $\eta_\tau\in\CC$, then, given $c\in\mathcal{C}(f)$, we see that $\{c+k+m\tau\}_{k,m\in\ZZ}\subset \mathcal{C}(f)$, since $f'$ is doubly periodic. Moreover, \vspace{-0.3mm}
        \begin{equation}
            \label{eq:proofThmB_doublyPseudo}
            f(c+k+m\tau)=f(c)+\ell k + \eta_\tau m,
            \vspace{-0.6mm}
        \end{equation}
        which is a critical value different from $f(c)$ for all $k,m\in\ZZ^*$, provided that $c\in\CC\backslash f^{-1}(\infty)$. Hence, using \pref{eq:proofThmB_doublyPseudo}, we deduce from \pref{rem:projCV} that the set \vspace{-0.4mm}
        \begin{equation*}
            \mathcal{CV}(g)\hspace{0.2mm}\cap\hspace{0.2mm}\CC^*=\exp_1\hspace{-0.2mm}\left(\mathcal{CV}(f)\backslash\{\infty\}\right) = \big\{e^{2\pi i f(c)}e^{2\pi i \eta_\tau m}: c\in\mathcal{C}(f),m\in\ZZ\big\} \vspace{-0.4mm}
        \end{equation*}
        is finite if and only if $\eta_\tau\in\QQ$, and also $\#\big(\mathcal{CV}(f)\cap\widetilde{S}_0\big)<\infty$ as argued above. The same follows for the set $\mathcal{AV}(g)$ from \pref{eq:proofThmB_doublyPseudo}, proceeding in terms of the projection via $\exp_1$ of asymptotic paths of $f$.
        \item[\textit{(ii)}] If $\ell=0$, i.e. $f$ is $1$-periodic and $g=\exp_1\hspace{-0.3mm}\circ \hspace{0.4mm}\Phi$, the statement follows from \pref{rem:projCV} and \pref{prop:projAV}.
        \vspace{-3.5mm}
    \end{enumerate}
    In all cases, $\{0,\infty\}\subset\mathcal{AV}(g)$ due to \pref{prop:projAV}. From Propositions \ref{prop:Regular0inf} and \ref{prop:projC}, we obtain that $0\in\mathcal{C}(g)$ if and only if $0\notin\mathcal{E}(\Phi)\cup\Phi^{-1}(\infty)$ and either $|\ell|\geq 2$, or $\ell=0$ with $\Phi'(0)=0$. Notice that the first condition does not hold if $f$ were doubly pseudoperiodic (see \pref{rem:doublyPcase}), and in the $1$-periodic case, $g(0)\in\CC^*$ when defined. In the remaining cases, as $g(w)=w^\ell e^{2\pi i \Phi(w)}$, we conclude that $0\in\mathcal{D}(g)$ and either $g(0)=0\in\mathcal{CV}(g)$ if $\ell\geq 2$, or $g(0)=\infty\in\mathcal{CV}(g)$ if $\ell\leq -2$. We can argue similarly for the point at $\infty$.
    $\hfill\square$} 
\end{customproof}

The following corollary for projections $g$ of functions in the class $\mathbf{R}_\ell$ (see \pref{def:classR}), which is central to our study, follows directly from Propositions \ref{prop:PolesToEssential} and \ref{prop:Regular0inf} on $\mathcal{E}(g)$, and \pref{rem:projCV} and \pref{prop:projAV} on $\mathcal{S}(g)$. Recall that the (non-entire) periodic part $R(e^{2\pi i z})$ of $f\in\mathbf{R}_\ell$ (where $R$ is a non-constant rational map) has finite limits as $\operatorname{Im}z\to\pm\infty$. In this situation,  note that $\mathcal{AV}(f)=\{\infty\}$ in the non-periodic case, while $\mathcal{AV}(f)=\{R(0),R(\infty)\}\subset\CC$ in the $1$-periodic case ($\ell=0$).
\begin{corollary}[Projections of functions in class $\mathbf{R}_\ell$]
    \label{cor:classRsv}
    Let $f\in\mathbf{R}_\ell$, written as $f(z)=\ell z + R(e^{2\pi i z})$, $\ell\in\ZZ$ and $R$ as a non-constant rational map with $R(0)\neq\infty$ and $R(\infty)\neq\infty$, and $g$ its exponential projection. Then $g$ is a finite-type map with $\mathcal{AV}(g) = \{0,\infty\}$, and $\mathcal{E}(g)\hspace{-0.2mm}=\hspace{-0.2mm}\exp_1\hspace{-0.4mm}\left(f^{-1}(\infty)\right)$. Moreover,
    \begin{enumerate}[label = (\roman*)] 
        \vspace{-1.35mm}
        \item If $\ell = 1$, then $0$ and $\infty$ are fixed points of $g$ with $g'(0)=e^{2\pi i R(0)}$ and $g'(\infty)=e^{-2\pi i R(\infty)}$, while both are fixed critical points if $\ell\geq 2$.
        \item If $\ell = -1$, then $\{0,\infty\}$ is a $2$-cycle with multiplier $e^{2\pi i \left(R(\infty)-R(0)\right)}$, while it is critical if $\ell\leq -2$.
        \vspace{0.15mm}
        \item If $\ell = 0$, then $0$ (resp. $\infty$) is a critical point of $g$ if and only if $R'(0)=0$ (resp. $R'(\infty)=0$).
        \vspace{-0.05mm}
    \end{enumerate}
\end{corollary}
%
\section{Lifting periodic Fatou components}
\label{sec:4_Lifting}

A good number of the explicit examples of entire functions with Baker or wandering domains come from Herman's idea \cite{Herman1985} on lifting periodic Fatou components via an appropriate branch of the logarithm (see e.g. \cite{Baker1984, Rippon2008}). This method is indeed applicable to any projectable function $f$ (see \pref{thm:A_FormProj}), which may be meromorphic. Here we infer the iterative behavior of $f$ from that of its exponential projection $g$; especially from those $g$ of finite-type (see \pref{thm:B_FiniteType}).

The lifting procedure in our context grounds on a theorem by Bergweiler \cite{Bergweiler1995j} in the entire case, who proved that the dynamical partition of $\CC$ into the Fatou and Julia sets is preserved via $\exp_1$ (see \cite{Keen1988}, too). 
This result was extended to Bolsch's class (and beyond) by Zheng \cite[Cor. 3.1]{Zheng2005}, building on the fact that, for a transcendental map $g\in\mathbf{K}$ with at least one non-omitted essential pole (see \pref{def:BolschClass}, and \cite{Bolsch1997}),\vspace{-0.5mm}
\begin{equation}
    \label{eq:generalMeroJulia}
    \mathcal{J}(g) = \overline{\bigcup\limits_{n=0}^\infty g^{-n}(\mathcal{E}(g))} = \overline{\bigcup\limits_{n=1}^\infty \left(\mathcal{E}(g^n)\backslash \mathcal{E}(g^{n-1})\right)}, \vspace{-0.5mm} 
\end{equation}
where $\mathcal{E}(g^n)$ denotes the set of essential singularities of $g^n$ (let $\mathcal{E}(g^0):=\emptyset$).

To be precise, the dynamical relation between $f$ and $g$ is given by: $\exp_1 \mathcal{F}(f) = \mathcal{F}(g)\cap \CC^*$, as stated in \pref{eq:expFatouJulia}, taking into account that $0$ and/or $\infty$ (the omitted values of $\exp_1$) may be or not be in the Fatou set of $g$.
Equivalently, given that the inverse image operation commutes with complements, we have that
\begin{equation}
    \label{eq:logFatouJulia}
    \exp_1^{-1} \hspace{-0.2mm}\left(\mathcal{F}(g) \cap \CC^* \right) = \mathcal{F}(f), \quad \exp_1^{-1}\hspace{-0.2mm}\left( \mathcal{J}(g)\cap \CC^* \right) = \mathcal{J}(f) \hspace{0.3mm} \backslash \hspace{0.3mm} \{\infty\}.
    \vspace{-1mm}
\end{equation}

In this section we reinforce the fact that a Fatou component of $f$ which comes from the lift of a periodic component $V\subset\mathcal{F}(g)$, needs not to be of the same type. In particular, we shall give conditions which makes such a $V$ deliver a wandering or Baker domains of $f$ by lifting. For this purpose, the location of the points at $0$ and $\infty$ with respect to $V$, which is encoded by the following standard notion, is going to be crucial. The \textit{fill} of a set $A\subset\hCC$, denoted by $\operatorname{fill}(A)$, is the union of $A$ and all bounded components of $\CC\backslash A$.

\begin{remark}[Fill of a Fatou component]
    \label{rem:FillFatou}
    Consider a function $g$ in Bolsch's class, and view a Fatou component $V$ of $g$ as a domain in $\hCC\backslash \mathcal{E}(g)$. Observe that if $\infty\in V$, then $\operatorname{fill}(V)=\hCC$. However, if $V\subset\CC$, then $\operatorname{fill}(V)$ is a simply-connected domain in $\CC$ (as $V$ is open). In this case, we have that $w\in\operatorname{fill}(V)$ if and only if there is a Jordan curve $\Gamma\subset V$ such that $w\in\mathrm{int}(\Gamma)$, where $\mathrm{int}(\Gamma)$ is the bounded component of $\CC\backslash\Gamma$.
\end{remark}

In this context, the following lemma provides a simple way to single out those Fatou components $V$ of $g$ which do not lift via $\exp_1$ to infinitely many distinct Fatou components of $f$.
\begin{lemma}[Lifting components surrounding $0$]
    \label{lem:FillBaker}
    Let $f$ be a projectable function via $\exp_1$, and $g$ its exponential projection. Suppose $V$ is a component of $\mathcal{F}(g)$, and $U$ is a component of $\exp_1^{-1}(V\cap\CC^*)$. Then
    $$0\in\operatorname{fill}(V)\quad \iff \quad U+k=U, \mbox{ for all } k\in\ZZ.$$
    In this case, $U$ is an unbounded component of $\mathcal{F}(f)$ which is equal to $\exp_1^{-1}(V\cap\CC^*)$.
\end{lemma}
\begin{proof}
    Recall that $\exp_1^{-1}(V\cap\CC^*)=\{U+k\}_{k\in\ZZ}^{}\subset\mathcal{F}(f)$ due to the relation \pref{eq:logFatouJulia}. Notice that the statement is clear when $\infty$ (resp. $0$) lies in $V$, as $U$ must then contain a lower (resp. upper) half-plane. 
    
    \noindent In the remaining cases, $\operatorname{fill}(V)\subset\CC$, and it follows from \pref{rem:FillFatou} that $0\in\mathrm{fill}(V)$ if and only if there is a simple closed curve $\Gamma\subset V\cap\CC^*$ such that $0\in \mathrm{int}(\Gamma)$. Hence, $\gamma:=\exp_1^{-1}\Gamma$ consists of a single (simple) curve such that $\gamma+k=\gamma$ for all $k\in\ZZ$. Since $\gamma$ is connected and belongs to $U$, we conclude that $U$ is an unbounded Fatou component of $f$, and actually $U=\exp_1^{-1}(V\cap\CC^*)$ . 
    
    \noindent Conversely, if $U+k=U$ for any $k\in\ZZ$, then take a simple curve $\gamma\subset U$ which is invariant under translation by $\pm 1$, and so projects via $\exp_1$ to a closed curve in $V\cap\CC^*$ which surrounds the origin.  
    $\hfill\square$
\end{proof}

\subsection{Pseudoperiodic points}

The key players for the detection and tracking of wandering domains are going to be the pseudoperiodic points of $f$ (of type $(p,\sigma)$; see \pref{def:PseudoperiodicPoint}). These may be identified as \textit{periodic points $z^*$ modulo an integer $\sigma$}, in the sense that $f^p(z^*)-z^*=\sigma$ for some $p\geq 1$. The following lemma is straightforward from the fact that $g^p\circ \exp_1 = \exp_1 \circ f^p$, and the explicit relation between the iterates of $f$, which in the general form $f(z)=\ell z +\Phi(e^{2\pi i z})$ (as stated in \pref{thm:A_FormProj}), is given by \vspace{-1mm}
\begin{equation}
    \label{eq:4_RelationIteratesF}
    f^{n}(z) = \ell^{n} z + \sum_{j=0}^{n-1} \ell^{\scriptstyle n-1-j} \hspace{0.2mm}\Phi\big( e^{2\pi i f^j(z)}\big)
    \vspace{-0.5mm}
\end{equation}
for every $n\geq 1$. Recall that by a $p$-periodic point we mean a point of minimal period $p$, and also that a $(p,\sigma)$-pseudoperiodic point refers to one of minimal type $(p,\sigma)$.

\begin{lemma}[Characterization of pseudoperodic points]
    \label{lem:characPseudoPoint}
    Let $f$ be a projectable function via $\exp_1$, written as $f(z)=\ell z + \Phi(e^{2\pi i z})$ with $\ell\in\ZZ$ and $\Phi$ meromorphic in $\CC^*$, and $g$ its exponential projection. Then $z^*$ is a $(p,\sigma)$-pseudoperiodic point of $f$, where $p\geq 1$ and $\sigma\in\ZZ$, if and only if $e^{2\pi i z^*}$ is a $p$-periodic point of $g$. In this case, $z^*$ is a solution of
    \vspace{-1.0mm}
    \begin{equation}\label{eq:characPseudo}
        (\ell^p-1) z + \sum_{j=0}^{p-1} \ell^{\scriptstyle p-1-j} \hspace{0.2mm}\Phi\hspace{-0.4mm}\left( g^j(e^{2\pi i z})\right) = \sigma.
    \end{equation}
\end{lemma}
\vspace{-1.5mm}
Note that if $z^*$ is actually $q$-periodic, then $e^{2\pi i z^*}$ is a $p$-periodic point of $g$, where $p|q$, but the reverse is not true. Observe here that we may have $p<q$, which indicates that $z^*$ is indeed pseudoperiodic of minimal type $(p,\sigma)$ for some $\sigma\in\ZZ^*$; consider e.g. $f(z)=-z-\sin{2\pi z}$, for which $z^*=\shortminus\frac{1}{2}$ is $2$-periodic but $(1,1)$-pseudoperiodic. We are interested in the limiting behavior of pseudoperiodic points under iteration, so we next specify the orbit of those (and of their translates by an integer) according to the pseudoperiod of $f$.

\begin{proposition}[Iterates of pseudoperiodic points and of their translates]
    \label{prop:IteratesPseudo}
    Let $f$ be a projectable function via $\exp_1$, and $\ell\in\ZZ$ such that $f(z+1)=f(z)+\ell$. Suppose $z^*$ is a pseudoperiodic point of $f$ of minimal type $(p,\sigma)$, where $p\geq 1$ and $\sigma\in\ZZ$, and let $z^*_k:=z^*+k$, $k\in\ZZ$. Then,
    \begin{enumerate}[label = (\roman*)] 
        \vspace{-1mm}
        \item If $|\ell|\geq 2$, the points $\{z^*_k\}_{k}^{}$ are pseudoperiodic of minimal type $(p,\sigma+(\ell^p-1)k)$, and for all $m\in\NN$,\vspace{-0.5mm}
        \begin{equation}
            \label{eq:statement_IteratesPseudo}
            f^{mp}(z^*_k)=z^*_k+(\ell^{mp}-1)(k-\delta), \quad \mbox{ where } \quad  \delta:=\frac{\sigma}{1-\ell^{{}^p}}. \vspace{-0.5mm}
        \end{equation}
        In particular, $f^{mp}(z_k^*)\to\infty$ as $m\to\infty$, except for $z^*_\delta$ if $\delta\in\ZZ$ (which is $p$-periodic).
        \item($\ell^p$=\hspace{0.4mm}1 case) If $\ell=1$, or $\ell=\shortminus 1$ with $p$ even, the points $\{z^*_k\}_{k}^{}$ are pseudoperiodic of minimal type $(p,\sigma)$, and for all $m\in\NN$,
        \begin{equation}
            \label{eq:statement2_IteratesPseudo}
            f^{mp}(z^*_k)=z^*_k+m\sigma.
        \end{equation}
        In particular, $f^{mp}(z_k^*)\to\infty$ as $m\to\infty$, unless $\sigma=0$ (in which case each $z_k^*$ is $p$-periodic).
        \item($\ell^p$=\hspace{0.4mm}$\shortminus$1 case) If $\ell=-1$ with $p$ odd, the points $\{z^*_k\}_{k}^{}$ are pseudoperiodic of minimal type $(p,\sigma-2k)$. \\ In particular, each $z_k^*$ is $2p$-periodic, except for $z^*_{\sfrac{\sigma}{2}}$ if $\frac{\sigma}{2}\in\ZZ$ (which is $p$-periodic).
        \item(1-periodic case) If $\ell=0$, the point $z^*_\sigma$ is $p$-periodic, and $f(z^*_k)=z^*_\sigma$ for all $k\in\ZZ$.
    \end{enumerate}
\end{proposition}

\begin{proof}
    Since $f^p(z^*)=z^*+
    \sigma$, \pref{lem:ProjCond} and induction lead to the pseudoperiodicity relation
    \vspace{-1mm}
    \begin{equation}
        \label{eq:proofPseudo_Iterates}  
        f^{mp}(z^*+k) = z^* + \ell^{mp} k + \sigma G_m, \quad\hspace{0.4mm} \mbox{ where } \quad G_m:=\sum_{j=0}^{m-1} \ell^{j p}, \vspace{-1mm}  
    \end{equation}
    for all $k\in\ZZ$, $m\in\NN$ (let $G_0:=0$). We distinguish cases based on the sum of the geometric series $G_n$:
    \begin{enumerate}[label = (\roman*)] 
        \vspace{-1mm}
        \item If $|\ell|\geq 2$, then it follows from \pref{eq:proofPseudo_Iterates} that, for each $z_k^*=z+k$,
        \vspace{-1mm}
        \begin{equation}
            \label{eq:proofPseudo_IteratesL2}  
            f^{mp}(z^*_k) = z^*_k + (\ell^{mp}-1) k + \sigma G_m, \quad \mbox{ where } \quad G_m=\frac{\ell^{mp}-1}{\ell^p-1}, \hspace{1.2mm}\qquad \vspace{-1mm} 
        \end{equation}
        which is equivalent to \pref{eq:statement_IteratesPseudo} because $\sigma G_m=(1-\ell^{mp})\delta$ for all $m$. Notice that $|G_m|\to\infty$ as $m\to\infty$, and we deduce from \pref{eq:proofPseudo_IteratesL2} that $z_k^*$ is $p$-periodic if and only if $(\ell^{mp}-1)k+\sigma G_m = 0$ for some $k\in\ZZ$. 
        Hence, $z_k^*$ escapes to $\infty$ under iteration, unless $k=\delta\in\ZZ$.
        \item If $\ell=1$, or $\ell=-1$ with $p$ even (i.e. $\ell^p=1$), then $G_m = m$. Thus, the expression \pref{eq:statement2_IteratesPseudo} on the iterates of $z_k^*$ under $f$ comes directly from \pref{eq:proofPseudo_Iterates}, and the conclusion follows as before.
        \item If $\ell=-1$ with $p$ odd (i.e. $\ell^p=-1$), then, for any $s\in\NN$, $G_{2s+1}=1$ while $G_{2s}=0$, that is, 
        \vspace{-0.5mm}
        \begin{equation*}
            \label{eq:proofPseudo_IteratesL3}  
            f^{(2s+1)p}(z^*_k)=z^*_k+\sigma-2k \quad \mbox{ and } \quad f^{2sp}(z^*_k)=z^*_k,\hspace{0.6mm} \vspace{-0.5mm} 
        \end{equation*}
        due to \pref{eq:proofPseudo_Iterates}.
        Observe that $z^*_k$ is $2p$-periodic (indeed $(p,\sigma-2k)$-pseudoperiodic), unless $k=\frac{\sigma}{2}\in\ZZ$.
        \item If $\ell=0$, i.e. $f$ is $1$-periodic, then clearly, for each $m\geq 1$, $G_m=1$ and so $f^{mp}(z^*_k)=z^*_\sigma$ for all $k$ by \pref{eq:proofPseudo_Iterates}.
        \vspace{-4.0mm}
    \end{enumerate}
    \noindent Moreover, note that in all cases the pseudoperiodic points $z_k^*$ are of minimal type. If this were not true, say $z_k^*$ is $(\tilde{p},\tilde{\sigma})$-pseudoperiodic for some $k\in\ZZ^*$, $\tilde{p}<p$ and $\tilde{\sigma}\in\ZZ$, then \pref{lem:ProjCond} would imply that \vspace{-0.5mm}
    \begin{equation*}
        f^{\tilde{p}}(z_k^*-k) = z_k^*+\tilde{\sigma}-\ell^{\tilde{p}} k,
        \vspace{-0.5mm}
    \end{equation*}
    in contradiction with $z^*=z_k^*-k$ being of minimal type $(p,\sigma)$.
    $\hfill\square$
\end{proof}

Observe that if we are able to identify a pseudoperiodic (but non-periodic) point $z^*$ of $f$ (see \pref{lem:characPseudoPoint}) in $\mathcal{F}(f)$, then $z^*$ (and all but at most one of its translates by an integer) turns out to lie in a escaping Fatou component in the cases \textit{(i)} and \textit{(ii)} of \pref{prop:IteratesPseudo}. This, together with \pref{lem:FillBaker}, leads directly to the following result which gives a sufficient condition for $f$ to have indeed a wandering domain (containing $z^*$).

\begin{corollary}[Wandering domains through pseudoperiodic points]
    \label{cor:WDviaPseudoperiodicity}
    Let $f$ be a projectable function via $\exp_1$, and $\ell\in\ZZ$ such that $f(z+1)=f(z)+\ell$. Suppose $z^*$ is a $(p,\sigma)$-pseudoperiodic point of $f$ with $\sigma\in\ZZ^*$ (and $p\geq 1$), and is contained in a Fatou component $U$ of $f$. If $0\notin\operatorname{fill}(\exp_1 U)$ and either $|\ell|\geq 2$, or $\ell^p=1$, then $U$ is a escaping wandering domain.
\end{corollary}

\subsection{Lifting components of projections of finite-type}

Zheng in \cite[Thm.~3.3]{Zheng2005} analyzed the connection between the types of Fatou components of $f$ and $g$ in our context, showing e.g. that if $f$ does not have wandering domains, then $g$ does not either. In the opposite direction we have the following theorem. Here we concentrate on (non-periodic) functions $f$ whose projection $g$ is of finite-type (detailed in \pref{thm:B_FiniteType}), and so $g$ has no Baker domains nor wandering components (see e.g. \cite[Thm. E, F]{Baker2001}). The special case where $f$ is periodic, is going to be clarified at the end of the section. 
But first let us recall that the basin of attraction of an attracting $p$-periodic point $w_0$ of $g$ is defined as \vspace{-0.5mm}
\begin{equation}
    \label{eq:4_basinAttraction}
    \mathcal{A}(w_0):=\left\{ w: g^{mp}(w)\to w_0 \quad \mbox{ as } \quad m\to\infty \right\}, \vspace{-0.5mm}
\end{equation}
and the (Fatou) component containing $w_0$ is called its \textit{immediate basin}, denoted by $\mathcal{A}^*(w_0)$. For a set $A\subset\hCC$, we denote by $\partial A$ its boundary in $\hCC$, and by $\overline{A}$ its closure in $\hCC$ as done previously.

Similarly, by the \textit{immediate basin of attraction} of a parabolic $\tilde{p}$-periodic point $\widetilde{w}_0$ of $g$, we understand the union of Fatou components $V$ in $\mathcal{A}(\widetilde{w}_0)$ for which $\widetilde{w}_0\in\partial V$. We remark that if $\big( (g^{\tilde{p}})'(\widetilde{w}_0) \big)^q = 1$ for some $q\in\NN^*$, then the number of attracting (invariant) petals for $g^q$ at $\widetilde{w}_0$, in the sense of the \textit{Leau-Fatou flower theorem}, is an integer multiple of $q$, and the map $g$ permutes these petals in cycles of length $q$ (see more details in \cite[\S 10]{Milnor2006}). Hence, any component $V\subset\mathcal{A}^*(\widetilde{w}_0)$ is indeed a Fatou component of period $\tilde{p}\hspace{0.2mm} q$. This is going to be important when lifting parabolic basins in the proof of cases \textit{(2-i)} and \textit{(2-ii)}  that follows.

\begin{theorem}[Lifting periodic components]
    \label{thm:logFatouJulia}
    Let $f$ be a non-periodic projectable function via $\exp_1$, written as $f(z)=\ell z + \Phi(e^{2\pi i z})$ for some $\ell\in\ZZ^*$ and $\Phi$ meromorphic in $\CC^*$, and $g$ its exponential projection. 
    Suppose $g$ is a finite-type map, $V$ is a $p$-periodic component of $\mathcal{F}(g)$, and $U$ is a connected component of $\exp_1^{-1}(V\cap\CC^*)$. 
    Then $U$ is a $p$-periodic Baker domain of $f$ if and only if one of the following holds:
    \begin{enumerate}[label = (\arabic*)] 
    \vspace{-1mm}
        \item $0\in \operatorname{fill}(V)$.
        \item $\{0,\infty\}\cap\partial\operatorname{fill}(V)\neq\emptyset$ and either
        \begin{enumerate}[label = (\roman*)]
            \item($p$\hspace{0.2mm}=\hspace{0.4mm}1 case) $\ell=1$ with $\Phi(0)=0$ (resp. $\Phi(\infty)=0$) and $g^m|_V^{}\to 0$ (resp. $\infty$) as $m\to\infty$; or \vspace{0.5mm}
            \item($p$\hspace{0.2mm}=\hspace{0.4mm}2 case) $\ell=-1$ with $\Phi(0)=\Phi(\infty)$ and $g^{2m}|_V^{}\to 0$ or $\infty$ as $m\to\infty$.
        \end{enumerate}
    \end{enumerate}\vspace{-0.5mm}
    Alternatively, either
    \begin{enumerate}[label = (\alph*)] 
    \vspace{-1mm}
        \item $U$ is a component of $\mathcal{F}(f)$ of period $p$, or $2p$ (which is only possible if $\ell^p$=\hspace{0.4mm}$\shortminus$1), of the same type as $V$; or
        \item $U$ is a wandering domain with $f^p(U)\subset U+\sigma$, $\sigma\in\ZZ^*$, which is unbounded if $\{0,\infty\}\cap\partial\operatorname{fill}(V)\neq\emptyset$.
    \end{enumerate}
\end{theorem}

\begin{proof} %
    Let $U_k:=U+k$, $k\in\ZZ$, and note that $f(z+1)=f(z)+\ell$ for all $z$, where $\ell\neq 0$. Given that $g^p(V)\subset V$, it is clear that $U_k\subset \mathcal{F}(f)$ and  $f^p(U)\subset U_\sigma$ for some $\sigma\in\ZZ$, due to the relation \pref{eq:logFatouJulia} via $\exp_1$. 

    \noindent We are going to prove this theorem by considering the following collection of mutually exclusive cases, which cover all the possibilities for the $p$-periodic component $V\subset\mathcal{F}(g)$. We proceed in terms of the location of $0$ or $\infty$ with respect to the fill of $V$, and the limit function of $\{g^{mp}|_V^{}\}_{m\in\NN}$. \vspace{0.5mm}

    \textbf{\textit{Case 1.}} \textit{$0\in\operatorname{fill}(V)$}: \vspace{-0.5mm}
    
    \noindent It follows from \pref{lem:FillBaker} that $U$ is an unbounded Fatou component of $f$ (of period $p$) which is invariant under translation by $\pm 1$. In this situation, $U$ must be a Baker domain (case \textit{(1)} of the theorem); otherwise:
    \begin{enumerate}[label = (\roman*)] 
        \vspace{-1mm} 
        \item If $U$ is a component of the immediate basin of attraction of an attracting or parabolic periodic point $z^*$, then $z^*\in \overline{U}\cap\CC$ with $f^p(z^*)=z^*$ and $f^{mp}|_U^{}\to z^*$, as $m\to\infty$. By pseudoperiodicty (see \pref{lem:ProjCond}), we have that %
        \begin{equation}
            \label{eq:proof4_case1pseudo}
            f^{mp}(z+k) = f^{mp}(z) + \ell^{mp} k, 
        \end{equation}
        for all $z\in U$, $k\in\ZZ$. Hence, if $|\ell|\geq 2$, then $|f^{mp}(z+k)|\to\infty$ as $m\to\infty$ for any $k\neq 0$, contradicting that $z+k\in U$. The same expression gives that for $z\in U$, $f^{mp}(z+k)$ tends, as $m\to\infty$, to $z^*+k$ if $\ell^p=1$, or to the $2$-cycle $\{z^*+k,z^*-k\}$ if $\ell^p=-1$, which is also a contradiction for $k\neq 0$.
        \item If $U$ is a rotation domain, then there is a simple closed curve $\gamma\subset U$ which is invariant under $f^p$. Again by the pseudoperiodicity relation \pref{eq:proof4_case1pseudo}, for any $z\in \gamma$, $|f^{mp}(z+k)|\to\infty$ as $m\to\infty$ if $|\ell|\geq 2$ and $k\neq 0$, in contradiction with $\gamma$ being $f^p$-invariant. Moreover, if $\ell^p=1$ (resp. $\ell^p=-1$), then $\gamma+k\subset U$ is invariant under $f^p$ (resp. $f^{2p}$); a contradiction since $\{\gamma+k\}_k$ is a sequence of non-nested loops in $U$.
        \vspace{0.5mm}
    \end{enumerate}

    \textbf{\textit{Case 2.}} \textit{$0\notin\operatorname{fill}(V)$, $\{0,\infty\}\cap\partial\operatorname{fill}(V)\neq \emptyset$, and $g^{mp}|_V^{}\to 0$ (resp. $\infty$)}: \vspace{-0.5mm}
    
    \noindent Notice that $\infty\notin\operatorname{fill}(V)$ (see \pref{rem:FillFatou}), and $U_k\cap U_j\neq\emptyset$ for all $k\neq j$, due to \pref{lem:FillBaker}. Since $V$ cannot be a Baker domain (as $g$ is of finite-type), and $g^{mp}(w)\to 0$ (resp. $g^{mp}(w)\to \infty$) for all $w\in V$, as $m\to\infty$, we deduce that $0$ (resp. $\infty$) must be a parabolic periodic point of $g$, which lies outside of $\mathcal{E}(\Phi)\cup\Phi^{-1}(\infty)$ by \pref{prop:Regular0inf}. Hence, $V$ shall be a component of the immediate basin of attraction of $0$ (resp. $\infty$).

    \noindent This is only possible if $\ell=\pm 1$; otherwise $0$ (resp. $\infty$) would be a critical point of $g$ due to \pref{prop:projC}. Recall that $g(w)=w^\ell e^{2\pi i \Phi(w)}$, and consider each case separately:
    \begin{enumerate}[label = (\roman*)] 
        \vspace{-1mm} %
        \item If $\ell=1$, then $0$ (resp. $\infty$) is a fixed point of $g$ with $g'(0)=e^{2\pi i \Phi(0)}$ (resp. $g'(\infty)=e^{-2\pi i \Phi(\infty)}$), and $\Phi(0)\in\QQ$ (resp. $\Phi(\infty)\in\QQ$) because the fixed point needs to be of parabolic type. Denote by $\Upsilon$ the attracting axis of the petal contained in $V$, in which the iterates converge to $0$ (resp. $\infty$) tangentially to $\Upsilon$, at an angle, say, $\alpha\in (-\pi,\pi]$ with respect to the positive real axis. Assume, without loss of generality, that $\upsilon:=\{\operatorname{Re}z=\sfrac{\alpha}{2\pi}\}\subset\exp_1^{-1}\Upsilon$ intersects $U$ (\pref{fig:3_projAV} can serve as guidance for this situation).
        
        From the explicit relation \pref{eq:4_RelationIteratesF}, the asymptotic behavior of $f^{mp}$ near the upper (resp. lower) end of $\CC/\ZZ$, is given by \vspace{-0.5mm}
        \begin{equation}
            \label{eq:4_asymptotic21}
            f^{mp}(z)\sim z + mp\hspace{0.1mm}\Phi(0) \qquad \big(\mbox{resp. } f^{mp}(z)\sim z + mp\hspace{0.1mm}\Phi(\infty) \big),
            \vspace{0.25mm}
        \end{equation}
        as $\operatorname{Im}z\to +\infty$ (resp. $-\infty$). Hence, if $\Phi(0)=0$ (resp. $\Phi(\infty)=0$), it follows from \pref{eq:4_asymptotic21} that, for $z\in U_k$, \vspace{-0.5mm}
        \begin{equation*}
            \operatorname{Re} f^{mp}(z)\to \frac{\alpha}{2\pi}+k, \quad \mbox{ as } \quad m\to\infty,
        \end{equation*}
        given that $U_k$ crosses the vertical line $\upsilon+k$, $k\in\ZZ$. In this case, $g'(0)=1$ (resp. $g'(\infty)=1$), and so the petal in $V$ is invariant under $g$, i.e. $p=1$. Then, for any $k$, $\operatorname{Im}f^{mp}|_{U_k}^{}\to +\infty$ (resp. $-\infty$) tangentially to $\nu+k$, as $m\to\infty$, and we obtain that $U_k$ is a Baker domain. This proves case \textit{(2-i)} of the theorem.
        
        \noindent However, if $\Phi(0)\neq 0$ (resp. $\Phi(\infty)\neq 0$), we have from \pref{eq:4_asymptotic21} that $|\operatorname{Re}f^{mp}(z)|\to\infty$, as $m\to\infty$, for $z\in U$ with large imaginary part. Since each $U_k$ is asymptotically contained in the strip $\big\{z: |\operatorname{Re}z-k-\frac{\alpha}{2\pi}|< \frac{1}{2} \big\}$ (as $\upsilon+k\subset U_k$), we conclude that $U$ is a (unbounded) wandering domain of $f$.
        \item If $\ell=-1$, then, as $0$ (resp. $\infty$) is in the domain of definition of $g$, for which $g(0)=\infty$ (resp. $g(\infty)=0$), we have that $\{0,\infty\}$ is a parabolic $2$-cycle, so that $p$ is even, and the cases $g^{mp}|_V^{}\to 0$ and $g^{mp}|_V^{}\to \infty$ can be treated as one and the same. Furthermore, $f\in\mathbf{R}_{-1}$ and $\Phi(0)-\Phi(\infty)\in\QQ$ (see \pref{cor:classRsv}).

        \noindent Now, given that $\ell^p=1$ (as $p$ is even), the relation \pref{eq:4_RelationIteratesF} yields the asymptotics \vspace{-0.5mm}
        \begin{equation}
            \label{eq:4_asymptotics22}
            f^{mp}(z)\sim z \pm\frac{mp}{2}\big(\Phi(\infty)-\Phi(0)\big), \quad \mbox{ as } \quad \operatorname{Im}z\to \pm\infty.
        \end{equation}
        As in the previous subcase, by lifting via $\exp_1$ the axis $\Upsilon$ of the attracting petal in $V$ (which is invariant under $g^p$) and using \pref{eq:4_asymptotics22} to infer the real part of $f^{mp}(z)$ for $z\in U$ (with $|\operatorname{Im}z|$ large enough), as $m\to\infty$, we see that $U$ is a Baker domain if $\Phi(0)=\Phi(\infty)$. This corresponds to case \textit{(2-ii)} of the theorem, since $(g^2)'(0)=1$, i.e. $p=2$. As before, if $\Phi(0)\neq\Phi(\infty)$, $U$ is wandering and unbounded.
        \vspace{0.5mm}
    \end{enumerate}

    \textbf{\textit{Case 3.}} \textit{$0\notin\operatorname{fill}(V)$, $\{0,\infty\}\cap\partial\operatorname{fill}(V)\neq \emptyset$, $g^{mp}|_V^{}\not\to 0$ and $g^{mp}|_V^{}\not\to\infty$}: \vspace{-0.5mm}
    
    \noindent Observe that $\{0,\infty\}\cap\operatorname{fill}(V)=\emptyset$, and $U$ is unbounded. Recall that $V$ cannot be a Baker domain of $g$, so that we have the following possibilities:
    \begin{enumerate}[label = (\roman*)] 
        \vspace{-1mm} 
        \item $V$ is in the immediate basin of attraction of a $p$-periodic point $w^*$, and hence $w^*\in\overline{V}\cap\CC^*$, say $w^*=e^{2\pi i z^*}$ (as both $0$ and $\infty$ are not limit functions of $\{g^{mp}|_V^{}\}_{m\in\NN}$). Note that each $z_k^*:=z^*+k$, $k\in\ZZ$, is a $(p,\sigma_k^{})$-pseudoperiodic point of $f$, where $\{\sigma_k^{}\}_k^{}\subset\ZZ$, and suppose, without loss of generality, that $z^*\in \overline{U}\cap\CC$. 
        
        \noindent Given that $g^{mp}(w)\to w^*$ for all $w\in V$, as $m\to\infty$, the semiconjugacy yields that $f^{mp}(z)\to z_k^*\in \overline{U}_k\cap\CC$ for all $z\in U_k$, and from the relation \pref{eq:derivativeGandF} between the derivatives of $f$ and $g$, we have that, for $k\in\ZZ$, \vspace{-0.5mm}
        \begin{equation}
            \label{eq:4_multipliersUV}
            (f^p)'(z_k^*)=(g^p)'(e^{2\pi i z^*}). \vspace{-0.5mm}
        \end{equation}
        Thus, by inspection of the cases in \pref{prop:IteratesPseudo}, we assert that $z_k^*$ is either $p$-periodic, or $2p$-periodic ($\ell^p$=\hspace{0.4mm}$\shortminus$1 case), or it escapes to $\infty$ under iteration. In other words, $U_k$ can be either a Fatou component of $f$ of period $p$, or $2p$, of the same type as $V$ due to \pref{eq:4_multipliersUV}, or an escaping wandering domain in the latter situation. These belong to the cases \textit{(a)} and \textit{(b)} of the theorem, respectively.
        \item $V$ is a rotation domain, and thus there is a simple closed curve $\Gamma\subset V\cap\CC^*$ which is invariant under $g^p$. Since $0\notin\operatorname{fill}(V)$, i.e. $0\notin \operatorname{int}(\Gamma)$, we have that $\exp_1^{-1}\Gamma$ is a (non-nested) sequence of disjoint loops $\gamma_k:=\gamma+k$, $k\in\ZZ$, chosen such that $\gamma\subset U$. Then, each $\gamma_k$ belongs to the component $U_k\subset\mathcal{F}(f)$, which is of the same connectivity as $V$, and is mapped to $\gamma_{k+\sigma_k^{}}$ by $f^p$, where $\sigma_k^{}\in\ZZ$, due to the semiconjugacy. 
        
        \noindent Therefore, we may consider any $\gamma_k$ as a pseudoperiodic object of minimal type $(p,\sigma_k^{})$, and apply \pref{prop:IteratesPseudo} in analogy to the previous case, so that the same conclusion follows.
        \vspace{0.5mm}
    \end{enumerate}
    
    \textbf{\textit{Case 4.}} \textit{$0\notin\operatorname{fill}(V)$, $\{0,\infty\}\cap\partial\operatorname{fill}(V)= \emptyset$}: \vspace{-0.5mm}
    
    \noindent Following exactly the same arguments as in \textit{Case 3}, we obtain the same options for $U\subset\mathcal{F}(f)$, but here $U$ must be bounded, since $\{0,\infty\}\cap\overline{\operatorname{fill}(V)}=\emptyset$. Hence, these also lead to the cases \textit{(a)} and \textit{(b)} of the theorem.
    
    \noindent Finally, notice that the compilation of possibilities, as indicated in each case, gives the statement.
    $\hfill\square$
\end{proof}

The following is an example which is interesting in its own right, in the spirit of Arnol'd family. It illustrates the case \textit{(1)} of \pref{thm:logFatouJulia}, since we build different kinds of Baker domains (we refer to \cite{Fagella2006} for their classification into three types) by lifting periodic Fatou components $V$ with $0\in\operatorname{fill}(V)$.

\begin{example}[Meromorphic standard family]
    \label{ex:MeroStandard}
    Consider the non-entire function $f$ of the form
    \begin{equation}
        f(z) = z + \alpha - \frac{\beta}{4\pi i}\left( B_a(e^{2\pi i z})-\frac{1}{B_a(e^{2\pi i z})} \right), \quad \mbox{ where } \quad B_a(w) = \frac{w-a}{1-\overline{a}w}
    \end{equation}
    for some $a\in\DD^*$, $\alpha\in[0,1)$ and $\beta>0$ (note that for $a=0$, $f$ degenerates to the entire standard map $f(z)=z+\alpha-\frac{\beta}{2\pi}\sin{2\pi z}$; see \pref{ex:standardWeierstrass}). Its exponential projection $g$ via $\exp_1$ may be written as
    \begin{equation}
        g(w)=w e^{2\pi i R(w)}, \quad \mbox{ where } \quad R(w) = \alpha - \frac{\beta}{4\pi i} \frac{(1-\overline{a}^2)w^2 - 2 w\operatorname{Im}{a}-(1-a^2)}{(1-\overline{a} w)(w-a)}.
    \end{equation}
    Notice that $B_a$ is a finite Blaschke product, and hence it preserves the unit circle $\partial\DD$ and orientation. For all $\theta\in\RR$, $f(\theta) = \theta+\alpha-\frac{\beta}{2\pi}\operatorname{Im}B_a(e^{2\pi i \theta})\in\RR$, which implies that $g(\partial\DD)\subset\partial\DD$. Since $R(0)=\alpha-\frac{\beta}{4\pi i}\frac{1-a^2}{a}\neq\infty$ and $R(\infty)=\overline{R(0)}$ (as $\alpha,\beta\in\RR$), $f\in\mathbf{R}_1$, and \pref{cor:classRsv} asserts that $\mathcal{E}(g)=\{a,1/\overline{a}\}$, and
    \begin{equation}
        \label{eq:MeroStandard_derivatives0inf}
        g'(0) = e^{2\pi i \alpha} e^{\beta (a-1/a)/2}, \qquad g'(\infty) = \overline{g'(0)}.
    \end{equation}
    Here we choose the real parameters $a=\frac{1}{2}$, $\beta=\frac{1}{4}$, and $\alpha$ such that the restriction of $g$ to $\partial\DD$ is a real-analytic diffeomorphism (of topological degree one) with rotation number equal to the golden mean $\rho=\frac{\sqrt{5}-1}{2}$ (numerically, we find that $\alpha \approx 0.61783128$). Since $\rho$ is Diophantine, a theorem due to Herman and Yoccoz (see e.g. \cite[\S 15]{Milnor2006}) yields that $\partial\DD$ is contained in a Herman ring $V\subset\mathcal{F}(g)$. Note from \pref{eq:MeroStandard_derivatives0inf} that $0$ (resp. $\infty$) is an attracting fixed point of $g$, and denote by $V^+$ (resp. $V^-$) its immediate basin of attraction.

    \noindent Now observe that the origin is in the fill of $V$ and $V^\pm$, and hence \pref{lem:FillBaker} shows that their lifts via $\exp_1$ correspond to invariant Baker domains of $f$, say $U$ and $U^\pm$ (see \pref{fig:Blaschke_fHR}). It can be checked that $U$ is a hyperbolic Baker domain containing $\RR$, and $U^+$ (resp. $U^-$) is a doubly parabolic Baker domain containing an upper (resp. lower) half-plane, analogously to \cite[\S5]{Baranski2001} (using the terminology in \cite{Fagella2006}).

    \noindent The finite-type map $g$ has exactly two pairs of symmetric critical points with respect to $\partial\DD$, which belong to the positive real axis. Numerically, each boundary component of the Herman ring $V$ seems to contain one of them ($c^+\approx 0.81135$ and $c^-\approx 1.23251$, whose orbits are dense in $\partial V$ as displayed in \pref{fig:Blaschke_fHR}), in which case $g$ would have no other periodic Fatou components, since $V^\pm$ must enclose the two remaining critical points.
    
    \begin{figure}[H] 
        \centering
        \hfill
        \begin{tikzpicture}[cross/.style={path picture={ 
\draw[white!80]
(path picture bounding box.south) -- (path picture bounding box.north) (path picture bounding box.west) -- (path picture bounding box.east);
}}, cross3/.style={path picture={ 
\draw[white!80]
(path picture bounding box.south west) -- (path picture bounding box.north west) (path picture bounding box.west) -- (path picture bounding box.east);
}}, cross2/.style={path picture={ 
\draw[white!80]
(path picture bounding box.south east) -- (path picture bounding box.north east) (path picture bounding box.west) -- (path picture bounding box.east);
}}]
        \node[anchor=south west,inner sep=0] (image) at (0,0) {\includegraphics[width = 0.602\textwidth]{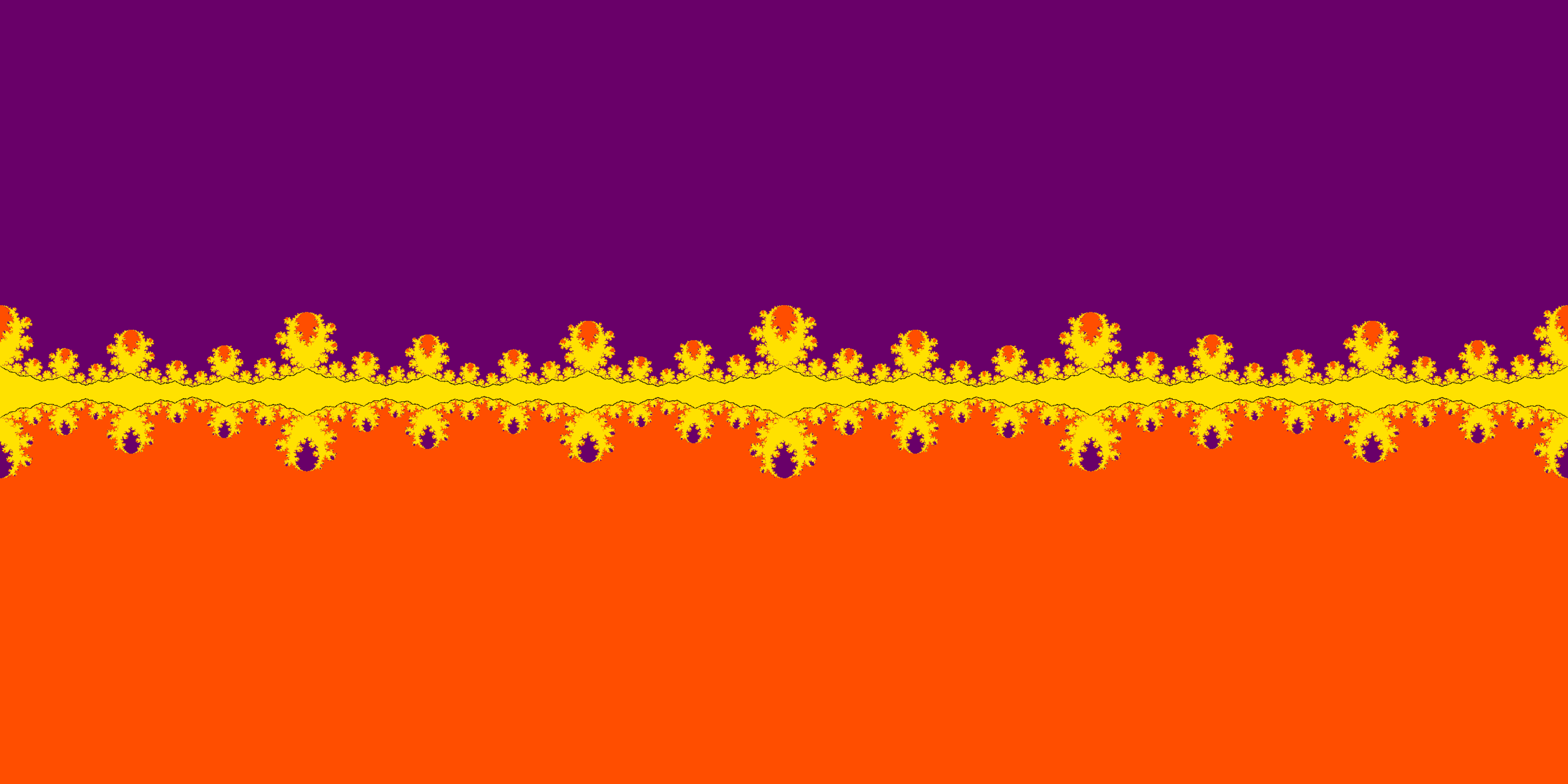}};
        \begin{scope}[x={(image.south east)},y={(image.north west)}]
        \coordinate (a) at ($ ({(0+1)/2},{((-ln(0.5))/(2*pi)+0.5)/1}) $);
        \coordinate (a2) at ($ ({(-0.00155+1+1)/2},{((-ln(0.5))/(2*pi)+0.5)/1}) $);
        \coordinate (a3) at ($ ({(0.0014-1+1)/2},{((-ln(0.5))/(2*pi)+0.5)/1}) $);
        \coordinate (b) at ($ ({(0+1)/2},{((-ln(2))/(2*pi)+0.0008+0.5)/1}) $);
        \coordinate (b2) at ($ ({(-0.00155+1+1)/2},{((-ln(2))/(2*pi)+0.0008+0.5)/1}) $);
        \coordinate (b3) at ($ ({(0.0014-1+1)/2},{((-ln(2))/(2*pi)+0.0008+0.5)/1}) $);
        \draw (0.0435,0.9394) node[draw=none,fill=none, color=white] {\scriptsize $U^{\shortplus}$};
        \draw (0.0435,0.0706) node[draw=none,fill=none, color=white] {\scriptsize $U^{\shortminus}$};
        \draw[line width=0.32pt,cross,color=white!80] (a) circle (1.1pt);
        \draw[line width=0.32pt,cross2,color=white!80] ([shift=(90:1.1pt)]a2) arc (90:270:1.1pt);
        \draw[line width=0.32pt,cross3,color=white!80] ([shift=(-90:1.1pt)]a3) arc (-90:90:1.1pt);
        \draw[line width=0.32pt,cross,color=white!80] (b) circle (1.1pt);
        \draw[line width=0.32pt,cross2,color=white!80] ([shift=(90:1.1pt)]b2) arc (90:270:1.1pt);
        \draw[line width=0.32pt,cross3,color=white!80] ([shift=(-90:1.1pt)]b3) arc (-90:90:1.1pt);
        \draw[line width=0.1pt, color = white!80, dashed] ($({1/4},0)$) -- ($({1/4},1)$);
        \draw[line width=0.1pt, color = white!80, dashed] ($({3/4},0)$) -- ($({3/4},1)$);
        \draw[dashed,line width=0.02pt,color=darkgray!80,-stealth] ($ ({(-1+1)/2},{(0+0.5)/1}) $) to ($ ({(-0.825+1)/2},{(0+0.5)/1}) $);
        \draw[dashed,line width=0.02pt,color=darkgray!80,-stealth] ($ ({(-0.825+1)/2},{(0+0.5)/1}) $) to ($ ({(0.175+1)/2},{(0+0.5)/1}) $);
        \draw[dashed,line width=0.02pt,color=darkgray!80] ($ ({(0.175+1)/2},{(0+0.5)/1}) $) to ($ ({(1+1)/2},{(0+0.5)/1}) $);
        \end{scope}
        \end{tikzpicture} 
        \hfill 
        \begin{tikzpicture}[cross/.style={path picture={ 
\draw[white!80]
(path picture bounding box.south) -- (path picture bounding box.north) (path picture bounding box.west) -- (path picture bounding box.east);
}}]
        \node[anchor=south west,inner sep=0] (image) at (0,0) {\includegraphics[width = 0.302\textwidth]{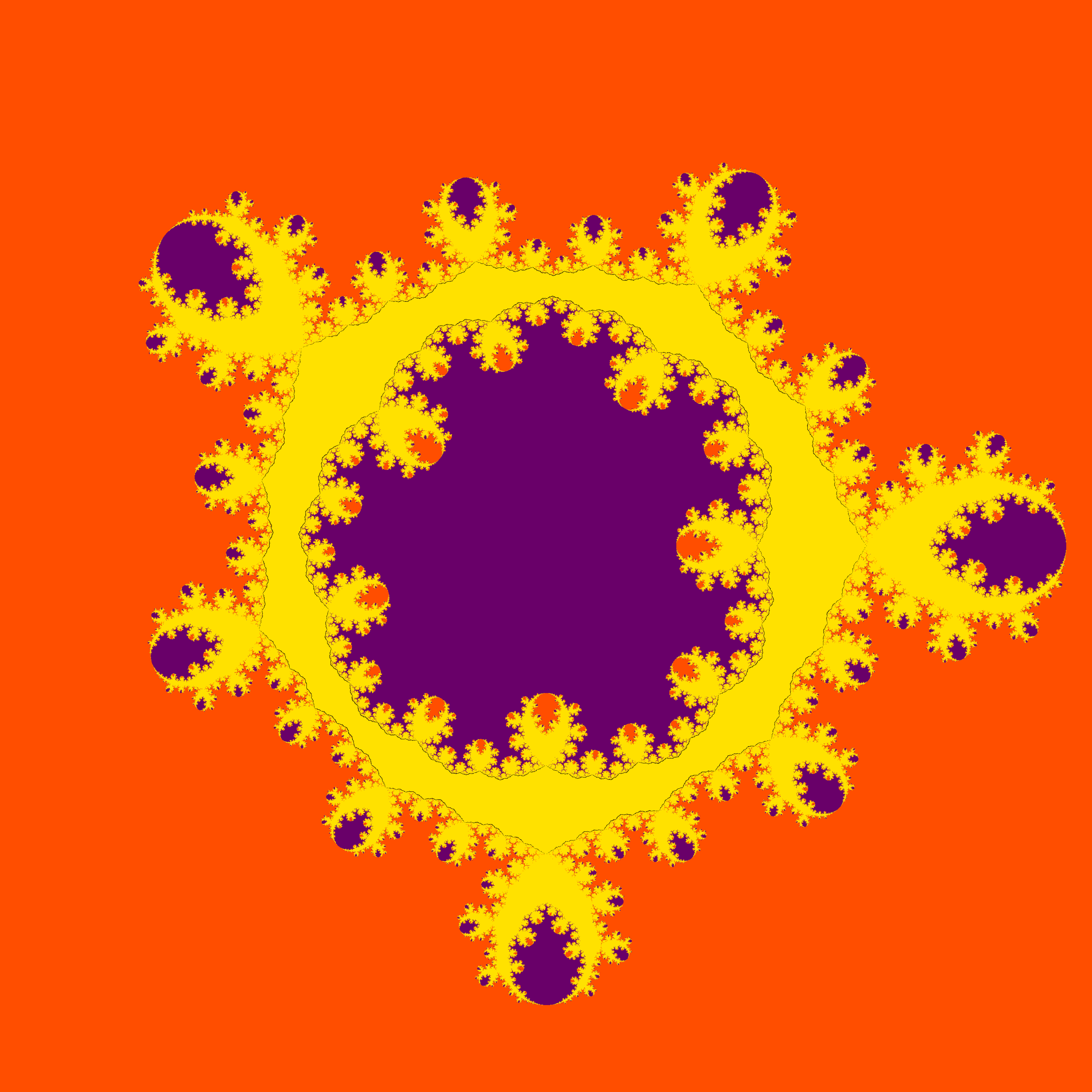}};
        \begin{scope}[x={(image.south east)},y={(image.north west)}]
        \coordinate (a) at ($ ({(0.5-0.001+2.1)/4.2},{(0+2.1)/4.2}) $);
        \coordinate (b) at ($ ({(2-0.0033+2.1)/4.2},{(0+2.1)/4.2}) $);
        \fill[white!80] (0.5,0.5) circle (0.75pt);
        \draw[line width=0.32pt,cross,color=white!80] (a) circle (1.1pt);
        \draw[line width=0.32pt,cross,color=white!80] (b) circle (1.1pt);
        \draw[dashed,line width=0.02pt,color=darkgray!80,-stealth] ([shift=(61.5:0.238095)]0.5,0.5) arc (61.5:423:0.238095);
        \end{scope}
        \end{tikzpicture}
        \hfill \vspace{-0.5mm}
        \caption{\textit{Left (dynamical plane of the meromorphic standard map $f$)}: Three invariant Baker domains in purple, yellow and orange, containing an upper half-plane, the real axis (dashed line) and a lower half-plane, respectively; using the same parameters as in \pref{ex:MeroStandard}. The vertical dashed lines refer to $\{\operatorname{Re}z=\pm\sfrac{1}{2}\}$. Range: $[-1,1]\times[-0.5,0.5]$.        \textit{Right (dynamical plane of its projection $g$)}: The immediate basin of attraction in purple (resp. orange) of the fixed point at $0$ (resp. $\infty$), and the Herman ring (in yellow) containing the unit circle, lift to the Baker domains of $f$ (same colors). The white dot is the origin, and $\oplus$ at $\sfrac{1}{2}$ and $2$ refer to the projections of the poles of $f$. Range: $[-2.1,2.1]^2$. }
        \label{fig:Blaschke_fHR}
    \end{figure}
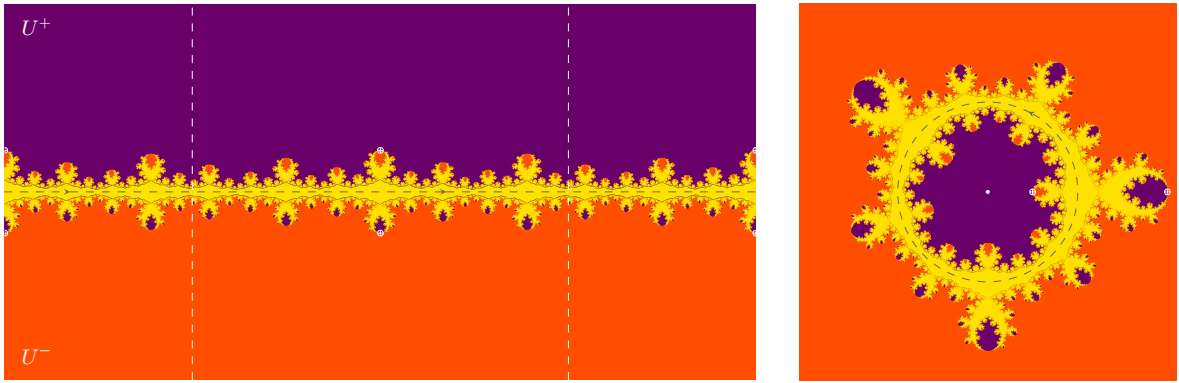
    \vspace{-3.5mm}
    \noindent Although the dynamical plane of $f$ shows a close similarity to the one of the Arnol'd standard map near the real line (compare \pref{fig:Blaschke_fHR} with, e.g., \cite[Fig. 4]{Baranski2001}), here we do not observe the so-called Devaney hairs near the essential singularity \cite{Devaney1986}. Indeed, it looks like $\mathcal{I}(f)\cap \mathcal{J}(f)$ has no unbounded continua (apart from $\partial U$), where $\mathcal{I}(f)$ denotes the \textit{escaping set} of $f$ (i.e. the set of points $z\in\CC$ for which $\{f^n(z)\}_{n\in\NN}$ is defined and $f^n(z)\to\infty$ as $n\to\infty$), in contrast to the situation for many transcendental meromorphic maps (see \cite{Bergweiler2008}). It seems plausible that this sort of baldness property holds for a wide range of functions in the class $\mathbf{R}_\ell$.
\end{example}

Next we give some examples of Baker and wandering domains which arise from the lifting of periodic components $V$ with $0\notin\operatorname{fill}(V)$. These lie in a family of Newton maps of entire functions $F_\beta$ (without roots), introduced in \cite[\S5]{Buff2006} to show that a direct non-logarithmic singularity of $F_\beta^{-1}$ over $0$ does not need to induce Baker domains of the Newton map (often called virtual immediate basins in this context); see e.g. \cite{Bergweiler2008} for the classification (and examples) of singularities of the inverse function of meromorphic maps following Iversen.

\begin{example}[Buff-R\"uckert's family of Newton maps]
    \label{ex:LiftParabolicBR}
    For $\beta>0$, consider the one-parameter family of Newton's methods for the zero-free entire function $F_\beta^{}(z)=\exp{\hspace{-0.2mm}\left(-\frac{z}{\beta}-\frac{1}{2\pi i\beta}e^{2\pi i z}\right)}$, given by \vspace{-1mm}
    \begin{equation}
        N_\beta^{}(z) = z + \frac{\beta}{e^{2\pi i z}+1}.
        \vspace{-0.5mm}
    \end{equation}
    The dynamics of $N_\beta$ modulo $1$ is analyzed in \cite[\S 5]{Buff2006} through its projection $g_\beta^{}$ via $\exp_1$, written as \vspace{-0.5mm}
    \begin{equation}
        g_\beta^{}(w) = w e^{2\pi i R_\beta(w)}, \quad \mbox{ where } \quad R_\beta^{}(w) = \frac{\beta}{w+1}.
        \vspace{-0.5mm}
    \end{equation}
    Notice that $R_\beta^{}(\infty)=0$ and $R_\beta^{}(0)=\beta$, so that $N_\beta^{}\in\mathbf{R}_1$, and $\mathcal{E}(g_\beta^{})=\{-1\}$ (see \pref{cor:classRsv}). Moreover, $g_\beta^{}$ has a parabolic fixed at $\infty$, and another fixed point at $0$, with multipliers $g_\beta'(\infty)=1$ and $g_\beta'(0)=e^{2\pi i \beta}$.

    \noindent If we denote by $V^{\shortminus}$ the immediate basin of attraction of $\infty$, then it follows from \pref{thm:logFatouJulia}, case \textit{(2-i)}, that $N_\beta$ has infinitely many invariant Baker domain $U_k:=U+k$, $k\in\ZZ$, since $0\notin\operatorname{fill}(V^{-})$ (see also \pref{lem:FillBaker}). This agrees with \cite[Thm. 4.1]{Buff2006} as each $U_k$ is induced by a logarithmic singularity of $F_\beta^{-1}$ over $0$ along the asymptotic path $\gamma_k^{-}:=\left\{\operatorname{Re}z=\frac{1}{4}+k\right\}\cap\HH^{-}$, so $U_k$ is of doubly parabolic type, using again the terminology in \cite{Fagella2006}. As an explicit example, for $\beta^*=\sfrac{2i}{\pi}$, as shown in \cite[Ex. 7.3]{Baranski2016}, the Newton map \vspace{-0.5mm}
    \begin{equation*}
        N_{\beta^*}(z) = z + \frac{i}{\pi} + \frac{1}{\pi} \tan{\pi z} \vspace{-0.5mm}
    \end{equation*}
    has a completely invariant Baker domain $U^+$ of infinite degree which contains $\HH^+$ (note that $|g_{\beta^*}'(0)|<1$), and infinitely many Baker domains $U_k$ (of degree $2$), each one containing the vertical half-line $\gamma_k^{-}$, $k\in\ZZ$.

    \noindent Furthermore, observe that if $\beta\in\QQ^*$, i.e. $0$ is a parabolic fixed point of $g_\beta$, then its immediate basin of attraction, say $V^+$, lifts via $\exp_1$ to a chain of unbounded wandering domains of $N_\beta$ due to \pref{thm:logFatouJulia}, as $0\in\partial\operatorname{fill}(V^+)$ and $R_\beta(0)\neq 0$. Nonetheless, if $\beta$ is a Brjuno number, then $g_\beta$ has a Siegel disk about $0$, which lifts to a (simply parabolic) Baker domain. Note that these components of $N_\beta$ are now induced by a direct singularity of $F_\beta^{-1}$ over $0$ (along the real axis) which is not of logarithmic type since, although $0$ is an omitted value of $F_\beta$, the function has infinitely many critical points on $\RR^+$.
\end{example}

Before we turn our attention to projectable Newton maps of entire functions (with roots) in \pref{sec:5_Newton}, we conclude this section with some remarks about the case left aside in \pref{thm:logFatouJulia}, namely when $f$ is $1$-periodic ($\ell=0$).

\subsubsection*{The periodic case}
\label{subsec:PeriodicCase}

Recall that a meromorphic function $f$ which is $1$-periodic (and so projectable via $\exp_1$) may be written as $f(z)=\Phi(e^{2\pi i z})$, where $\Phi$ is meromorphic in $\CC^*$ (see \pref{rem:Fourier}). Moreover, both $0$ and $\infty$ are omitted values of its exponential projection, $g(w)=e^{2\pi i \Phi(w)}$, which is of finite-type if and only if $\exp_1\hspace{-0.2mm}\left( \mathcal{S}(\Phi)\backslash\{\infty\}\right)$ is a finite set (\pref{thm:B_FiniteType}). In this situation, note that $f$ does not need to be of finite-type (see \pref{ex:InfinitelyCinS1}).

\begin{proposition}[$1$-periodic case]
    \label{prop:periodicLifting}
    Let $f$ be a $1$-periodic function, and $g$ its exponential projection via $\exp_1$. If $g$ is a finite-type map, then $f$ has no wandering components nor Baker domains.
\end{proposition}

\begin{proof}
    Consider a Fatou component $U$ of $f$, and let $U_k:=U+k$, $k\in\ZZ$. Given that $f(z+k)=f(z)$ for all $z$, we have that $f(U_k)\subset f(U)$ for all $k$, and $V=\exp_1 U$ is a component of $\mathcal{F}(g)$ due to the relation \pref{eq:expFatouJulia}.
    
    \noindent First, if $U$ is assumed to be wandering, then we see that, for all $n\in\NN$, $k\in\ZZ$, \vspace{-0.75mm}
    \begin{equation}
        \label{eq:4_periodicWD}
        f^{n+1}(U)\cap \left(f^n(U)+k\right) = \emptyset.
        \vspace{-0.75mm}
    \end{equation}
    If this were not true, then there exists $\sigma\in\ZZ$ such that $f^{N+1}(U)\subset f^{N}(U)+\sigma$ for some $N\geq 0$, and so $f^{N+2}(U)\subset f^{N+1}(U)$ by periodicity; a contradiction. Hence, it follows from \pref{eq:4_periodicWD} that $g^n(V)\cap g^m(V)=\emptyset$ for all $n\neq m$, i.e. $V=\exp_1 U$ would be also wandering, which is in contradiction with $g$ being of finite-type.

    \noindent Now, suppose that $U$ is a $p$-periodic Baker domain, and thus $f^{mp+j}(z)\to\infty$ for some $j\in\{0,1,\dots,p-1\}$ and all $z\in U$, as $m\to\infty$. We show that this is not possible by arguing in terms of the type of the periodic component $V=\exp_1 U$ of $\mathcal{F}(g)$ (of the same period $p$, as $f(z+k)=f(z)$ for all $k\in\ZZ$):
    \begin{enumerate}[label = (\roman*)] 
        \vspace{-0.5mm}
        \item If $V$ is the immediate basin of attraction of a $p$-periodic point point $w^*$, then $w^*\in\CC^*$, as $g(0)$ and $g(\infty)$, if defined (see \pref{prop:Regular0inf}), belong to $\CC^*$. Due to the semiconjugacy, as $m\to\infty$, $f^{mp+j}(z)\to z^*$ for some $z^*\in\CC$ such that $e^{2\pi i z^*}=g^j(w^*)$, $j\in\{0,1,\dots,p-1\}$; a contradiction.
        \item If $V$ is a rotation domain, then there exists a $p$-periodic Jordan curve $\Gamma\subset V$, and $g^p$ is conformally conjugate on $V$ to an irrational rotation of a disk (resp. annulus). 
        
        \noindent On the one hand, if $0\in\operatorname{fill}(V)$, then $U$ contains an unbounded curve $\gamma:=\exp_1^{-1}\Gamma$ (which is invariant under translation by $\pm 1$; see \pref{lem:FillBaker}), and $f^p$ is conformally conjugate on $U$ to an horizontal translation of a horizontal strip (resp. half-plane). This is a contradiction, as $f^p(z+1)=f^p(z)$ for all $z$.

        \noindent On the other hand, if $0\notin\operatorname{fill}(V)$, then the lift of $\Gamma$ via $\exp_1$ is a collection of disjoint loops $\gamma_k:=\gamma+k$, $k\in\ZZ$, say $\gamma\subset U$, and every $U_k$ is a disk (resp. annulus). Hence, by periodicity, $f^p$ maps every $\gamma_k$ to $\gamma_\sigma$ for some $\sigma\in\ZZ$, that is, $\gamma_\sigma$ is $f^p$-invariant, which is in contradiction with $U$ being escaping. 
        $\hfill\square$
    \end{enumerate}
\end{proof}

It can be derived that for a $p$-periodic component $V$ of $\mathcal{F}(g)$, any connected component $U$ of $\exp_1^{-1}(V\hspace{0.1mm}\cap\hspace{0.1mm}\CC^*)$ is a Fatou component of $f$ of the same type as $V$, which is either $p$-periodic and invariant under translation by $\pm 1$ if $0\in\operatorname{fill}(V)$ (see \pref{lem:FillBaker}), or eventually $p$-periodic otherwise, since $f(U+k)\subset f(U)$ for all $k\in\ZZ$. Both cases occur already for the sine family as shown in the following example by lifting, respectively, a doubly or simply connected attracting basin. The connectivity of periodic Fatou components in Bolsch's class $\mathbf{K}$ is known to be $1$, $2$ or $\infty$; however, note that in contrast to meromorphic functions in $\CC$, an invariant doubly-connected Fatou component of $g\in\mathbf{K}$ does not need to be a Herman ring (see more details in \cite{Bolsch1999}).

\begin{example}[Sine family]
    \label{ex:LiftAttractingA}
    Consider the entire function $f_\beta(z)=\frac{\beta}{2\pi}\sin{2\pi z}$, $\beta\in\RR^*$, which is $1$-periodic. Its exponential projection $g_\beta$ (via $\exp_1$) is given by \vspace{-1mm}
    \begin{equation}
        g_\beta^{}(w)=e^{2\pi i \Phi_\beta(w)}, \quad \mbox{ where } \quad \Phi_\beta(w) = \frac{\beta}{4\pi i}\Big(w-\frac{1}{w}\Big).
    \end{equation}
    Note that $g_\beta$ is a transcendental self-map of $\CC^*$ (with $\Phi_\beta(0)=\Phi_\beta(\infty)=\infty$; see \pref{rem:EntireCase}), for which the unit circle is invariant, and the fixed point at $1$ has multiplier $g_\beta'(1)=\beta$. The critical points of $g_\beta$ are $\pm i$.

    \noindent On the hand, it is known that, for $0<\beta<1$, the Fatou set of $g_\beta$ consists of a single doubly-connected component $V$ containing the unit circle and $\pm i$, with $\{0,\infty\}\subset\partial V$ (see \cite[Thm. 2]{Baker1987}). Hence, as $0\in\operatorname{fill}(V)$, the only Fatou component of $f_\beta$ is the (unbounded) basin of the attracting fixed point at $0$, due to \pref{eq:logFatouJulia}.

    \noindent On the other hand, if we choose $\beta=-\frac{\pi}{2}$, then the fixed point of $g_\beta$ at $1$ becomes repelling, and $\{+i,-i\}$ is a superattracting $2$-cycle. Then, the immediate basin $V^\pm$ of attraction of $\pm i$ lifts via $\exp_1$ to an infinite collection of disjoint Fatou components $U_k^\pm:=U^\pm+k$, $k\in\ZZ$, chosen such that $\pm\frac{1}{4}\in U^\pm$. As $f_\beta\big(\pm \frac{1}{4}\big)=\mp \frac{1}{4}$, we obtain that $\{U^+,U^-\}$ is a $2$-periodic cycle of immediate attracting basins, and $f_\beta(U_k^\pm)\subset U^{\mp}$ for all $k$.
\end{example}
%
\section{Projectable Newton maps and Proof of Thm. \ref{thm:C_AtlasBWD}}
\label{sec:5_Newton}

The meromorphic function in \pref{ex:LiftParabolicBR} (introduced by Buff and R\"uckert \cite{Buff2006}) is a first instance of a Newton's root-finding method (in class $\mathbf{R}_1$) with wandering domains obtained by the lifting method; however, as the Newton map of a zero-free function, there are no fixed points to search for. In this section, building on \pref{cor:WDviaPseudoperiodicity} and \pref{thm:logFatouJulia}, we present a broad class of explicit (projectable) Newton's methods with fixed points whose attracting invariant basins do often coexist with Baker and escaping wandering domains.

For this purpose, we first characterize those Newton's methods (with fixed points) in the class $\mathbf{R}_\ell$ (see \pref{def:classR}). We claim in \pref{thm:C_AtlasBWD} that they are the Newton maps $N_F:\CC\to\hCC$ of the entire functions \vspace{-0.4mm}
\begin{equation}
    \label{eq:5intro_FzNewtonR1}
    F(z)= e^{(\Lambda+2\pi i m_0)z} \hspace{0.1mm} P(e^{2\pi i z}) \hspace{0.3mm} e^{Q(e^{2\pi i z})+\widetilde{Q}(e^{-2\pi i z})},
\end{equation}
where $\Lambda\in\CC$, $m_0\in\ZZ$, $P$, $Q$ and $\widetilde{Q}$ are polynomials with $P(0)\neq 0$, and $P$ has zeros in $\CC^*$. Additionally, $\Lambda \neq -2\pi i (m_0+\deg{P})$ if $Q$ is constant, and $\Lambda \neq -2\pi i m_0$ if $\widetilde{Q}$ is constant. 

If we consider that the polynomial $P$ has $M$ distinct roots $\{A_j\}_{j=1}^M$ (all in $\CC^*$), each of multiplicity $m_j$, and $\alpha\in\CC^*$ is its leading coefficient, then we may write \vspace{-0.5mm}
\begin{equation}
    \label{eq:5intro_PolynP}
    P(w) = \alpha\prod\limits_{j=1}^M (w-A_j)^{m_j}, \quad \mbox{ and } \quad \deg{P}=\sum_{j=1}^M m_j. \vspace{-0.5mm}
\end{equation}
Moreover, if we define $\widetilde{P}(w) := (w-A_1)\cdots (w-A_M)$, it follows from the product rule that the logarithmic derivative of the polynomial $P$ in \pref{eq:5intro_PolynP} gives that \vspace{-0.5mm}
\begin{equation}
    \frac{P'(w)}{P(w)} = \sum_{j=1}^M \frac{m_j}{w-A_j}, \quad \mbox{ and } \quad \frac{P'(w) \widetilde{P}(w)}{P(w)} = \sum_{j=1}^M m_j \prod_{k\neq j} (w-A_k). \vspace{-0.5mm}
\end{equation}

Denote by $p$, $q$, $\tilde{p}$, $\tilde{q}$ the degrees of $P$, $Q$, $\widetilde{P}$, $\widetilde{Q}$, respectively. Hence, the Newton map of the entire function $F$ in \pref{eq:5intro_FzNewtonR1} takes the form $N_\Lambda(z) = z + R_\Lambda(e^{2\pi i z})$, where $R_\Lambda$ is the quotient of two coprime polynomials:
\begin{equation}
    \label{eq:NewtonProjR1PQ}
    R_\Lambda(w) = -\frac{ w^{\tilde{q}} \widetilde{P}(w)}{(\Lambda+2\pi i m_0)\hspace{0.1mm} w^{\tilde{q}} \widetilde{P}(w) + 2\pi i \hspace{0.1mm}w^{\tilde{q}+1}\Big(\frac{P'(w) \widetilde{P}(w)}{P(w)}+\widetilde{P}(w) Q'(w)\Big) - 2\pi i \hspace{0.1mm} \widetilde{P}(w) \hspace{0.1mm} w^{\tilde{q}-1} \widetilde{Q}'(\sfrac{1}{w})}.
\end{equation}
Expanding the polynomials in \pref{eq:NewtonProjR1PQ} in powers of $w$, we obtain that the degree of the numerator of $R_\Lambda$ is $\tilde{p}+\tilde{q}$, while the degree of its denominator turns out to be $\tilde{p}+q+\tilde{q}$; note that $\deg\big(\sfrac{P'(w)\widetilde{P}(w)}{P(w)}\big)=\tilde{p}-1$, and $\deg\big( w^{\tilde{q}-1} \widetilde{Q}'(\sfrac{1}{w}) \big)=\tilde{q}-1$. Therefore, $\deg{R_\Lambda}=\tilde{p}+q+\tilde{q}$, and the following lemma is straightforward.
\begin{lemma}[Rational map $R_\Lambda$]
    \label{lem:5_rationalRnewton}
    Using the notation in \pref{thm:C_AtlasBWD}, consider the rational map $R_\Lambda$ in \pref{eq:NewtonProjR1PQ} with $\Lambda\in\CC$, $m_0\in\ZZ$, and let $p:=\deg{P}$, $q:=\deg{Q}$, $\tilde{q}:=\deg{\widetilde{Q}}$. Then, if $q>0$ (resp. $\tilde{q}>0$), we have that 
    \begin{equation}
        \label{eq:RationalC_greater}
        R_\Lambda(\infty)=0 \quad \big(\hspace{0.4mm}\mbox{resp. } \ R_\Lambda(0)=0 \hspace{0.4mm}\big). 
    \end{equation}
    In the case that $q=0$ (resp. $\tilde{q}=0$), we have that 
    \begin{equation}
        \label{eq:RationalC_equal}
        R_\Lambda(\infty)=-\frac{1}{\Lambda+2\pi i (m_0+p)} \quad \bigg(\mbox{resp. } \ R_\Lambda(0)=-\frac{1}{\Lambda+2\pi i m_0}\bigg). \quad\quad 
    \end{equation}
\end{lemma}

This is going to be used to prove \pref{thm:C_AtlasBWD} as follows, particularly to establish the conditions on $\Lambda$.

\begin{customproof}{of \pref{thm:C_AtlasBWD}}
    \textup{Suppose that the Newton map $N_F$ is in the class $\mathbf{R}_\ell$, that is, $N_F(z)=\ell z + R(e^{2\pi i z})$ for some $\ell\in\ZZ$ and a non-constant rational map $R$ with $\{0,\infty\}\cap R^{-1}(\infty) =\emptyset$ (see \pref{def:classR}), and let us find an expression for the entire function $F$. Observe that such a Newton map is transcendental, and so is $F$. The exponential projection of $N_F$ (via $\exp_1$) is given by $g(w)=w^\ell e^{2\pi i R(w)}$.}

    \noindent \textup{We shall first show that $\ell=1$. Let $\xi$ be a (attracting) fixed point of $N_F$, and $\mathcal{A}^*(\xi)$ its immediate basin. It is known that there is an invariant access to $\infty$ from $\mathcal{A}^*(\xi)$, represented by a curve $\gamma:[0,\infty)\to \mathcal{A}^*(\xi)$ that lands at $\infty$ with $\gamma(0)=\xi$ and, for $t\geq 1$, $N_F(\gamma(t))=\gamma(t-1)$ (see \cite{Mayer2006}).
    Note that $N_F\notin\mathbf{R}_0$, as $\infty\in\mathcal{AV}(f)$. Furthermore, if $\operatorname{Im}\gamma(t)\to +\infty$ (resp. $-\infty$), then the explicit form of the Newton map gives that, as $t\to\infty$,
    \begin{equation}
        \label{eq:proofThmC_1}
        \operatorname{Im}N_F(\gamma(t))\sim \ell\operatorname{Im}\gamma(t)+\operatorname{Im}R(0) \quad \big(\hspace{0.3mm}\mbox{resp. } \operatorname{Im}N_F(\gamma(t))\sim \ell\operatorname{Im}\gamma(t)+\operatorname{Im}R(\infty) \hspace{0.3mm}\big).
    \end{equation}
    Since $R$ is a rational map for which $0$ and $\infty$ are not poles, we deduce from \pref{eq:proofThmC_1} that, for large enough $t$, the iterate of $\gamma(t)$ moves further away from $\xi$ if $\ell\geq 2$, or lies outside $\gamma$ if $\ell\leq -1$, contradicting the fact that $\gamma$ is a $N_F$-invariant curve whose points converge to $\xi$ under iteration. Analogously, the same contradiction can be obtained if $\operatorname{Re}\gamma(t)$ were unbounded as $t\to\infty$. This proves that $\ell=1$.}

    \noindent \textup{It follows that $N_F(z)=z+R(e^{2\pi i z})$, that is, $\sfrac{F}{F'}=-R\circ\exp_1$. Since $N_F(z+1)=N_F(z)+1$, this leads to the difference equation $\frac{F(z+1)}{F'(z+1)}=\frac{F(z)}{F'(z)}$, which means that the logarithmic derivative operator $\sfrac{F'}{F}$ is $1$-periodic. By direct integration,
    \begin{equation*}
        F(z+1) = e^{\Lambda} F(z)
    \end{equation*}
    for some constant $\Lambda\in\CC$. Hence, $F(z)e^{-\Lambda z}$ is $1$-periodic, that is, $F(z)=e^{\Lambda z}\psi(z)$ for a periodic entire function $\psi$ of period $1$. Note that $\psi$ is non-constant since $N_F$ is transcendental (see e.g. \cite[Prop. 2.11]{Ruckert2007}).}
    
    \noindent \textup{In fact, $\psi(z)=\Psi(e^{2\pi i z})$, where $\Psi$ is an analytic function in $\CC^*$ for which $0$ or $\infty$ may be essential singularities or poles (see Remarks \ref{rem:Fourier} and \ref{rem:EntireCase}).
    Thus, the rational map $R$ is given by 
    \begin{equation}
        \label{eq:5_thmC_ratPhi}
        R(w)=-\frac{\Psi(w)}{\Lambda\Psi(w)+2\pi i w \Psi'(w)},
    \end{equation}
    with finitely many distinct roots in $\CC^*$, say $\{A_j\}_{j=1}^M$, each of multiplicity $m_j\geq 1$. These are exactly the roots of $\Psi$ which are different from $0$ and $\infty$, and therefore, we may write \vspace{-1.5mm}
    \begin{equation*}
        \Psi(w)=w^{m_0} P(w) e^{Q(w)+\widetilde{Q}(\sfrac{1}{w})}, \quad \mbox{ where } \quad P(w) = \alpha\prod_{j=1}^M (w-A_j)^{m_j},
        \vspace{-1.5mm}
    \end{equation*} 
    for some $m_0\in\ZZ$, $\alpha\in\CC^*$, and entire functions $Q$, $\widetilde{Q}$. Since $F(z)=e^{\Lambda z}\Psi(e^{2\pi i z})$ has zeros in $\CC$ by assumption, so does $P(w)$ in $\CC^*$. A direct computation shows that $R(w)$ in \pref{eq:5_thmC_ratPhi} takes the explicit form of the rational map \pref{eq:NewtonProjR1PQ}, where $\widetilde{P}(w) = (w-A_1)\cdots (w-A_M)$.
    Then, both $Q$ and $\widetilde{Q}$ must be polynomials, or otherwise $R$ would be transcendental as either $\lim\limits_{w\to 0} R(w)$ or $\lim\limits_{w\to \infty} R(w)$ would not be well-defined.}
    
    \noindent \textup{Finally, as $N_F\in\mathbf{R}_1$, we need some additional conditions on $\Lambda$ so that $R(\infty)\neq\infty$ and $R(0)\neq\infty$. If we denote by $p$, $q$, $\tilde{q}$ the degrees of the polynomials $P$, $Q$, $\widetilde{Q}$, then by \pref{lem:5_rationalRnewton} we must require that $\Lambda\neq -2\pi i (m_0+p)$ if $q=0$, and $\Lambda\neq -2\pi i m_0$ if $\tilde{q}=0$.}

    \noindent \textup{For the reverse direction, suppose that $P$, $Q$, $\widetilde{Q}$ are polynomials (with $P(0)\neq 0$ and $P^{-1}(0)\cap\CC^*\neq\emptyset$), $m_0\in\ZZ$, $\Lambda\in\CC$ satisfies the requirements stated in \pref{thm:C_AtlasBWD}, and $F$ is the entire function given by \pref{eq:ThmC_ExpressionF}. Then, by a straight computation, $\sfrac{-F(z)}{F'(z)}$ takes the form $R(e^{2\pi i z})$ with $R$ as the rational map in \pref{eq:NewtonProjR1PQ}.} 
    
    \noindent \textup{In this situation, to see that the Newton map of $F$ is in class $\mathbf{R}_\ell$ (note here that $\ell=1$), we shall verify that $R(\infty)\neq\infty$ and $R(0)\neq\infty$. It follows from \pref{lem:5_rationalRnewton}, together with the conditions on $\Lambda$, that $R(\infty)\in\CC^*$ if $Q$ is constant, while $R(\infty)=0$ otherwise, and $R(0)\in\CC^*$ if $\widetilde{Q}$ is constant, or else $R(0)=0$.
    $\hfill\square$
    }
\end{customproof}

This provides good candidates of projectable Newton maps to showcase the existence of wandering domains via the lifting procedure developed in \pref{sec:4_Lifting}. But first let us recall that if $N_F:=\operatorname{Id}-\frac{F}{F'}$ is the Newton map of an entire function $F$, then 
\begin{equation}
    N_F'(z) = \frac{F(z) F''(z)}{\big(F'(z)\big)^2},
\end{equation}
that is, the zeros of $F''$ which are neither roots of $F'$ nor $F$, are the \textit{free critical points} of $N_F$ (in the sense that they are not fixed points in general). \pref{tab:NewtonRoots} summarizes well-known relations between dynamically relevant points of $F$ and $N_F$, which can be easily extended to its projection $g$ via $\exp_1$ as an outcome of \pref{sec:3_ExpProjection}.

\begin{table}[h!]
    \centering
    \resizebox{0.82\textwidth}{!}{%
    \begin{tabular}{||c|c|c||} 
    \hline
    $F(z)$ & $N_{F}(z)$ & $g(w)$ \\ [0.5ex] 
    \hline\hline 
    \textit{Zero} of multiplicity $m$ &  \textit{Attracting fixed point} &  \textit{Attracting fixed point} \\ [0.5ex]
    $F(a)=0$ &  $N_F'(a)=\frac{m-1}{m}$ & $A=e^{2\pi i a}$, \ $g'(A)=\frac{m-1}{m}$ \\ [0.7ex]
    \hline
    \textit{Critical point} (not root of $F$) &  \textit{Pole} & \textit{Essential singularity} \\ [0.5ex]
    $F'(b)=0$ &  $N_F(b)=\infty$ & $B=e^{2\pi i b}$  \\ [0.7ex]
    \hline
    \textit{Inflection point} (not root of $F$, $F'$) &  \textit{Free critical point} & \textit{Free critical point} \\ [0.5ex]
    $F''(c)=0$  &  $N_F'(c)=0$ & $C=e^{2\pi i c}$, \ $g'(C)=0$ \\ [0.7ex]
    \hline
    \end{tabular}}\vspace{-0.5mm}
    \caption{Character of the roots of $F$ (of multiplicity $m\geq 1$), $F'$ and $F''$ as points of the Newton map $N_F(z)=z-\frac{F(z)}{F'(z)}$, and in the case that $N_F$ is projectable, of its exponential projection $g(w)$ via $w=e^{2\pi i z}$.}
    \label{tab:NewtonRoots}
\end{table}

It is well-known that all fixed points of the Newton map $N_F$ are attracting and roots of $F$, and $\infty$ is either an essential singularity of $N_F$, or a parabolic or repelling fixed point if $N_F$ is rational (see e.g. \cite{Ruckert2007}). In \cite{Mayer2006} it was shown that the immediate basin of attraction for every root of $F$ is simply-connected and unbounded, extending Przytycki's result \cite{Przytycki1989} in the rational case. Furthermore, in \cite{Baranski2014} it was proven that every Fatou component of $N_F$ is simply-connected, based on the absence of weakly repelling fixed points in the transcendental setting, which generalizes Shishikura's result \cite{Shishikura2009} in the rational case.

\subsection{Baker and wandering domains for Newton maps}
\label{sec:5.1_BDWDnewton}

In this context, our results in relation to the lifting method (see \pref{sec:4_Lifting}) can be applied to find Baker and wandering domains for one-parameter families of projectable Newton maps. This is displayed by the following corollary for Newton's methods in the class $\mathbf{R}_1$, which are specified by \pref{thm:C_AtlasBWD}. We consider, for convenience, the parameter $\lambda:=\Lambda+\pi i (2m_0+\deg{P})$, where $m_0\in\ZZ$ and $P(w)$ is a polynomial whose zeros are the projection of the fixed points of the Newton map of $F(z)$ under consideration, i.e. $P^{-1}(0)=\exp_1\hspace{-0.3mm}\big(F^{-1}(0)\big)$.

\begin{corollary}[Baker and wandering domains for Newton's methods in class $\mathbf{R}_1$]
\label{cor:C_BakerWandering}
Using the notation in \pref{thm:C_AtlasBWD} with $\Lambda=\lambda-\pi i (2m_0+\deg{P})$, let $N_\lambda$ be a Newton map in class $\mathbf{R}_1$ with fixed points, and $g_\lambda$ its exponential projection. Then, $N_\lambda$ has infinitely many attracting invariant basins, and a simply-connected 
    \begin{enumerate}[label = (\roman*)]
        \vspace{-1mm}
        \item invariant Baker domain for any $\lambda$ such that $|\operatorname{Im}\lambda|>\pi \deg{P}$ if both $Q$ and $\widetilde{Q}$ are constant, as well as for all $\lambda$ if $Q$ or $\widetilde{Q}$ is non-constant;
        \item escaping wandering domain for any $\lambda$ such that the projection via $\exp_1$ of a $(p,\sigma)$-pseudoperiodic point of $N_\lambda$ with $p\in\NN^*$ and $\sigma\in\ZZ^*$, as a $p$-periodic point of $g_\lambda$, is attracting, parabolic or of Siegel type.
    \end{enumerate}
\end{corollary}

\begin{proof}
    \noindent Since $N_\lambda$ has fixed points, as a Newton map, these must be attracting. Furthermore, there should be infinitely many of them because $N_\lambda$ is $1$-pseudoperiodic, i.e. $N_\lambda(z+k)=N_\lambda(z)+k$ for all $k\in\ZZ$. 
    
    \noindent The projection $g_\lambda$ is given by $g_\lambda(w)=w e^{2\pi i R_\lambda(w)}$, where, by \pref{thm:C_AtlasBWD}, $R_\lambda$ is a rational map of the form
    \begin{equation}
        \label{eq:5_corBDWD}
        R_\lambda(w)=-\frac{\Psi(w)}{\big(\lambda -\pi i (2m_0+\deg{P})\big)\Psi(w)+2\pi i w \Psi'(w)}, \quad \mbox{ with } \quad \Psi(w)=w^{m_0} P(w) e^{Q(w)+\widetilde{Q}(\sfrac{1}{w})},
    \end{equation}
    $m_0\in\ZZ$, and $P$, $Q$, $\widetilde{Q}$ as polynomials ($P$ has all zeros in $\CC^*$). Additionally, $\lambda\in\CC$ satisfies that $\lambda\neq -\pi i \deg{P}$ (resp. $\lambda\neq \pi i \deg{P}$) in the case that $Q$ (resp. $\widetilde{Q}$) is constant. Given that $N_\lambda\in\mathbf{R}_1$, we have that, as stated in \pref{cor:classRsv}, both $\infty$ and $0$ are fixed points of $g_\lambda$ with
    \begin{equation}
        \label{eq:CorC_dg0_infty}
        |g_\lambda'(\infty)|=e^{2\pi \operatorname{Im}{R_\lambda(\infty)}}, \quad \mbox{ and } \quad |g_\lambda'(0)|=e^{-2\pi \operatorname{Im}{R_\lambda(0)}}.
    \end{equation}
    \noindent To prove \textit{(i)}, we observe first that if both $Q$ and $\widetilde{Q}$ are constant, then from \pref{eq:CorC_dg0_infty} and the explicit expressions of $R_\lambda(\infty)$ and $R_\lambda(0)$ in \pref{eq:RationalC_equal}, with $\Lambda=\lambda-\pi i(2m_0+\deg{P})$, we obtain that
    \begin{equation}
        \label{eq:CorC_dg0_infty_QQconst}
        |g_\lambda'(\infty)|< 1 \iff \operatorname{Im}\lambda < - \pi\deg{P}, \qquad |g_\lambda'(0)|< 1 \iff \operatorname{Im}\lambda > \pi\deg{P}.
    \end{equation}
    Hence, for any $\lambda$ satisfying the first (resp. second) condition in \pref{eq:CorC_dg0_infty_QQconst}, we deduce that $\infty$ (resp. $0$) is an attracting fixed point of $g_\lambda$. In any case, if we denote by $V$ the immediate basin of attraction of $\infty$ (resp. $0$), then $V$ lifts via $\exp_1$ to an invariant Baker domain of $N_\Lambda$ by \pref{thm:logFatouJulia}, case \textit{(1)}.
    
    \noindent In the case that $Q$ (resp. $\widetilde{Q}$) is non-constant, it follows from \pref{lem:5_rationalRnewton} that $R_\lambda(\infty)=0$ (resp. $R_\lambda(0)=0$). Then, from \pref{eq:CorC_dg0_infty} the fixed point of $g_\lambda$ at $\infty$ (resp. $0$) is parabolic with multiplier $1$, and so there is at least one attracting petal attached to it, contained in an invariant component $V$ of $\mathcal{F}(g)$. The case \textit{(2-i)} of \pref{thm:logFatouJulia} gives that any connected component of $\exp_1^{-1}V$ is an invariant Baker domain of $N_\lambda$ for all $\lambda$.
    
    \noindent To see \textit{(ii)}, assume $z_\lambda^*$ to be a pseudoperiodic point of $N_\lambda$ of minimal type $(p,\sigma)$, with $p\geq 1$ and $\sigma\in\ZZ^*$, i.e.
    \begin{equation*}
        N_\lambda^{p}(z_\lambda^*)=z_\lambda^*+\sigma.
    \end{equation*}
    On the one hand, if $\lambda$ is such that $w_\lambda^*:=e^{2\pi i z_\lambda^*}$, as a $p$-periodic point of $g_\lambda$ (see \pref{lem:characPseudoPoint}), is attracting or of Siegel type, then $z_\lambda^*\in\mathcal{F}(N_\lambda)$, and the conclusion follows from \pref{cor:WDviaPseudoperiodicity}. 
    
    \noindent On the other hand, when $\lambda$ is such that $w_\lambda^*$ is parabolic, any component $V$ of its immediate basin of attraction lifts via $\exp_1$ to a chain of wandering domain $\{U+k\}_{k\in\ZZ}$ of $N_\lambda$ due to \pref{thm:logFatouJulia}; note that $0\notin \mathrm{fill}(V)$, $g_\lambda^{mp}(w)\to w_\lambda^*\in\CC^*$ for all $w\in V$, as $m\to\infty$, and $N_\lambda^{mp}(U)\subset U+m\sigma$ by \pref{prop:IteratesPseudo} ($\ell=1$ case).
    $\hfill\square$
\end{proof}

Observe that if $Q$ (resp. $\widetilde{Q}$) is non-constant, then a component $V$ of the immediate parabolic basin for $g_\lambda$ of $\infty$ (resp. $0$) lifts to infinitely many distinct Baker domains of the Newton map $N_\lambda$, as $0\notin\mathrm{fill}(V)$. 
Moreover, we point out that the internal dynamics of a wandering domain $U\subset\mathcal{F}(N_\lambda)$ built by this lifting procedure, relates to the type of the periodic component $V=\exp_1 U$ of $g_\lambda$. Given an appropriate pseudoperiodic (but non-periodic) point of $N_\lambda$, say $z^*\in\overline{U}$, if $V$ is an attracting (resp. parabolic) basin, the iterates of all points in $U$ accumulate about the forward orbit of $z^*$, which lies in the interior (resp. boundary) of $U$. However, if $V$ is a Siegel disk, points in $U$ travel along $\{N_\lambda^n(U)\}_{n\in\NN}^{}$ in a rotation-like behavior around the orbit of $z^*$ (as a moving center); see \cite{Benini2022} for a classification (and examples) of simply-connected wandering domains.

This result showcase the possible coexistence of wandering domains and attracting invariant basins for Newton's methods, as we were looking for. To make it explicit, in the following sections we inspect the dynamical and parameter planes for the simplest family of Newton maps (with fixed points) in this class.

\subsection{Pseudotrigonometric family: dynamical planes}
\label{sec:5.2_PseudoTrig}

It follows from \pref{thm:C_AtlasBWD} that any Newton map in class $\mathbf{R}_1$ is the Newton's method of a transcendental entire function
\begin{equation}
    \label{eq:5_FzNewtonR1}
    F(z)= e^{(\lambda-\pi i \deg{P}) z} \hspace{0.1mm} P(e^{2\pi i z}) \hspace{0.3mm} e^{Q(e^{2\pi i z})+\widetilde{Q}(e^{-2\pi i z})},
\end{equation}
considering $\lambda:=\Lambda+\pi i (2m_0+\deg{P})$ as in \pref{sec:5.1_BDWDnewton}, where $m_0\in\ZZ$ and $P$, $Q$, $\widetilde{Q}$ are polynomials. 

These Newton maps are of the form $N_\lambda(z) = z + R_\lambda(e^{2\pi i z})$, with $R_\lambda$ as the rational map in \pref{eq:NewtonProjR1PQ}, where $\widetilde{P}$ is a monic polynomial whose roots are simple and exactly those of $P$. It is easy to verify that $R_\lambda(\infty)\neq\infty$ and $R_\lambda(0)\neq\infty$ from \pref{lem:5_rationalRnewton}, assuming that $\lambda \neq -\pi i \deg{P}$ if $\deg{Q}=0$, and $\Lambda \neq \pi i \deg{P}$ if $\deg{\widetilde{Q}}=0$. Furthermore, counting with multiplicity, \vspace{0.25mm}
\begin{equation}
    \label{eq:5_PolesZerosRlambda}
    \#\big(R_\lambda^{-1}(0)\cap\CC^*)=\deg{\widetilde{P}}, \quad \mbox{ and } \quad \#\hspace{0.2mm} R_\lambda^{-1}(\infty) = \deg{\widetilde{P}}+\deg{Q}+\deg{\widetilde{Q}}, \vspace{0.25mm}
\end{equation} 
In general, such a $N_\lambda\in\mathbf{R}_1$ may have multiple poles, as it is the case, for example, of the Newton map of $\exp\big(z-\sfrac{e^{2\pi i z}}{\pi i} + \sfrac{e^{4\pi i z}}{4\pi i} \big)$ at all integers. In the following we give representatives for the simplest case, that is, when the Newton's method $N_\lambda$ has a unique pole of multiplicity one in the period strip of $\exp_1$, i.e. $\#\hspace{0.2mm} R_\lambda^{-1}(\infty) = 1$, and in particular, its exponential projection $g_\lambda$ is a transcendental meromorphic map (in the sense that $\#\hspace{0.2mm}\mathcal{E}(g_\lambda) = 1$). Recall that two entire functions have the same Newton map $N_F$ if and only if they differ by a multiplicative constant (see \cite[Prop. 2.8]{Ruckert2008}); indeed $F(z)=K \exp{\hspace{-0.2mm}\left(\int_0^z \frac{du}{u-N_F(u)}\right)}$ for $K\in\CC^*$.

\begin{proposition}[Newton's methods in class $\mathbf{R}_1$ with a simple pole in the period strip]
    \label{prop:ConjugationNewtonPole}
    Every Newton map in class $\mathbf{R}_1$ with exactly a simple pole in the period strip of $\exp_1$, is conjugate to either
    \begin{enumerate}[label = (\roman*)] 
        \vspace{-1.0mm}
        \item The Newton's method of $F_{\alpha,m}(z)=\big(e^{\alpha z} \sin{\pi z}\big)^m$, for some $\alpha\in\CC\backslash\{\pm\pi i\}$ and $m\in\NN^*$.
        \item The Newton's method of $F_\beta(z)= \exp{\hspace{-0.2mm}\left(-\frac{z}{\beta}-\frac{1}{2\pi i\beta} e^{2\pi i z}\right)}$, for some $\beta\in\CC^*$.
    \end{enumerate}
\end{proposition}
\begin{proof}
    Using the notation in \pref{thm:C_AtlasBWD}, any Newton map $N_F\in\mathbf{R}_1$ is the Newton's method of an entire function $F$ of the form \pref{eq:5_FzNewtonR1}. Note that $N_F$ may have no fixed points if we allow the case $P^{-1}(0)\cap\CC^*=\emptyset$.

    \noindent Since $N_F$ is required to have only one simple pole in $S_0=\{ z: -\frac{1}{2}<\operatorname{Re}z \leq \frac{1}{2} \}$ (the period strip of $\exp_1$), it follows from \pref{eq:5_PolesZerosRlambda} that $P$ has at most one root (which may be multiple). There are two possible cases:
    \begin{enumerate}[label = (\roman*)] 
        \vspace{-1.5mm}
        \item If $P$ has exactly one zero (which lies in $\CC^*$, as $P(0)\neq 0$), say $e^{2\pi i a_0}$ of multiplicity $m\geq 1$, then $\deg{P}=m$ and both $Q$ and $\widetilde{Q}$ must be constants due to \pref{eq:5_PolesZerosRlambda}. In this situation, without loss of generality, we may assume that $P(w)=(w-e^{2\pi i a_0})^m$, and \pref{thm:C_AtlasBWD} yields that
        \begin{equation}
            \label{eq:5_NewtonFFformM}
            F(z) = e^{(\lambda-\pi i)z} \left(e^{2\pi i z}-e^{2\pi i a_0}\right)^m, \quad \mbox{ and } \quad N_F(z) = z - \frac{e^{2\pi i z}-e^{2\pi i a_0}}{(\lambda+\pi i m)e^{2\pi i z}-(\lambda-\pi i m) e^{2\pi i a_0}}, 
        \end{equation}
        where $\lambda\in\CC\backslash\{\pm \hspace{0.2mm}\pi i m\}$. It can be checked that the Newton map in \pref{eq:5_NewtonFFformM} is conjugate, via $z\mapsto z-a_0$, to the Newton's method of $F_{\alpha,m}$ in the statement, with $\alpha=\sfrac{\lambda}{m}$.
        \item If $P$ is a (non-zero) constant, then either $\deg{Q}=1$ or $\deg{\widetilde{Q}}=1$ from \pref{eq:5_PolesZerosRlambda}, which may be assumed to be linear, as a multiplicative constant in $F$ do not alter the Newton map. Consider first that $Q(w)=w$, and so $\widetilde{Q}$ is constant, then the entire function $F$ in \pref{eq:5_FzNewtonR1} may be written as
        \begin{equation}
            \label{eq:5_NewtonBRform}
            F(z) = e^{\lambda z} \exp{\hspace{-0.3mm}\left(e^{2\pi i z}\right)}, \quad \mbox{ and } \quad N_F(z) = z - \frac{1}{\lambda+2\pi i\hspace{0.2mm} e^{2\pi i z}}, 
        \end{equation}
        where $\lambda\in\CC^*$ such that $N_F\in\mathbf{R}_1$. The meromorphic map $N_F$ is conjugate through $z\mapsto z-a_\lambda$, where $a_\lambda\in\CC$ satisfies that $e^{2\pi i a_\lambda}=\sfrac{\lambda}{2\pi i}$, to the Newton's method of $F_{\beta}$ in the statement, with $\beta=-\sfrac{1}{\lambda}$.
        
        \noindent The case $\deg{\widetilde{Q}}=1$ (and so $Q$ constant) follows in exactly the same manner, since the Newton map of $e^{\lambda z} \exp{\hspace{-0.3mm}\left(e^{-2\pi i z}\right)}$ is conjugate to the one in \pref{eq:5_NewtonBRform} via $z\mapsto -z+\sfrac{1}{2}$.
        $\hfill\square$
    \end{enumerate}
\end{proof}

On the one hand, observe that the Newton's methods from the case \textit{(ii)} of \pref{prop:ConjugationNewtonPole} exactly correspond to the family of meromorphic maps $N_\beta$ in \pref{ex:LiftParabolicBR}, which was introduced by Buff and R\"uckert \cite{Buff2006}, and their projections $g_\beta$ (via $\exp_1$) have a unique essential singularity at $-1$ for all $\beta\in\CC^*$.

On the other hand, the case \textit{(i)} of \pref{prop:ConjugationNewtonPole} delivers a family of Newton maps $N_{\alpha,m}$ with fixed points (at the integers) of multiplier $\frac{m-1}{m}$, which leads to the pseudotrigonometric family $\mathbf{N}_\lambda$ (see \pref{def:classN}) when $\alpha=\lambda$ and $m=1$, i.e. to the Newton's method of $F_\lambda(z)=e^{\lambda z} \sin{\pi z}$, $\lambda\in\CC\backslash\{\pm \pi i\}$.
\begin{remark}[Relaxed Newton's method]
    \label{rem:relaxedNewton}
    Following \cite[\S6.2]{Bergweiler1993}, it is of interest to notice that for $m\geq2$, the function $N_{\lambda,m}$ is the \textit{relaxed Newton's method} of $F_\lambda$ with relaxation factor $\frac{1}{m}\in \big(0,\frac{1}{2}\big]$, that is,
    \begin{equation}
        N_{\lambda,m}(z) = z-\frac{1}{m}\frac{F_\lambda(z)}{F_\lambda'(z)} = z - \frac{1}{m} \frac{\hspace{2mm} e^{2\pi i z}-1}{\left(\lambda+\pi i\right)e^{2\pi i z}-\left(\lambda-\pi i\right)}
    \end{equation}
    Due to \cite[Thm 2.2]{Bergweiler1994}, if $0$ is not an asymptotic value of $F_\lambda$, as $m$ increases, the complement of the union of the immediate basins of the roots of $F_\lambda$ (including all those regions of starting values for which Newton's method fails) shrinks, in the sense of the Lebesgue measure on $\hCC$. Notice that 
    \begin{equation}
        F_\lambda(z)=\frac{1}{2i}e^{(\lambda-\pi i)z}(e^{2\pi i z}-1), \quad \mbox{ with } \quad F_\lambda(z)\sim e^{(\lambda\mp \pi i)z} \mbox{ as } \operatorname{Im}z\to\pm\infty,
    \end{equation}
    and the asymptotic path $\gamma_0^{\pm}:[t_0,\infty)\to\CC$ such that $e^{(\lambda\pm \pi i) \gamma_0^\pm(t)}\to 0$ as $t\to\infty$, is given by $\gamma_0^\pm(t)=(-\overline{\lambda}\pm \pi i) t$. Hence, $0\in\mathcal{AV}(F_\lambda)$ if and only if $\operatorname{Im}\gamma_0^\mp(t)\to\pm\infty$ as $t\to\infty$, or equivalently, $\pm \operatorname{Im}\lambda > \pi$. Then, in the case that $|\operatorname{Im}\lambda|<\pi$, the relaxation of $N_\lambda$ trades speed of convergence for the ease of finding a good initial guess.
\end{remark}

In what follows we study the uniparametric family $\mathbf{N}_\lambda$ of pseudotrigonometric Newton's methods, as a representative of those Newton maps in class $\mathbf{R}_1$ with exactly one superattracting fixed point and a simple pole in the period strip of $\exp_1$. Let us highlight the dynamically important points of their projections $g_\lambda$.

\begin{lemma}[Projections of Newton's methods in the family $\mathbf{N}_\lambda$]
    \label{lem:PseudoTrig}
    Let $N_\lambda$ be a Newton map in the pseudo-trigonometric family $\mathbf{N}_\lambda$, $\lambda\in\CC\backslash\{\pm\pi i\}$. Then its exponential projection $g_\lambda$, which is given by
    \vspace{-1mm}
    \begin{equation}
        \label{eq:projNewtonMobius}
        g_\lambda(w)=w e^{2\pi i M_\lambda(w)}, \quad \mbox{ where } \quad M_\lambda(w) = -\frac{\hspace{2mm} w-1}{\left(\lambda+\pi i\right)w-\left(\lambda-\pi i\right)},\vspace{-0.5mm}
    \end{equation}
    has a unique essential singularity at $B_\lambda:=\frac{\lambda-\pi i}{\lambda+\pi i}$. Moreover,
    \begin{enumerate}[label = (\roman*)] 
        \vspace{-0.5mm}
        \item The set of fixed points of $g_\lambda$ consists of the points at $0$ and $\infty$, with multipliers \vspace{-0.5mm}
        \begin{equation}
            \label{eq:MultiProjAVG}
            g_\lambda'(0)=\exp{\left( \frac{2\pi i \ }{\pi i -\lambda} \right)} \quad \mbox{ and } \quad g_\lambda'(\infty)=\exp{\left( \frac{2\pi i}{\pi i+\lambda} \right)},
        \end{equation}
        and, for $\sigma\in\ZZ$, the points at 
        \vspace{-0.5mm}
        \begin{equation}
            \label{eq:ProjPseudoFixedG}
            w_{\sigma}^* := \frac{1+\left( \lambda-\pi i \right) \sigma}{1+\left( \lambda+\pi i \right) \sigma}, \quad \mbox{ with } \quad g_\lambda'(w_\sigma^*) = 1 - \Big( 1+ ( \lambda-\pi i)\sigma \Big) \Big( 1+( \lambda+\pi i)\sigma \Big).
        \end{equation}
        
        \vspace{-1.5mm}
        \item The set of singular values of $g_\lambda$ consists of the fixed asymptotic values at $0$ and $\infty$, the fixed critical point at $w_0^*=1$, and the image of the only free critical point at $C_\lambda:=B_\lambda^2$.
    \end{enumerate}
\end{lemma}
\begin{proof} %
    \noindent Recall that $N_\lambda$ is the Newton's method of the entire function $F_\lambda(z)=e^{\lambda z} \sin{\pi z}$, $\lambda\in\CC\backslash\{\pm\pi i\}$, which is given by $N_\lambda(z)=z+M_\lambda(e^{2\pi i z})$, where $M_\lambda$ is the M\"obius transformation in \pref{eq:projNewtonMobius}. For any $k\in \ZZ$, \vspace{-0.75mm}
    \begin{equation*}
        a_{k}:=k, \qquad b_{\lambda,k} := \frac{1}{2\pi i}\mathrm{Log}\left(\frac{\lambda-\pi i}{\lambda+\pi i}\right) + k, \qquad c_{\lambda,k} := 2 b_{\lambda,0}+k,
    \end{equation*}
    are the (simple) roots of $F_\lambda$, $F_\lambda'$, $F_\lambda''$, respectively, where $\mathrm{Log}$ denotes the principal branch of the logarithm. Then, as summarized in \pref{tab:NewtonRoots}, we have that, for all $k$, $e^{2\pi i a_k}=1$ is a superattracting fixed point of $g_\lambda$, $B_\lambda:=e^{2\pi i b_{\lambda,k}}=\frac{\lambda-\pi i}{\lambda+\pi i}$ is an essential singularity of $g_\lambda$, and $C_\lambda:=e^{2\pi i c_{\lambda,k}}=B_\lambda^2$ is a free critical point of $g_\lambda$.

    \noindent Given that $N_\lambda$ is in class $\mathbf{R}_1$, it follows from \pref{cor:classRsv} that  $\mathcal{E}(g_\lambda) =\{B_\lambda\}$, and the points at $0$ and $\infty$ are the asymptotic values of $g_\lambda$. They are also fixed points with multipliers $g_\lambda'(0)=e^{2\pi i M_\lambda(0)}$ and $g_\lambda'(\infty)=e^{-2\pi i M_\lambda(\infty)}$, respectively, which gives \pref{eq:MultiProjAVG} since \vspace{-0.75mm} 
    \begin{equation*}
        \lim_{\operatorname{Im}z\to\pm\infty} M_\lambda(e^{2\pi i z}) = \frac{1}{\pm\pi i-\lambda}.
    \end{equation*}
    We know that the projection via $\exp_1$ of $(1,\sigma)$-pseudoperiodic points of $N_\lambda$, $\sigma\in\ZZ$, are the fixed points of $g_\lambda$ (see \pref{lem:characPseudoPoint}), which we denote by $w_\sigma^*$. Hence, these satisfy the equation $M_\lambda(w_\sigma^*)=\sigma$, that is, they are of the form given in the statement of \textit{(i)}. Furthermore, notice that
    \vspace{-0.75mm}
    \begin{equation*}
        g_\lambda'(w) = e^{2\pi i M_\lambda(w)} \Big( 1+2\pi i w M_\lambda'(w) \Big), \quad \mbox{ where } \quad M_\lambda'(w)=\frac{-2\pi i}{\big(\left(\lambda+\pi i\right)w-\left(\lambda-\pi i \right)\big)^2}.
    \end{equation*}
    Since $M_\lambda'(w_\sigma^*)=\frac{-1}{2\pi i}\big(1+(\lambda+\pi i)\sigma\big)^2$, a straight computation leads to the multipliers of $w_\sigma^*$ in \pref{eq:ProjPseudoFixedG}. Observe that $w_\sigma^*$ may be the fixed point at infinity if $\lambda=-\pi i-\frac{1}{\sigma}$, in which case we already know that $g_\lambda'(w_\sigma^*)=1$.
    
    \noindent To see \textit{(ii)}, we note that \pref{rem:projCV} implies that the only critical values of $g_\lambda$ are the fixed point at $1$, and the image of $C_\lambda$. Therefore, as $\mathcal{AV}(g_\lambda)=\{0,\infty\}$, we conclude that $\mathcal{S}(g_\lambda)=\{0,1,g_\lambda(C_\lambda),\infty\}$.
    $\hfill\square$
\end{proof}

We observe from \pref{eq:MultiProjAVG} that one of the fixed asymptotic values of $g_\lambda$ (either the point at $0$ or $\infty$) is attracting if $|\operatorname{Im}\lambda|>\pi$. This is indeed the case in which $0$ is an asymptotic value of $F_\lambda$ (see \pref{rem:relaxedNewton}).

\begin{remark}[Logarithmic singularities]
    \label{rem:logarithmicSing}
    The projection $g_\lambda$ of a pseudotrigonometric Newton map is conjugate via $M_\lambda$ (which places the essential singularity at $\infty$ and the fixed points $w_\sigma^*$ at $\sigma\in\ZZ$) to the function 
    \begin{equation}
        \tilde{g}_\lambda(\zeta):=M_\lambda\circ g_\lambda \circ M_\lambda^{-1}(\zeta) = - \frac{\big(1+(\lambda-\pi i)\zeta\big) e^{2\pi i\zeta} - \big(1+(\lambda+\pi i)\zeta\big)}{(\lambda+\pi i)\big(1+(\lambda-\pi i)\zeta\big) e^{2\pi i\zeta} - (\lambda-\pi i)\big(1+(\lambda+\pi i)\zeta\big)}.
    \end{equation}
    The singular values of $\tilde{g}_\lambda$ are the fixed asymptotic values at $\frac{1}{\pi i \pm \lambda}$, the superattracting fixed point at $0$, and the free critical point at $\widetilde{C}_\lambda:=\frac{-2\lambda}{\lambda^2+\pi^2}$. Furthermore, its order of growth (in the sense of Nevanlinna \cite{Nevanlinna1970}) is $\rho(\tilde{g}_\lambda)=1$, as the sum and division of meromorphic functions do not increment the order (this is also the case for $F_\lambda$ and functions in class $\mathbf{R}_\ell$, given that their periodic part is the quotient of trigonometric polynomials). 
    
    \noindent Since $\tilde{g}_\lambda$ is a finite-type map of finite order, all singularities of $\tilde{g}_\lambda^{-1}$ over an asymptotic value are \textit{logarithmic} due to \cite[Cor. 1]{Bergweiler1995e}. This means that for $v\in\mathcal{AV}(\tilde{g}_\lambda)$, there exists $r>0$ and a component $\mathcal{U}_r$ of $g_\lambda^{\hspace{0.2mm}-1}(B(v,r)\backslash\{v\})$, where $B(v,r)$ denotes the open disk of center $v$ and radius $r$ in the spherical metric on $\hCC$, such that $U$ contains no preimages of $v$ and $f:\mathcal{U}_r\to B(v,r)\backslash\{v\}$ is a universal covering. The domain $\mathcal{U}_r$ is called a \textit{logarithmic tract}, and $\tilde{g}_\lambda = \exp\circ\hspace{0.2mm}\eta+v$ in $\mathcal{U}_r$ (or $\tilde{g}_\lambda=e^{-\eta}$ if $v=\infty$) for some conformal (i.e. one-one analytic) map $\eta$ from $\mathcal{U}_r$ onto the half-plane $\{z: \operatorname{Re}z < \ln{r}\}$; see more details in \cite{Bergweiler2008} and \cite[\S6]{Zheng2011}. The same holds for $g_\lambda$.
\end{remark}

Hence, in the situation where $0\in\mathcal{AV}(F_\lambda)$, as its critical values only accumulate at $\infty$, there is a logarithmic singularity of $F_\lambda^{-1}$ over $0$. This one induces a (doubly parabolic) Baker domain by a result of Buff and R\"uckert \cite[Thm. 4.1]{Buff2006}, as already mentioned in \pref{ex:LiftParabolicBR}. This agrees with our results as follows.  

\begin{example}[Coexistence of Baker domains and the basins of the roots]
   \label{ex:BakerWD_SuperCoexistence}
   It can be seen from \pref{cor:C_BakerWandering} that the (pseudotrigonometric) Newton map of $F_\lambda(z)=\frac{e^{(\lambda-\pi i)z}}{2i}(e^{2\pi i z}-1)$, has a Baker domain if $|\operatorname{Im}\lambda|>\pi$, alongside the infinitely many basins of the roots of $F_\lambda$. We recall that this follows from \pref{thm:logFatouJulia} since the fixed point at $\infty$ (resp. $0$) of the projection $g_\lambda$ is attracting if and only if $\operatorname{Im}\lambda<-\pi$ (resp. $\operatorname{Im}\lambda>\pi$). 
   
   \noindent As a particular example, for $\lambda^*=-3\pi i$, \pref{lem:PseudoTrig} asserts that
   \begin{equation}
       g_{\lambda^*}^{}(w)=w \exp{\left(\frac{w-1}{w-2}\right)}, \qquad g_{\lambda^*}'(0) = e^{\sfrac{1}{2}}, \quad \mbox{ and } \quad g_{\lambda^*}'(\infty) = e^{-1}.
   \end{equation}
   Moreover, the half-line $\gamma_0^+:=-i[t_0,\infty)$, where $t_0>\frac{\ln{2}}{2\pi}$ such that it avoids the critical points of $F_\lambda$ (i.e. the poles of the Newton's method), is an asymptotic path associated to $0\in\mathcal{AV}(F_{\lambda^*})$ (see \pref{rem:relaxedNewton}). As shown in \pref{fig:BakerWDf_fixed}, the immediate basin of attraction of $\infty$ of $g_{\lambda^*}^{}$, which contains the free critical point at $C_\lambda=4$, lifts via $\exp_1$ to a simply-connected invariant Baker domain $U$ of infinite degree for the Newton's method, as it contains a logarithmic tract of $F_\lambda$ over $0$ with $\gamma_0^+$ as its asymptotic path (see \pref{rem:logarithmicSing}).
   \vspace{-1.5mm}
    \begin{figure}[H] 
        \centering
        \hfill
        \begin{tikzpicture}[cross/.style={path picture={ 
            \draw[white!80]
            (path picture bounding box.south) -- (path picture bounding box.north) (path picture bounding box.west) -- (path picture bounding box.east);}
        }]
        \node[anchor=south west,inner sep=0] (image) at (0,0) {\includegraphics[width=0.526\linewidth]{BD2FishFc_x-1.5_1.5_y-0.85_0.85.pdf}};
        \begin{scope}[x={(image.south east)},y={(image.north west)}]
        \coordinate (a) at ($ ({(0+1.5)/3},{(0+0.85)/1.7}) $);
        \coordinate (a2) at ($ ({(1+1.5)/3},{(0+0.85)/1.7}) $);
        \coordinate (a3) at ($ ({(-1+1.5)/3},{(0+0.85)/1.7}) $);
        \coordinate (b) at ($ ({(0+1.5)/3},{(-0.11+0.85)/1.7}) $);
        \coordinate (b2) at ($ ({(1+1.5)/3},{(-0.11+0.85)/1.7}) $);
        \coordinate (b3) at ($ ({(-1+1.5)/3},{(-0.11+0.85)/1.7}) $);
        \fill[white!80] (a) circle (1pt) node[above]{\tiny$0$};
        \fill[white!80] (a2) circle (1pt) node[above]{\tiny$1$};
        \fill[white!80] (a3) circle (1pt) node[above]{\tiny\hspace{-0.55em}$\shortminus 1$};
        \draw[line width=0.42pt,cross,color=white!80] (b) circle (1.25pt);
        \draw[line width=0.42pt,cross,color=white!80] (b2) circle (1.25pt);
        \draw[line width=0.42pt,cross,color=white!80] (b3) circle (1.25pt);
        \draw (0.0435,0.0706) node[draw=none,fill=none, color=white] {\scriptsize $U$};
        \draw[line width=0.1pt, color = darkgray!80, dashed] ($({1/3},0)$) -- ($({1/3},1)$);
        \draw[line width=0.1pt, color = darkgray!80, dashed] ($({2/3},0)$) -- ($({2/3},1)$);
        \draw[line width=0.5pt,color=white,-stealth] ($ ({(-0.5+1.5)/3},{(0.36+0.85)/1.7}) $) to ($ ({(-0.5+1.5)/3},{(0.276484+0.85)/1.7}) $);
        \draw[line width=0.5pt,color=white,-stealth] ($ ({(0.5+1.5)/3},{(0.36+0.85)/1.7}) $) to ($ ({(0.5+1.5)/3},{(0.276484+0.85)/1.7}) $);
        \draw[line width=0.5pt,color=white,-stealth] ($ ({(+0.5+1.5)/3},{(-0.33+0.85)/1.7}) $) to ($ ({(+0.5+1.5)/3},{(-0.473163+0.85)/1.7}) $);
        \draw[line width=0.5pt,color=white,-stealth] ($ ({(-0.5+1.5)/3},{(-0.33+0.85)/1.7}) $) to ($ ({(-0.5+1.5)/3},{(-0.473163+0.85)/1.7}) $);
        \end{scope}
    
    \end{tikzpicture}\hfill
        \begin{tikzpicture}[cross/.style={path picture={ 
            \draw[white!80]
            (path picture bounding box.south) -- (path picture bounding box.north) (path picture bounding box.west) -- (path picture bounding box.east);}
        }]
        \node[anchor=south west,inner sep=0] (image) at (0,0) {\includegraphics[width = 0.360\textwidth]{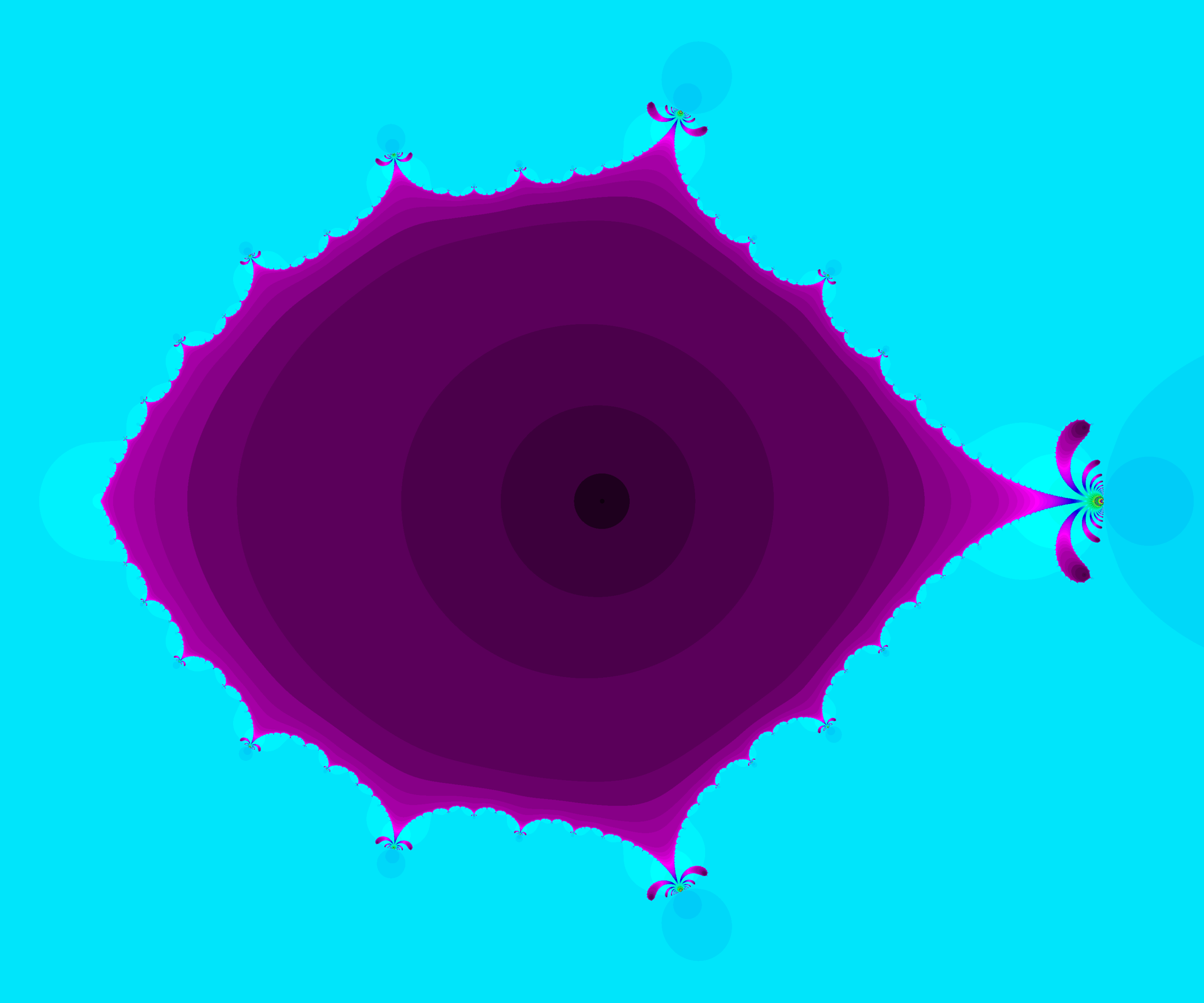}};
        \begin{scope}[x={(image.south east)},y={(image.north west)}]
        \coordinate (a) at ($ ({(1+0.2)/2.4},{(0+1)/2}) $);
        \coordinate (b) at ($ ({(2+0.2)/2.4},{(0+1)/2}) $);
        \fill[white!80] (a) circle (1.1pt) node[above]{\tiny$1$};
        \draw[line width=0.42pt,cross,color=white!80] (b) circle (1.4pt);
        \end{scope}
        
        \end{tikzpicture}\hfill\vspace{-1mm}
        \caption{\textit{Left (dynamical plane of $N_\lambda$ for $\lambda=-3\pi i$)}: The basins of the fixed points of $N_\lambda$ at the integers coexist with a simply-connected Baker domain $U$ (in blue) as in \pref{fig:BakerWDf_fixed}; see \pref{ex:BakerWD_SuperCoexistence}. Range: $[-1.5,1.5]\times[-0.85,0.85]$.
        \textit{Right (dynamical plane of its projection $g_\lambda$)}: The superattracting basin of $1$ (in purple), and the basin of the fixed point at $\infty$ (in blue) which lifts via $\exp_1$ to the Baker domain $U$. Range: $[-0.2,2.2] \hspace{-0.3mm}\times\hspace{-0.3mm} [-1,1]$. The white $\oplus$ at $2$ refers to the projection of the poles of $N_\lambda$, and the color palettes to the speed of convergence to these fixed points of $g_\lambda$.}
        \label{fig:PseudoBakerfg_fixed}
    \end{figure}
\end{example} 

\vspace{-3.5mm}
It is important to emphasize that the fixed points $w_\sigma^*$ of $g_\lambda$ in \pref{lem:PseudoTrig}, $\sigma\in\ZZ$, are the projections via $\exp_1$ of a $(1,\sigma)$-pseudoperiodic point $z_\sigma^*$ of the Newton's method $N_\lambda$, in the sense that $N_\lambda(z_\sigma^*)=z_\sigma^*+\sigma$, which escapes to $\infty$ under iteration unless $\sigma=0$ (see \pref{prop:IteratesPseudo}, $\ell=1$ case).

\begin{example}[Coexistence of wandering domains and the basins of the roots]
    \label{ex:WD_SuperCoexistence}
    Applying \pref{cor:C_BakerWandering} we explicitly identify escaping wandering domains in the family $\mathbf{N}_\lambda$, for those $\lambda\in\CC\backslash\{\pm\pi i\}$ such that the fixed point of the projection $g_\lambda$ at $w_\sigma^*$, for some $\sigma\in\ZZ^*$ (i.e. $w_\sigma^*\neq 1$), is attracting, parabolic or of Siegel type.

    \noindent In particular, $w_\sigma^*(\lambda)$ is a superattracting fixed point of $g_\lambda$ if its multiplier \pref{eq:ProjPseudoFixedG} vanishes, which occurs for \vspace{-0.55mm}
    \begin{equation}
        \label{eq:WD_fixedOmega1}
        \lambda_\sigma^\pm:=\frac{-1}{\sigma}\pm i \sqrt{\pi^2-\frac{1}{\sigma^2}}, \quad \mbox{ and so } \quad w_\sigma^*(\lambda_\sigma^\pm) = 1 + 2\pi\sigma^2 \left( -\pi \pm \sqrt{\pi^2-\frac{1}{\sigma^2}} \hspace{0.2mm}\right). \vspace{-1.05mm}
    \end{equation}
    In this case, the Newton map $N_{\lambda_\sigma^\pm}$ has infinitely many wandering domains $U_k$, $k\in\ZZ$, which emerge as the lift via $\exp_1$ of the immediate basin of $w_\sigma^*$ of $g_{{\lambda_\sigma^\pm}}$, $\sigma\in\ZZ^*$, where $w_\sigma^*$ is indeed the free critical $C_{\lambda_\sigma^\pm}$. The chain of wandering components contains the pseudofixed points $z_\sigma^*$ such that $e^{2\pi i z_\sigma^*}=w_\sigma^*$, with $N_{\lambda}^n(U_k)\subset U_k+n\sigma$, $n\in\NN$, and coexists with the infinitely many basins of the roots of $F_\lambda$, as displayed in \pref{fig:PseudoWDfg_fixed} for $\lambda=\lambda_1^-$.
    \vspace{-1.5mm}
    \begin{figure}[H] 
        \centering
        \hfill
        \begin{tikzpicture}[cross/.style={path picture={ 
            \draw[white!80]
            (path picture bounding box.south) -- (path picture bounding box.north) (path picture bounding box.west) -- (path picture bounding box.east);}
        }]
        \node[anchor=south west,inner sep=0] (image) at (0,0) {\includegraphics[width=0.526\linewidth]{fNfamily2WD1c_x-1.5_1.5_y-1.15_0.55.pdf}};
        \begin{scope}[x={(image.south east)},y={(image.north west)}]
            \coordinate (a) at ($ ({(0+1.5)/3},{(0+1.15)/1.7}) $);
            \coordinate (a2) at ($ ({(1+1.5)/3},{(0+1.15)/1.7}) $);
            \coordinate (a3) at ($ ({(-1+1.5)/3},{(0+1.15)/1.7}) $);
            \coordinate (b) at ($ ({(-0.5+1.5)/3},{(-0.57667+1.15)/1.7}) $);
            \coordinate (b2) at ($ ({(0.5+1.5)/3},{(-0.57667+1.15)/1.7}) $);
            \draw[line width=0.1pt, color = darkgray!80, dashed] ($({1/3},0)$) -- ($({1/3},1)$);
            \draw[line width=0.1pt, color = darkgray!80, dashed] ($({2/3},0)$) -- ($({2/3},1)$);
            \draw[line width=0.48pt,cross,color=white!80] ($ ({(-0.75+1.5)/3},{(-0.28831336+1.15)/1.7}) $) circle (1.65pt);
            \draw[line width=0.48pt,cross,color=white!80] ($ ({(0.25+1.5)/3},{(-0.28831336+1.15)/1.7}) $) circle (1.65pt);
            \draw[line width=0.48pt,cross,color=white!80] ($ ({(1.25+1.5)/3},{(-0.28831336+1.15)/1.7}) $) circle (1.65pt);
            \fill[white!80] (a) circle (1pt) node[above]{\tiny$0$};
            \fill[white!80] (a2) circle (1pt) node[above]{\tiny$1$};
            \fill[white!80] (a3) circle (1pt) node[above]{\tiny\hspace{-0.55em}$\shortminus 1$};
            \draw (0.3935,0.0690) node[draw=none,fill=none, color=white] {\scriptsize $U_{\shortminus 1}^{}$};
            \draw (0.7175,0.0745) node[draw=none,fill=none, color=white] {\scriptsize $U_{0}^{}$};
            \draw[line width=0.5pt,color=white,-stealth] ($ ({(-0.574615+1+1.5)/3},{(-0.094429+1.15)/1.7}) $) to[bend right] (b2);
            \draw[line width=0.5pt,color=white,-stealth] ($ ({(0.24955+1+1.5)/3},{(-0.161072+1.15)/1.7}) $) to[bend right] (b2);
            \draw[line width=0.5pt,color=white,-stealth] ($ ({(-0.5+1.5)/3},{(-0.57667+1.15)/1.7}) $) to[bend right] (b2);
            \draw[line width=0.5pt,color=white,-stealth] ($ ({(-1.5+1.5)/3},{(-0.57667+1.15)/1.7}) $) to[bend right] (b);
            \draw[line width=0.5pt,color=white,-stealth] (b2) to[bend right] ($ ({(1.5+1.5)/3},{(-0.57667+1.15)/1.7}) $);
            \draw (b2) node[color=black,scale=1] {\tiny$\star$} node[color=black,below]{\tiny\hspace{0.41em}$z_{\scalebox{0.7}{1}}^{\raisebox{-0.15\height}{\scalebox{0.7}{*}}}$};
            \draw (b) node[color=black,scale=1] {\tiny$\star$} node[color=black,below]{\tiny\hspace{0.2em}$z_{\scalebox{0.7}{1}}^{\raisebox{-0.15\height}{\scalebox{0.7}{*}}}\hspace{0.1mm}\shortminus\hspace{0.1mm} 1$};
        \end{scope}
        \end{tikzpicture}\hfill
        \begin{tikzpicture}[cross/.style={path picture={ 
            \draw[white!80]
            (path picture bounding box.south) -- (path picture bounding box.north) (path picture bounding box.west) -- (path picture bounding box.east);}
        }]
        \node[anchor=south west,inner sep=0] (image) at (0,0) {\includegraphics[width = 0.360\textwidth]{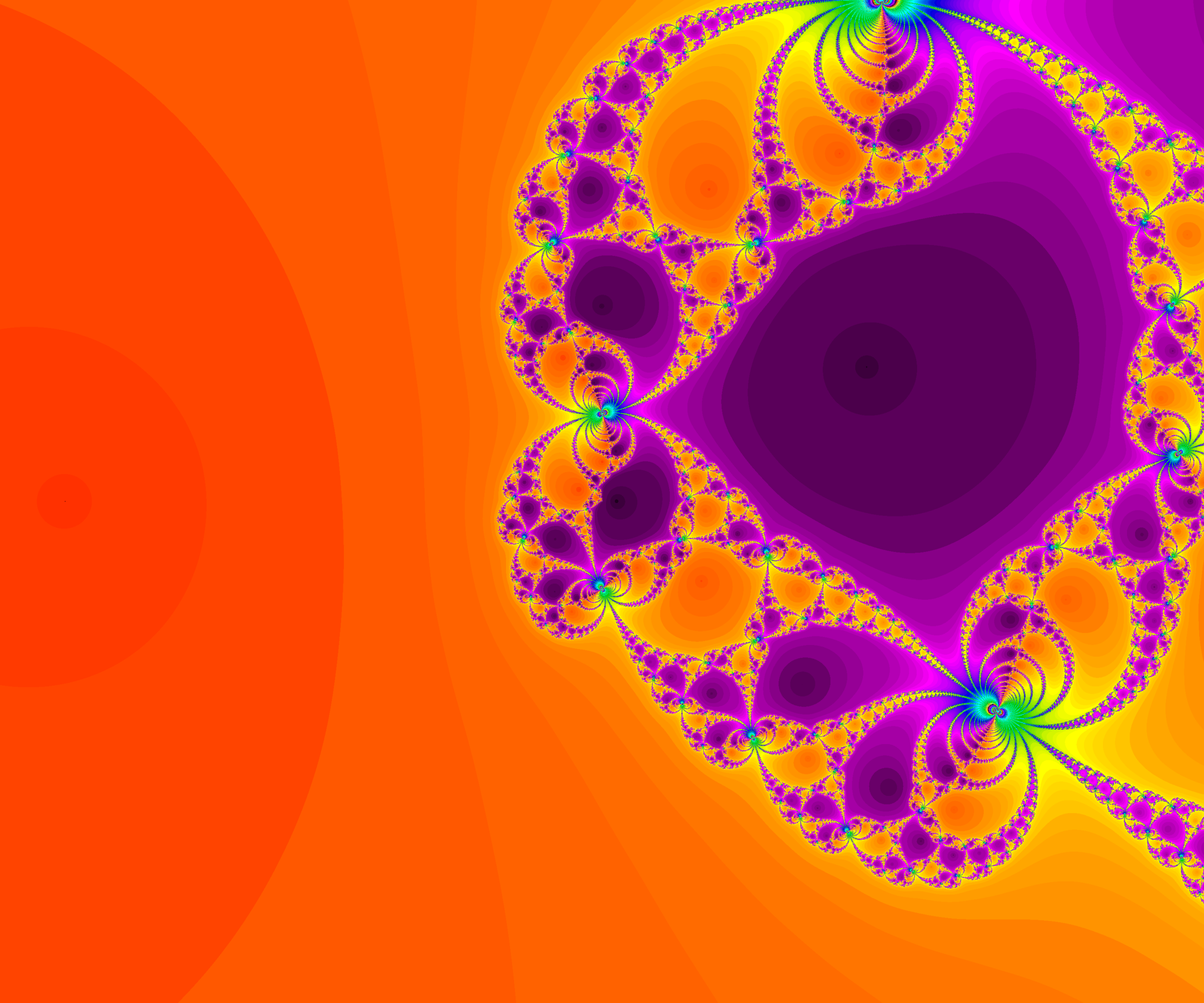}};
        \begin{scope}[x={(image.south east)},y={(image.north west)}]
        \coordinate (a) at ($ ({(1+42)/84},{(0+35)/70}) $);
        \coordinate (a1) at ($ ({(-37.45171656+42)/84},{(0+35)/70}) $);
        \coordinate (b) at ($ ({(0+42)/84},{(6.11978+35)/70}) $);
        \fill[white!80] (a) circle (0.8pt) node[right]{\hspace{-0.5mm}\tiny$1$};
        \draw[line width=0.48pt,cross,color=white!80] (b) circle (1.65pt);
        \draw (a1) node[color=black,scale=1] {\tiny$\star$} node[color=black,below]{\tiny\hspace{0.41em}$w_{\scalebox{0.7}{1}}^{\raisebox{-0.15\height}{\scalebox{0.7}{*}}}$};
        \draw[line width=0.48pt,color=white,-stealth] ($ ({(-1.61469+42)/84},{(0.817809+35)/70}) $) to[bend left] (a1);
        \draw[line width=0.48pt,color=white,-stealth] ($ ({(0.00777889+42)/84},{(2.75121+35)/70}) $) to[bend right] (a1);
        \end{scope}
        \end{tikzpicture}
        \hfill\vspace{-1mm}
        \caption{\textit{Left (dynamical plane of $N_\lambda$ for $\lambda=-1-\sqrt{\pi^2-1}$)}: The basins of the fixed points of $N_\lambda$ at $k\in\ZZ$ coexist with a chain of simply-connected wandering domains $U_k$ (in orange) such that $z_1^*+k\in U_k$ as in \pref{fig:BakerWDf_fixed}, where $z_1^*$ projects to the superattracting fixed point $w_1^*(\lambda_1^{\shortminus})\approx -37.45$ of $g_\lambda$ from \pref{ex:WD_SuperCoexistence}. Range: $[-1.5,1.5]\times[-1.15,0.55]$.
        \textit{Right (dynamical plane of its projection $g_\lambda$)}: The superattracting basin of $1$ (in purple), and the immediate basin of attraction of $w_1^*$ (in orange) which lifts to the wandering domains $U_k$. Range: $[-42,42] \hspace{-0.3mm}\times\hspace{-0.3mm} [-35,35]$. The white $\oplus$ refers to the projection of the poles of $N_\lambda$, and the color palettes to the speed of convergence to the fixed points of $g_\lambda$.} 
        \label{fig:PseudoWDfg_fixed}
    \end{figure}
\end{example}
\vspace{-3.5mm}

For the sake of completeness, we illustrate the dynamical planes in a case when the free critical point of the projection $g_\lambda$ lies in the basin of the superattracting fixed point at $1$ (i.e. the image under $\exp_1$ of the zeros of $F_\lambda$), and hence, as $g_\lambda$ is of finite-type, there are no other Fatou components.

\begin{example}[Basins of the roots as the only Fatou components]
    \label{ex:Newton0_tan}
    The free critical point of $g_\lambda$ coincides with the fixed point at $1$ if and only if $\lambda=0$. In this case, the pseudotrigonometric Newton map $N_0$ is precisely the Newton's method of $\sin{\pi z}$, given by 
    \begin{equation}
        \label{eq:N0_tan}
        N_0(z) = z-\frac{1}{\pi}\tan{\pi z}=z-\frac{1}{\pi i} \frac{e^{2\pi i z}-1}{e^{2\pi i z}+1},
    \end{equation}
    and its fixed points at $k\in\ZZ$ are critical points of $N_0$ of multiplicity two (i.e. $N_0''(k)=0$ but $N_0'''(k)\neq 0$).
    
    \noindent As detailed in \cite[Ex. 7.2]{Baranski2016}, the only periodic Fatou components of $\mathcal{F}(N_0)$ are the infinitely many immediate basins of attraction $U_k$ of the fixed points at $k\in\ZZ$, with $\operatorname{deg}N_0|_{U_k}^{}=3$. Moreover, each $U_k$ has two distinct accesses to $\infty$, and $\partial U$ contains exactly two accessible poles of $N_0$ as depicted in \pref{fig:Superfg_tan}. Observe that all the basins $U_k$ project to the superattracting basin of $1$ of $g_0$, and the lines $\{\operatorname{Re}z=\sfrac{1}{2}+k\}_{k\in\ZZ}\subset\mathcal{J}(N_0)$ are sent via $\exp_1$ to the negative real axis (e.g. the prepole of $N_0$ of order $2$ near the point $0.5-0.3816i$, is sent to an essential prepole of $g_0$ close to $-11$), showing the distortion of lengths by the exponential. \vspace{-1.0mm}
    \begin{figure}[H] 
        \centering\hfill
        \begin{tikzpicture}[cross/.style={path picture={ 
            \draw[white!80]
            (path picture bounding box.south) -- (path picture bounding box.north) (path picture bounding box.west) -- (path picture bounding box.east);}
        }]
        \node[anchor=south west,inner sep=0] (image) at (0,0) {\includegraphics[width = 0.561\textwidth]{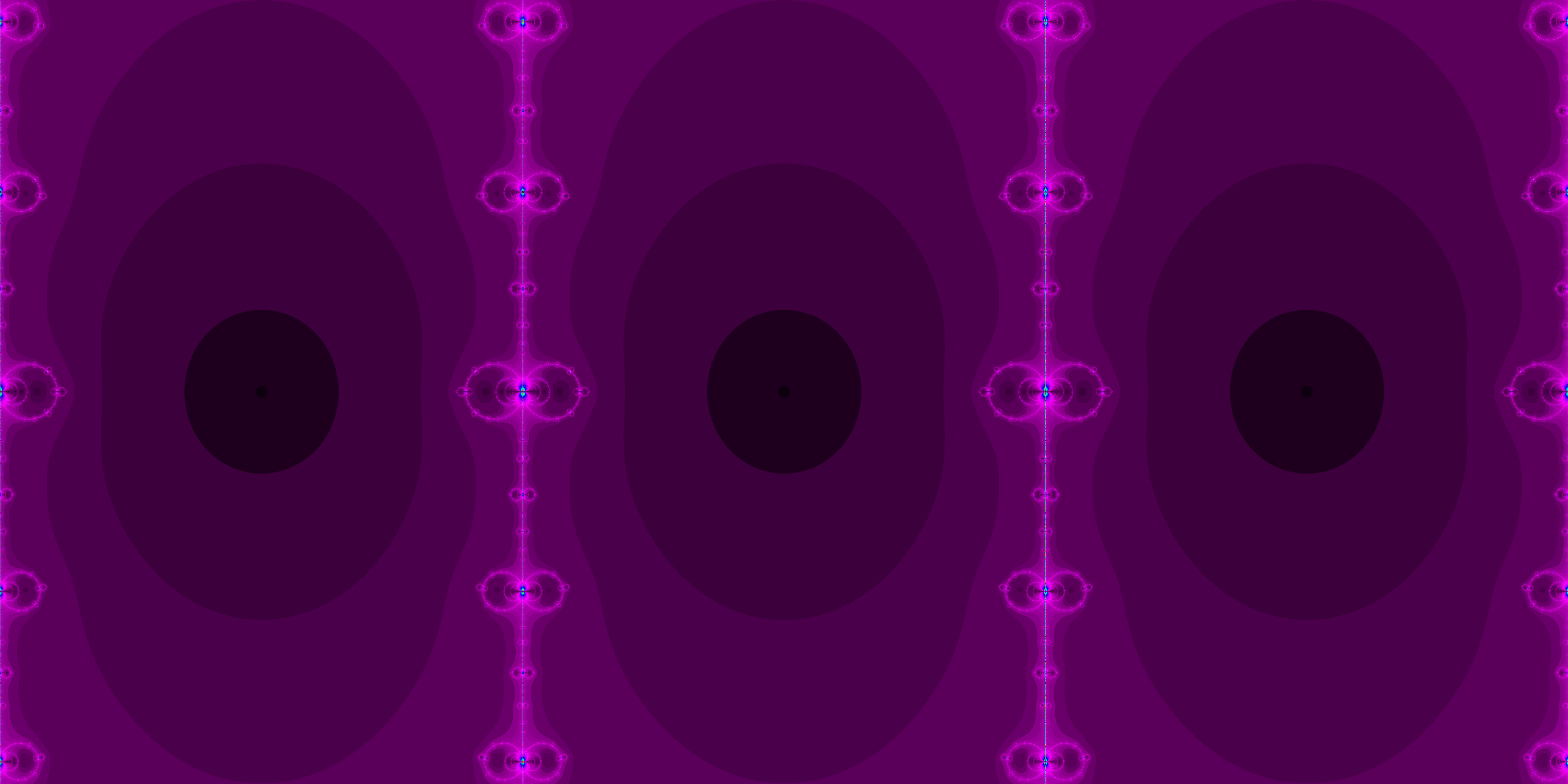}};
        \begin{scope}[x={(image.south east)},y={(image.north west)}]
        \coordinate (am1) at ($ ({(-1+1.5)/3},{(0+0.75)/1.5}) $);
        \coordinate (a0) at ($ ({(0+1.5)/3},{(0+0.75)/1.5}) $);
        \coordinate (a1) at ($ ({(1+1.5)/3},{(0+0.75)/1.5}) $);
        \fill[white!80] (a0) circle (1pt) node[above]{\tiny$0$};
        \fill[white!80] (a1) circle (1pt) node[above]{\tiny$1$};
        \fill[white!80] (am1) circle (1pt) node[above]{\tiny\hspace{-0.55em}$\shortminus 1$};
        \draw[line width=0.48pt,cross,color=white!80] ($ ({(-1.5+1.5)/3},{(0+0.75)/1.5}) $) circle (1.45pt);
        \draw[line width=0.48pt,cross,color=white!80] ($ ({(-0.5+1.5)/3},{(0+0.75)/1.5}) $) circle (1.45pt);
        \draw[line width=0.48pt,cross,color=white!80] ($ ({(0.5+1.5)/3},{(0+0.75)/1.5}) $) circle (1.45pt);
        \draw[line width=0.48pt,cross,color=white!80] ($ ({(1.5+1.5)/3},{(0+0.75)/1.5}) $) circle (1.45pt);
        \draw (0.177,0.0690) node[draw=none,fill=none, color=white] {\scriptsize $U_{\shortminus 1}^{}$};
        \draw (0.505,0.0720) node[draw=none,fill=none, color=white] {\scriptsize $U_{0}^{}$};
        \draw (0.837,0.0745) node[draw=none,fill=none, color=white] {\scriptsize $U_{1}^{}$};
        \end{scope}
        \end{tikzpicture}
        \hspace{2.7mm}\hfill
        \begin{tikzpicture}[cross/.style={path picture={ 
            \draw[white!80]
            (path picture bounding box.south) -- (path picture bounding box.north) (path picture bounding box.west) -- (path picture bounding box.east);}
        }]
        \node[anchor=south west,inner sep=0] (image) at (0,0) {\includegraphics[width = 0.357\textwidth]{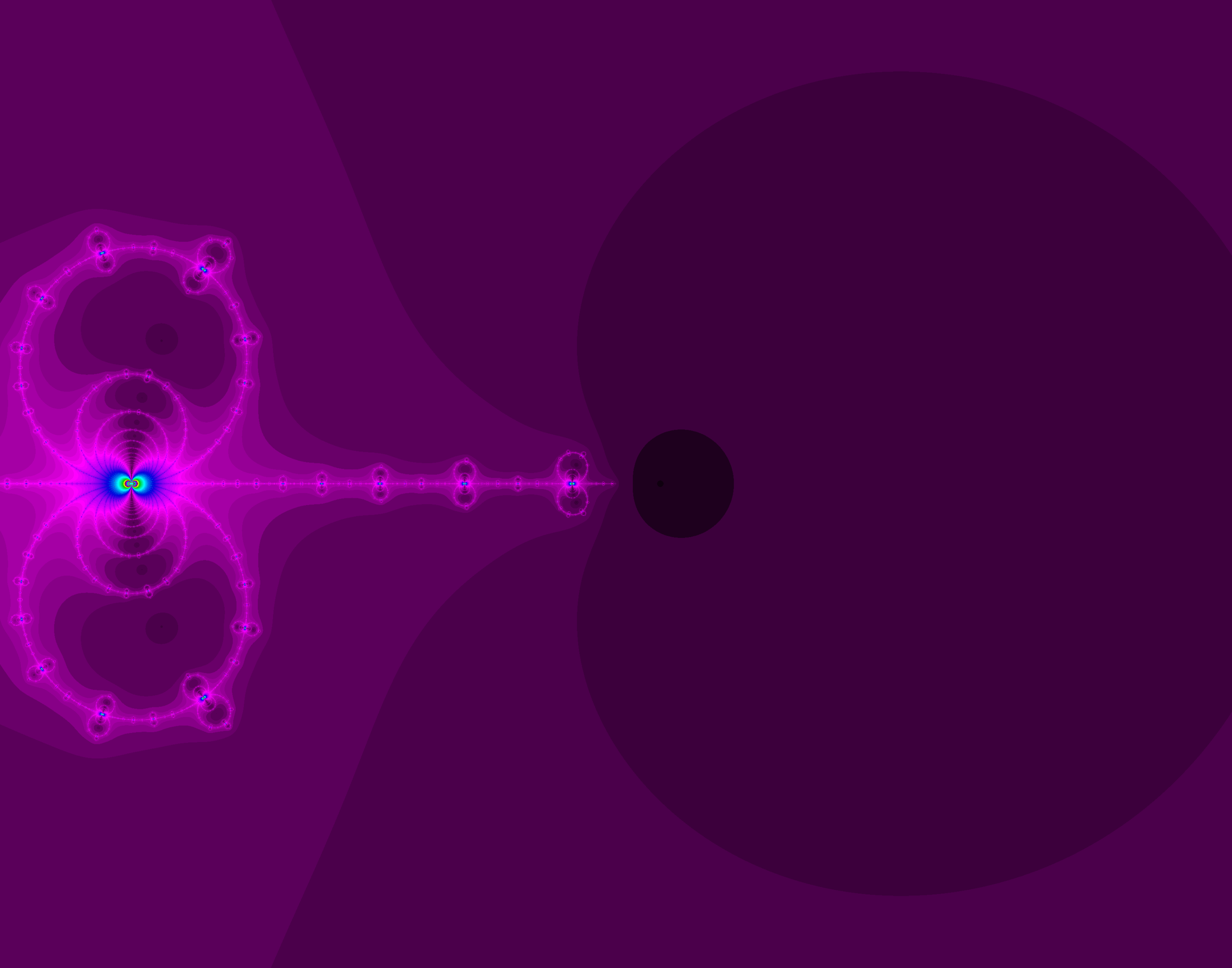}};
        \begin{scope}[x={(image.south east)},y={(image.north west)}]
        \coordinate (a) at ($ ({(1+14)/28},{(0+11)/22}) $);
        \coordinate (b) at ($ ({(-1+14)/28},{(0+11)/22}) $);
        \fill[white!80] (a) circle (0.9pt) node[right]{\tiny$1$};
        \draw[line width=0.48pt,cross,color=white!80] (b) circle (1.2pt);
        \end{scope}
        \end{tikzpicture}
        \hfill\vspace{-1mm}
        \caption{\textit{Left (dynamical plane of $N_\lambda$ for $\lambda=0$)}: The Fatou set of $N_0$ consists only of the superattracting basins $U_k$ of the fixed points at $k\in\ZZ$, with $\{\operatorname{Re}z=\sfrac{1}{2}+k\}_{k\in\ZZ}\subset\mathcal{J}(N_0)$; see \pref{ex:Newton0_tan}. Range: $[-1.5,1.5]\times[-0.75,0.75]$.
        \textit{Right (dynamical plane of its projection $g_\lambda$)}: The superattracting basin of $1$, and the essential singularity $\oplus$ at $-1$. Range: $[-14,14] \hspace{-0.3mm}\times\hspace{-0.3mm} [-11,11]$. The purple colors matches the speed of convergence to $1$ by $g_{0}$ (the lightest covers $\mathcal{J}(g_0)$).}
        \label{fig:Superfg_tan}
    \end{figure}
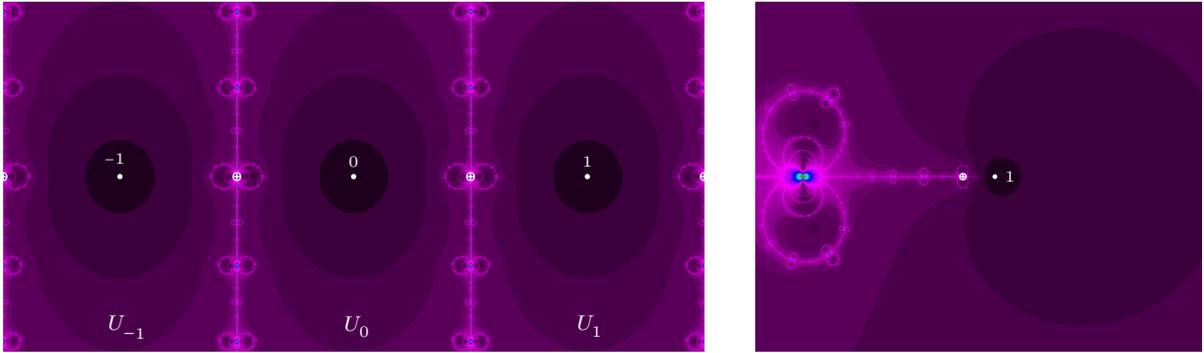
\end{example}
\vspace{-2.0mm}

\subsection{Atlas of wandering domains: \texorpdfstring{$\lambda$}{l}-plane}
\label{sec:5.3_AtlasWD}

The goal of this final section is to present some numerical observations and remarks on the parameter space of our pseudotrigonometric family $\mathbf{N}_\lambda$ of Newton maps (outside of the Eremenko-Lyubich class $\mathcal{B}$). It is convenient to transfer our analysis to their projections $g_\lambda$ via $\exp_1$, as they compose a one-parameter family of finite-type maps with a unique essential singularity and only one free critical orbit (see \pref{lem:PseudoTrig}).

It is of interest to identify the values of $\lambda$ for which the free critical point of $g_\lambda$, given by $C_\lambda=\big(\frac{\lambda-\pi i}{\lambda+\pi i}\big)^2$, does not converge to the superattracting fixed point at $1$ by iteration, since the corresponding Newton maps $N_\lambda$ exhibit other types of Fatou components besides the basins of attraction of the fixed points, such as Baker or wandering domains, which are clear obstructions to root-finding. We denote this set of parameters by \vspace{-0.25mm}
\begin{equation}
    \widetilde{\mathcal{M}}:= \Big\{ \lambda\in\CC\hspace{0.2mm}\backslash\{\pm\pi i\}: \ \lim_{n\to\infty} g_\lambda^n(C_\lambda) \neq 1 \Big\}, \vspace{-0.75mm}
\end{equation}
which is symmetric with respect to both the real and imaginary axes, as $g_{\shortminus\lambda}^{}\big(\frac{1}{w}\big) = \frac{1}{g_{\lambda}^{}(w)}$, and $g_{\shortminus\overline{\lambda}}^{}(\overline{w}) = \overline{g_{\lambda}^{}(w)}$. This is shown in \pref{fig:BakerWDf_fixed}, where we color each $\lambda$ according to the period of the cycle which attracts $C_\lambda$.

Notice that $g_\lambda$ is never hyperbolic (nor topologically hyperbolic) since at least one of its asymptotic values (at $0$ or $\infty$) lies in the Julia set of $g_\lambda$ for all $\lambda$. Nonetheless, we may say that $g_\lambda$ is \textit{subhyperbolic} if the forward orbit of any singular value of $g_\lambda$ is either finite or converges to an attracting periodic cycle, in analogy to the rational case (see \cite[\S19]{Milnor2006}). This allows singular values to be in $\mathcal{J}(g_\lambda)$ only if they are eventually periodic.

A connected component of the subhyperbolic locus of the family $\{g_\lambda^{}\}_{\lambda}^{}$ is called a \textit{subhyperbolic component}. Throughout such a component, the subhyperbolic maps $g_\lambda$ are structurally stable, which means that, roughly speaking, the qualitative dynamics of $g_\lambda$ does not change as we perturb $\lambda$; in particular, the period of the attracting cycle to which $C_\lambda$ converges under iteration is constant (see more details in \cite{Astorg2021,Fagella2021}).

\pref{fig:BakerWDf_fixed} shows notable similarities between the components of $\widetilde{\mathcal{M}}$ and the well-studied hyperbolic components of the Mandelbrot set for the quadratic polynomials (see e.g. \cite{Douady1984}). Furthermore, there is a remarkable elephant-like structure which reminds us of the fractal geometry of the Mandelbrot set near parabolic parameters. The parade of elephants in \pref{fig:MandelFrot2} (blow-up of $\widetilde{\mathcal{M}}$ in the third quadrant of the $\lambda$-plane) seems to be almost invariant under translation by $-1$, similarly to the illustrations in \cite{Branner1998} for the quadratic family. However, in our case it looks like the elephants' trunks actually terminate in neck of the next one, with the right-most one reaching the parameter singularity at $-\pi i$, whose Newton's method $N_{-\pi i}$ is indeed conjugate, via $z\mapsto -2\pi i z$, to a Fatou function of the form $z-1+e^{-z}$ (analyzed in \cite{Weinreich1990}). The rest of trunk endings seems to occur at $\lambda$-values for which the free critical point of $g_\lambda$ at $C_\lambda$ is an (essential) prepole of order $m\geq 1$, i.e. \vspace{-1.25mm}
\begin{equation}
    \label{eq:freeCriticalPrepole}
    g_{\lambda}^m(C_{\lambda}) = \frac{\lambda-\pi i}{\lambda+\pi i}.
\end{equation}\vspace{-3.5mm}

\begin{figure}[H] 
    \centering
    \begin{tikzpicture}
    \node[anchor=south west,inner sep=0] (image) at (0,0) {\includegraphics[width = 0.83\textwidth]{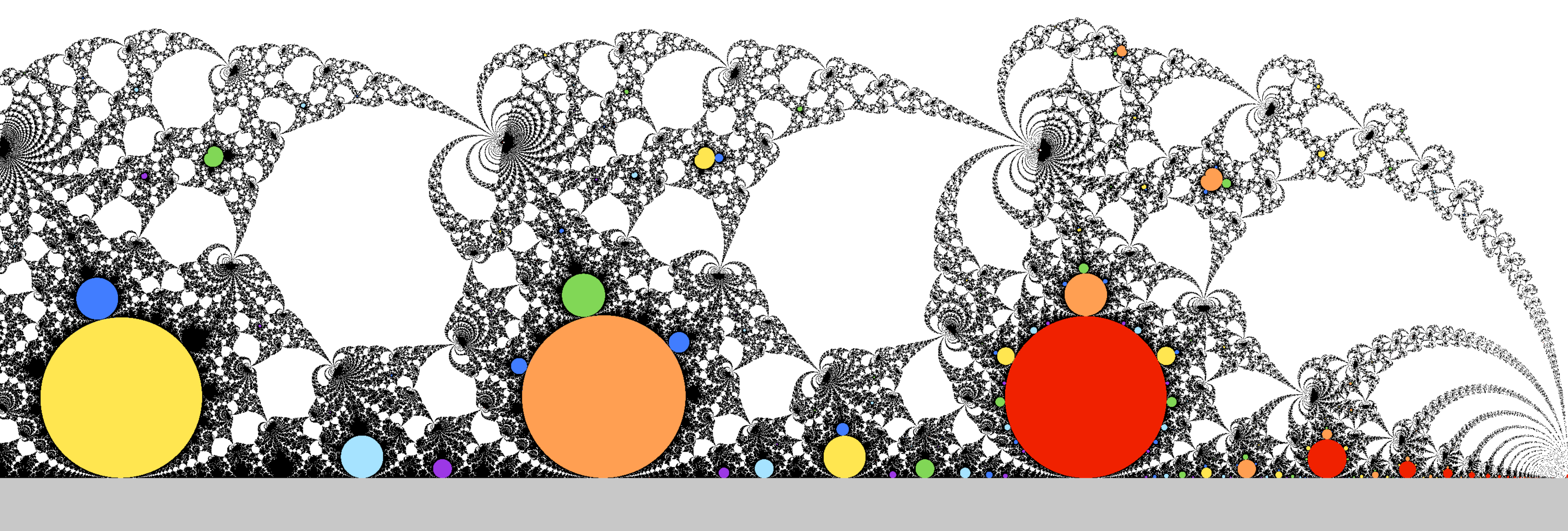}};
    \begin{scope}[x={(image.south east)},y={(image.north west)}]
    \coordinate (a1) at ($ ({(-1+3.25)/3.25},{(-pi+3.25)/1.1}) $);
    \coordinate (a1m2) at ($ ({(-1/2+3.25)/3.25},{(-pi+3.25)/1.1}) $);
    \coordinate (a2) at ($ ({(-2+3.25)/3.25},{(-pi+3.25)/1.1}) $);
    \coordinate (a2m2) at ($ ({(-2/3+3.25)/3.25},{(-pi+3.25)/1.1}) $);
    \coordinate (a3) at ($ ({(-3+3.25)/3.25},{(-pi+3.25)/1.1}) $);
    \coordinate (a3m1) at ($ ({(-3/2+3.25)/3.25},{(-pi+3.25)/1.1}) $);
    \coordinate (c1) at ($ ({(-1+3.25)/3.25},{(-sqrt(pi*pi-1)+3.25)/1.1}) $);
    \coordinate (c1m2) at ($ ({(-1/2+3.25)/3.25},{(-pi+3.25)/1.1}) $);
    \coordinate (c2) at ($ ({(-2+3.25)/3.25},{(-sqrt(pi*pi-1)+3.25)/1.1}) $);
    \coordinate (c2m2) at ($ ({(-2/3+3.25)/3.25},{(-pi+3.25)/1.1}) $);
    \coordinate (c3) at ($ ({(-3+3.25)/3.25},{(-sqrt(pi*pi-1)+3.25)/1.1}) $);
    \coordinate (c3m1) at ($ ({(-3/2+3.25)/3.25},{(-pi+3.25)/1.1}) $);
    \coordinate (b0) at ($ ({(-0.002+3.25)/3.25},{(-pi+3.25)/1.1}) $);
    \coordinate (b1) at ($ ({(-1.096212917+3.25)/3.25},{(-2.4616023296+3.25)/1.1}) $);
    \fill[darkgray!80] (a1) circle (1.25pt) node[below]{\tiny\hspace{-0.16em}$\shortminus 1 \hspace{0.1mm} \shortminus \hspace{0.1mm} \pi i$};
    \fill[darkgray!80] (a1m2) circle (1.25pt); 
    \fill[darkgray!80] (a2) circle (1.25pt) node[below]{\tiny\hspace{-0.16em}$\shortminus 2 \hspace{0.1mm} \shortminus \hspace{0.1mm} \pi i$};
    \fill[darkgray!80] (a2m2) circle (1.25pt);
    \fill[darkgray!80] (a3) circle (1.25pt) node[below]{\tiny\hspace{-0.16em}$\shortminus 3 \hspace{0.1mm} \shortminus \hspace{0.1mm} \pi i$};
    \fill[darkgray!80] (a3m1) circle (1.25pt);
    \draw (c1) node[color=darkgray!80] {\small$\Omega_{0,\shortminus 1}^{\shortminus}$};
    \draw (c1m2) node[color=red!80,below right=-0.85mm] {\tiny$\Omega_{0,\shortminus 2}^{\shortminus}$};
    \draw (c2) node[color=darkgray!80] {\small$\Omega_{\sfrac{1}{2},\shortminus 1}^{\shortminus}$};
    \draw (c2m2) node[color=orange!80,below=-0.6mm] {\tiny$\Omega_{\sfrac{1}{2},\shortminus 2}^{\shortminus}$};
    \draw (c3) node[color=darkgray!80] {\small$\Omega_{\sfrac{2}{3},\shortminus 1}^{\shortminus}$};
    \draw (c3m1) node[color=yellow!80,below=-0.6mm] {\tiny$\Omega_{\sfrac{1}{3},\shortminus 1}^{\shortminus}$};
    \fill[magenta!90] (b0) circle (1.0pt) node[above=-0.6mm]{\tiny$\sfrac{1}{2}$};
    \fill[magenta!90] (b1) circle (1.0pt) node[left=-0.6mm]{\tiny$\lambda_1^{\dagger}$};
    \draw (0.965,0.775) node[color=darkgray!80] {\small$\Omega^\dagger$};
    \draw (0.97,0.05) node[color=darkgray!80] {\small$\Omega^{\shortminus}$};
    \draw[color=cyan!90, line width=0.5pt] ($ ({(-0.985+3.25)/3.25},{(-2.315+3.25)/1.1}) $) rectangle ($ ({(-0.865+3.25)/3.25},{(-2.195+3.25)/1.1}) $);
    \draw[color=cyan!90, line width=0.5pt] ($ ({(-0.89+3.25)/3.25},{(-2.94+3.25)/1.1}) $) rectangle ($ ({(-0.78+3.25)/3.25},{(-2.83+3.25)/1.1}) $);
    \end{scope}
    \end{tikzpicture}\vspace{-1.5mm}
    \caption{Region of the set $\widetilde{\mathcal{M}}$ in the $3$rd quadrant of the $\lambda$-plane (same coloring as in \pref{fig:BakerWDf_fixed}). We indicate on $\partial\Omega^{-}=\{\lambda: \operatorname{Im}\lambda=-\pi\}$ the roots of some components of $\widetilde{\mathcal{M}}$ which emerge from the landing points of internal rays of rational argument $\theta$ in $\Omega^{-}$, denoted by $\Omega_{\theta,k}^{-}$ as described in Remarks \ref{rem:SubhypBD} and \ref{rem:SubhypWD}. The white region consists of capture components, and at $\lambda_1^\dagger\approx -1.096-2.462 i$ the free critical point happens to be an essential prepole of $g_\lambda$.} 
    \label{fig:MandelFrot2}
\end{figure}
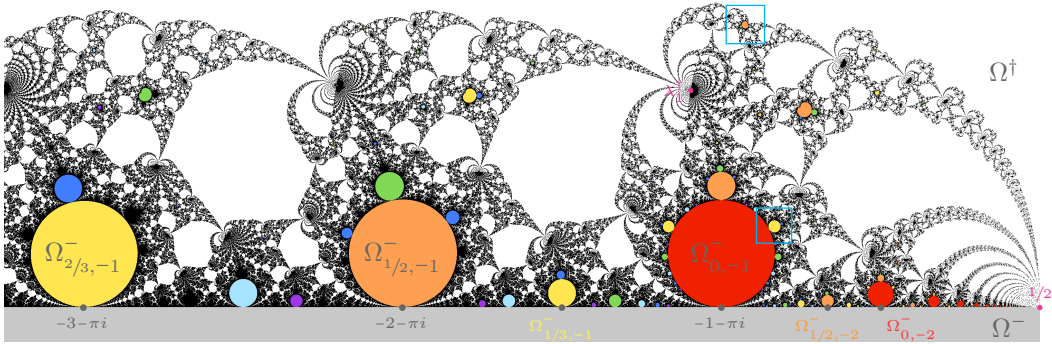
\vspace{-3.5mm}

Notice that all \textit{capture components} (i.e. subhyperbolic components in which the free critical point of $g_\lambda$ at $C_\lambda$ eventually falls in the immediate basin of attraction of the fixed point at $1$) are in the complement of $\widetilde{\mathcal{M}}$, shown in white in Figures \ref{fig:BakerWDf_fixed} and \ref{fig:MandelFrot2}. In particular, the one containing $\lambda=0$, denoted by $\Omega^\dagger$, is the analogue of the outside of the Mandelbrot set since $C_\lambda\in\mathcal{A}^*(1)$; see \pref{ex:Newton0_tan}. In each capture component, there seems to be a distinguished parameter (its \textit{center}) for which the free critical point is eventually fixed, so that its Newton map $N_\lambda$ may be called \textit{postcritically fixed} in analogy to the rational case (see \cite{Drach2019} for a combinatorial classification of them). Moreover, if $C_\lambda$ actually lands on a fixed point after $m+1$ iterations, $m\in\NN$, we have detected a value $\lambda_m^\dagger$ satisfying equation \pref{eq:freeCriticalPrepole}, in the boundary of the associated component.

The root-finding method has virtually no obstacles in capture components, as the Fatou set of these Newton maps consists only of the basins of the roots (and the Julia set has empty interior). Nevertheless, as already mentioned, this will not be the case along the set $\widetilde{\mathcal{M}}$, for instance, if the fixed point of its projection $g_\lambda$ at $0$ (resp. $\infty$) is non-repelling, that is, when $\operatorname{Im}\lambda\geq \pi$ (resp. $\operatorname{Im}\lambda\leq -\pi$) as follows from \pref{lem:PseudoTrig}.

\begin{remark}[Components of $\widetilde{\mathcal{M}}$ leading to Baker domains]
    \label{rem:SubhypBD}
    Denote by $\Omega^+$ (resp. $\Omega^{-}$) the subhyperbolic component in which the free critical point of $g_\lambda$ is attracted to the asymptotic value $0$ (resp. $\infty$), i.e.
    \begin{equation}
        \Omega^\pm := \big\{ \lambda: \pm \operatorname{Im}\lambda > \pi \big\} \subset \HH^\pm.
    \end{equation}
    Notice that the \textit{multiplier map} $\rho_{\Omega^\pm}^{}: \Omega^\pm \to \DD^*$, which is explicitly given by \pref{eq:MultiProjAVG}, is a universal covering of $\Omega^\pm$, as it happens in the exponential or tangent family for their hyperbolic components (see e.g. \cite{Fagella2021}). However, in our case we do not observe cusps on $\partial\Omega^\pm$ at the ends of the internal rays of argument zero. Recall that the internal rays of argument $\theta\in[0,1)$ in $\Omega^\pm$ are the curves $\mathcal{R}^{\Omega^\pm}_{\theta,k}:(-\infty,0)\to \Omega^\pm$, $k\in\ZZ$, which are sent by the multiplier map to the radial segment $\left\{e^t e^{2\pi i \theta}: t\in (-\infty,0)\right\}$ of $\DD^*$, and satisfy
    \begin{equation}
        \exp^{-1}\hspace{-0.3mm} \circ \hspace{0.3mm}\rho_{\Omega^\pm}^{}\big( \mathcal{R}^{\Omega^\pm}_{\theta,k}(t) \big) = t+2\pi i (\theta+k).
    \end{equation}
    It is easy to check that, for each $k\in\ZZ$, $\mathcal{R}^{\Omega^\pm}_{\theta,k}$ is the half-circle starting at $\pm\pi i\in\partial\Omega^\pm$ (the \textit{virtual center}) and ending at $\pm\big(\pi i-\frac{1}{\theta+k}\big)$, which degenerates to the half-line $\{\operatorname{Re}\lambda=0, \pm\operatorname{Im}\lambda>\pi\}$ if $\theta=k=0$. In \pref{fig:MandelPlogD} we show rays of argument $\theta=0$ in $\widetilde{\mathcal{M}}$, after the change of parameter $\lambda\mapsto\mu:=\frac{\lambda+2\pi i}{\lambda}$ which sends $\Omega^{-}$ (resp. $\Omega^{+}$) to the unit disk $\DD$ (resp. $\DD+2$); as a reference, $-\pi i\mapsto -1$, $-\pi(1+i)\mapsto -i$, $\infty\mapsto 1$.

    \vspace{-1.5mm}
    \begin{figure}[H] 
        \centering
        \begin{tikzpicture}
        \node[anchor=south west,inner sep=0] (image) at (0,0) {\includegraphics[width = 0.63\textwidth]{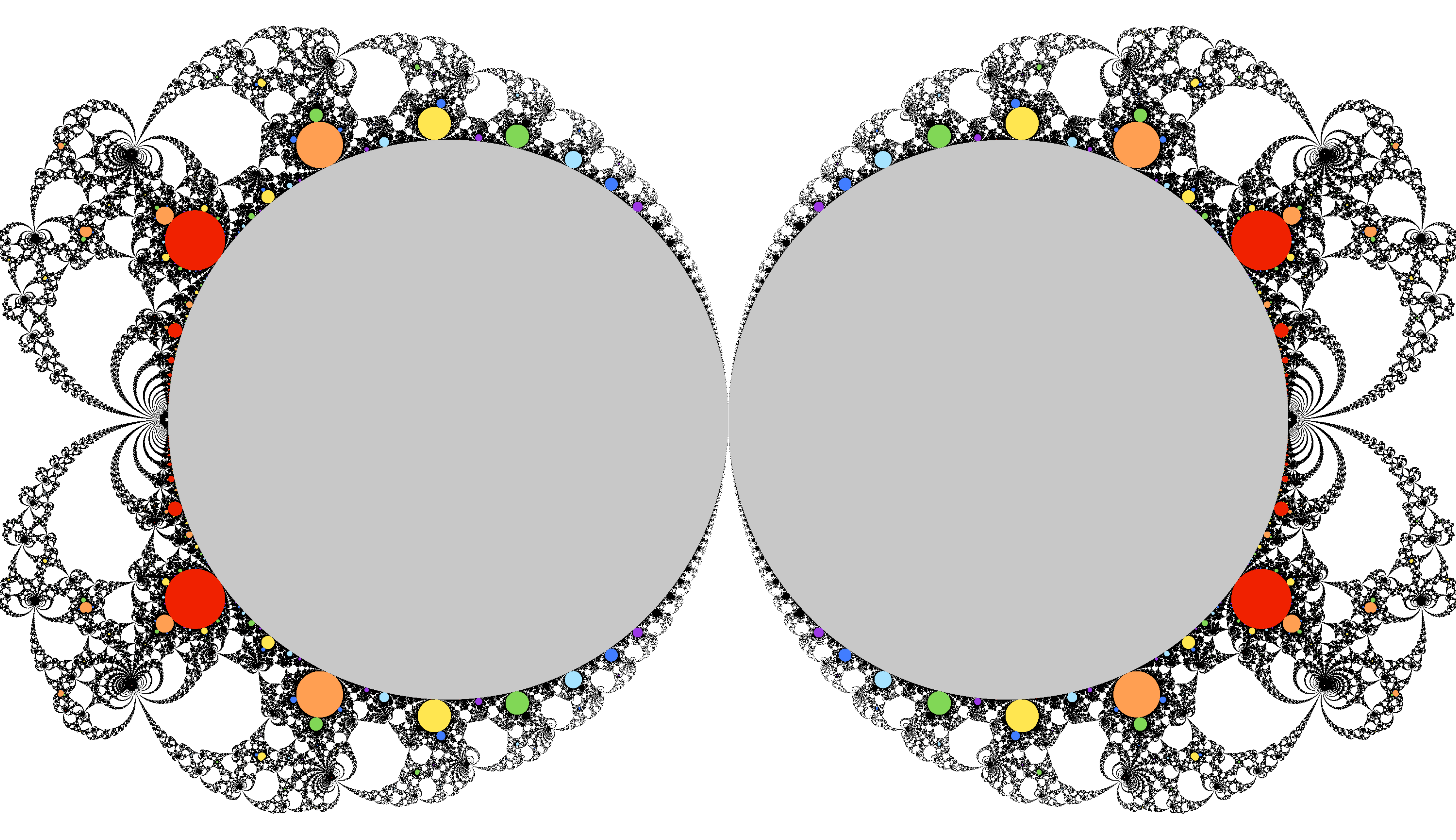}};
        \begin{scope}[x={(image.south east)},y={(image.north west)}]
        \coordinate (a1) at ($ ({(-1+1.6)/5.2},{(0+1.5)/3}) $);
        \coordinate (a2) at ($ ({(1+1.6)/5.2},{(0+1.5)/3}) $);
        \coordinate (a3) at ($ ({(3+1.6)/5.2},{(0+1.5)/3}) $);
        \draw[line width=0.3pt, color = darkgray!80, densely dashed] ($ ({(-1+1.6)/5.2},{(0+1.5)/3}) $) -- ($ ({(3+1.6)/5.2},{(0+1.5)/3}) $);
        \draw ($ ({(0.04+1.6)/5.2},{(0.2+1.5)/3}) $) node[color=darkgray!80] {\small$\Omega^{\shortminus}$};
        \draw ($ ({(2.03+1.6)/5.2},{(0.2+1.5)/3}) $) node[color=darkgray!80] {\small$\Omega^{\shortplus}$};
        \fill[magenta!80] (a1) circle (1.5pt);
        \fill[darkgray!80] (a2) circle (1.5pt) node[below=3mm]{\small$1$};
        \fill[magenta!80] (a3) circle (1.5pt);
    \draw[samples=205,color=darkgray!80,domain=0.5:0,line width=0.3pt, densely dashed] plot ({( (4*pi*sin(deg(2*pi*\x))-cos(deg(2*pi*\x))+6*pi*pi+1)/(2*pi*sin(deg(2*pi*\x))-cos(deg(2*pi*\x))+2*pi*pi+1)  +1.6)/5.2}, {( -(4*pi*sin(deg(pi*\x))*sin(deg(pi*\x)))/(2*pi*sin(deg(2*pi*\x))-cos(deg(2*pi*\x))+2*pi*pi+1) +1.5)/3}) node[below left=7.5mm]{\tiny$\mathcal{R}^{\Omega^{\shortplus}}_{0,1}$};
    \draw[samples=205,color=darkgray!80,domain=-0.5:0,line width=0.3pt, densely dashed] plot ({( (4*pi*sin(deg(2*pi*\x))+cos(deg(2*pi*\x))-6*pi*pi-1)/(2*pi*sin(deg(2*pi*\x))+cos(deg(2*pi*\x))-2*pi*pi-1)  +1.6)/5.2}, {( -(4*pi*sin(deg(pi*\x))*sin(deg(pi*\x)))/(2*pi*sin(deg(2*pi*\x))+cos(deg(2*pi*\x))-2*pi*pi-1) +1.5)/3});
    
    \draw[samples=205,color=darkgray!80,domain=0.5:0,line width=0.3pt, densely dashed] plot ({( (8*pi*sin(deg(2*pi*\x))-cos(deg(2*pi*\x))+24*pi*pi+1)/(4*pi*sin(deg(2*pi*\x))-cos(deg(2*pi*\x))+8*pi*pi+1)  +1.6)/5.2}, {( -(8*pi*sin(deg(pi*\x))*sin(deg(pi*\x)))/(4*pi*sin(deg(2*pi*\x))-cos(deg(2*pi*\x))+8*pi*pi+1) +1.5)/3});
    \draw[samples=205,color=darkgray!80,domain=-0.5:0,line width=0.3pt, densely dashed] plot ({( (8*pi*sin(deg(2*pi*\x))+cos(deg(2*pi*\x))-24*pi*pi-1)/(4*pi*sin(deg(2*pi*\x))+cos(deg(2*pi*\x))-8*pi*pi-1)  +1.6)/5.2}, {( -(8*pi*sin(deg(pi*\x))*sin(deg(pi*\x)))/(4*pi*sin(deg(2*pi*\x))+cos(deg(2*pi*\x))-8*pi*pi-1) +1.5)/3});
    \draw[samples=205,color=darkgray!80,domain=0.5:0,line width=0.3pt, densely dashed] plot ({( -(cos(deg(2*pi*\x))+2*pi*pi-1)/(2*pi*sin(deg(2*pi*\x))-cos(deg(2*pi*\x))+2*pi*pi+1)  +1.6)/5.2}, {( -(4*pi*sin(deg(pi*\x))*sin(deg(pi*\x)))/(2*pi*sin(deg(2*pi*\x))-cos(deg(2*pi*\x))+2*pi*pi+1) +1.5)/3}) node[below right=7.05mm]{\tiny$\mathcal{R}^{\Omega^{\shortminus}}_{0,\shortminus 1}$};
    \draw[samples=205,color=darkgray!80,domain=-0.5:0,line width=0.3pt, densely dashed] plot ({( (cos(deg(2*pi*\x))+2*pi*pi-1)/(2*pi*sin(deg(2*pi*\x))+cos(deg(2*pi*\x))-2*pi*pi-1)  +1.6)/5.2}, {( -(4*pi*sin(deg(pi*\x))*sin(deg(pi*\x)))/(2*pi*sin(deg(2*pi*\x))+cos(deg(2*pi*\x))-2*pi*pi-1) +1.5)/3});
    
    \draw[samples=205,color=darkgray!80,domain=0.5:0,line width=0.3pt, densely dashed] plot ({( -(cos(deg(2*pi*\x))+8*pi*pi-1)/(4*pi*sin(deg(2*pi*\x))-cos(deg(2*pi*\x))+8*pi*pi+1)  +1.6)/5.2}, {( -(8*pi*sin(deg(pi*\x))*sin(deg(pi*\x)))/(4*pi*sin(deg(2*pi*\x))-cos(deg(2*pi*\x))+8*pi*pi+1) +1.5)/3});
    \draw[samples=205,color=darkgray!80,domain=-0.5:0,line width=0.3pt, densely dashed] plot ({( (cos(deg(2*pi*\x))+8*pi*pi-1)/(4*pi*sin(deg(2*pi*\x))+cos(deg(2*pi*\x))-8*pi*pi-1)  +1.6)/5.2}, {( -(8*pi*sin(deg(pi*\x))*sin(deg(pi*\x)))/(4*pi*sin(deg(2*pi*\x))+cos(deg(2*pi*\x))-8*pi*pi-1) +1.5)/3});

        \end{scope}
        \end{tikzpicture}\vspace{-1mm}
        \caption{The set $\mathcal{\widetilde{M}}$ for the parameter $\mu:= \frac{\lambda+2\pi i}{\lambda}$ (same coloring as in Figures \ref{fig:BakerWDf_fixed} and \ref{fig:MandelFrot2}). We indicate the internal rays $\mathcal{R}^{\Omega^\pm}_{0,k}$ of argument zero in $\Omega^\pm$ for $k\in\{0,\pm 1,\pm 2\}$, following \pref{rem:SubhypBD}.} 
        \label{fig:MandelPlogD}
    \end{figure}
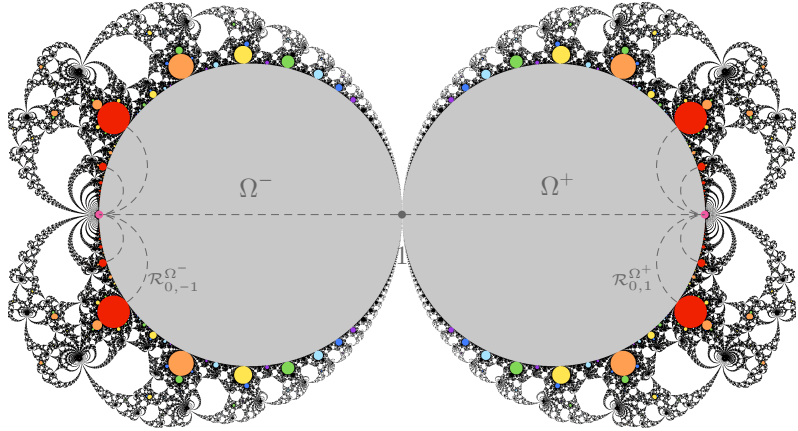
    \vspace{-2.5mm}

    \noindent In this situation, i.e. when $\lambda\in\Omega^-$ (resp. $\lambda\in\Omega^+$), we know that the basin of attraction $V$ of $0$ (resp. $\infty$), which is a Picard exceptional value, lifts via $\exp_1$ to a simply-connected Baker domain of the Newton map $N_\lambda$ (see \pref{ex:BakerWD_SuperCoexistence}) of infinite degree, as $V$ contains a logarithmic tract of $g_\lambda$ by \pref{rem:logarithmicSing}. We note that the landing points on $\partial\Omega^\pm$ of internal rays of rational (resp. Brjuno-type) arguments, give rise to a wandering (resp. Baker) domain of $N_\lambda$ as follows from \pref{thm:logFatouJulia} (see also the arguments in \pref{ex:LiftParabolicBR}).
\end{remark}

We point out that the internal rays of argument $\theta=0$ in $\Omega^\pm$ do not terminate at a cusp, but rather at the root of a component of $\widetilde{\mathcal{M}}$ of period $1$, which we denote by $\Omega_{0,k}^\pm$ if it is attached to $\Omega^\pm$ at $\mathcal{R}^{\Omega^\pm}_{0,k}(1)$, $k\in\ZZ^*$.

\begin{remark}[Some components of $\widetilde{\mathcal{M}}$ leading to wandering domains]
    \label{rem:SubhypWD}
    We observe that for every $\lambda\in\Omega_{0,k}^\pm$, $k\in\ZZ^*$, the free critical point of $g_\lambda$ is attracted to a fixed point $w_{\pm k}^*\neq 1$ of the form \pref{eq:ProjPseudoFixedG}, which is clearly the projection under $\exp_1$ of a $(1,\pm k)$-pseudoperiodic point of the Newton map $N_\lambda$. In fact, for any $\sigma\in\ZZ^*$, $g_\lambda'(w_{\sigma}^*)=1$ if and only if $\lambda\in\big\{\mathcal{R}^{\Omega^+}_{0,\sigma}(1),\mathcal{R}^{\Omega^-}_{0,-\sigma}(1)\big\}$, which means that a transcritical bifurcation between such a fixed point and one of the fixed asymptotic values of $g_\lambda$ occurs at the ends of rays of argument zero in $\Omega^\pm$.

    \noindent The multiplier map \vspace{-1mm}
    \begin{equation}
        \label{eq:multsubhypWD}
        \rho_{\Omega_{0,k}^\pm}^{}(\lambda):= g_\lambda'(w_{\pm k}^*), \quad \mbox{ where } \quad w_{\pm k}^*= \frac{1\pm (\lambda-\pi i)k}{1\pm (\lambda+\pi i)k}, \vspace{-1mm}
    \end{equation}
    provides a foliation of these subhyperbolic components by rays (those mapped to radial segments of $\DD$), which can be derived from \pref{eq:ProjPseudoFixedG} as usual. Note that $C_\lambda=w_{\sigma}^*$ at the center of $\Omega_{0,\pm\sigma}^\pm$, which is exactly the value $\lambda_{\sigma}^\pm$ given by \pref{eq:WD_fixedOmega1} in \pref{ex:WD_SuperCoexistence}, with $C_\lambda\in\RR_{}^{-}$. In analogy to the polynomial case (see e.g. \cite{Milnor2000}), for each $q\in\NN^*$, the internal rays of rational argument $\sfrac{s}{q}$ (in lowest terms, with $s>0$) in $\Omega_{0,k}^\pm$, land at period $q$-tupling bifurcation parameters, which are roots of (satellite) subhyperbolic component of $\widetilde{\mathcal{M}}$ of period $q$.

    \noindent We observe that the $q$-periodic cycle which attracts the free critical point at $C_\lambda$ in these satellite components, lifts via $\exp_1$ to a collection of pseudoperiodic points of type $(q,\pm q k)$ of the Newton map $N_\lambda$. For example, when $\lambda$ crosses the value at the end of the internal ray of argument $\sfrac{1}{3}$ in $\Omega_{0,-1}^-$ (indeed landing at $-\pi i-1$ by \pref{rem:SubhypBD}), the fixed point at $w_{1}^*$ becomes repelling and $C_\lambda$ is then attracted to a $3$-periodic cycle. The points in this $3$-cycle happen to be the projection of $(3,1)$-pseudoperiodic points of $N_\lambda$ (escaping to $\infty$ under iteration; see \pref{prop:IteratesPseudo}). Hence, the immediate basin of attraction of such a cycle lifts to a chain of wandering domains of $N_\lambda$, which looks like a menagerie of Douady rabbits as shown in \pref{fig:furtherNf} (\textit{left}).
\end{remark}

\vspace{-1.5mm}
\begin{figure}[H] 
    \centering
    \hfill
    \begin{tikzpicture}
    \node[anchor=south west,inner sep=0] (image) at (0,0) {\includegraphics[width = 0.48\textwidth]{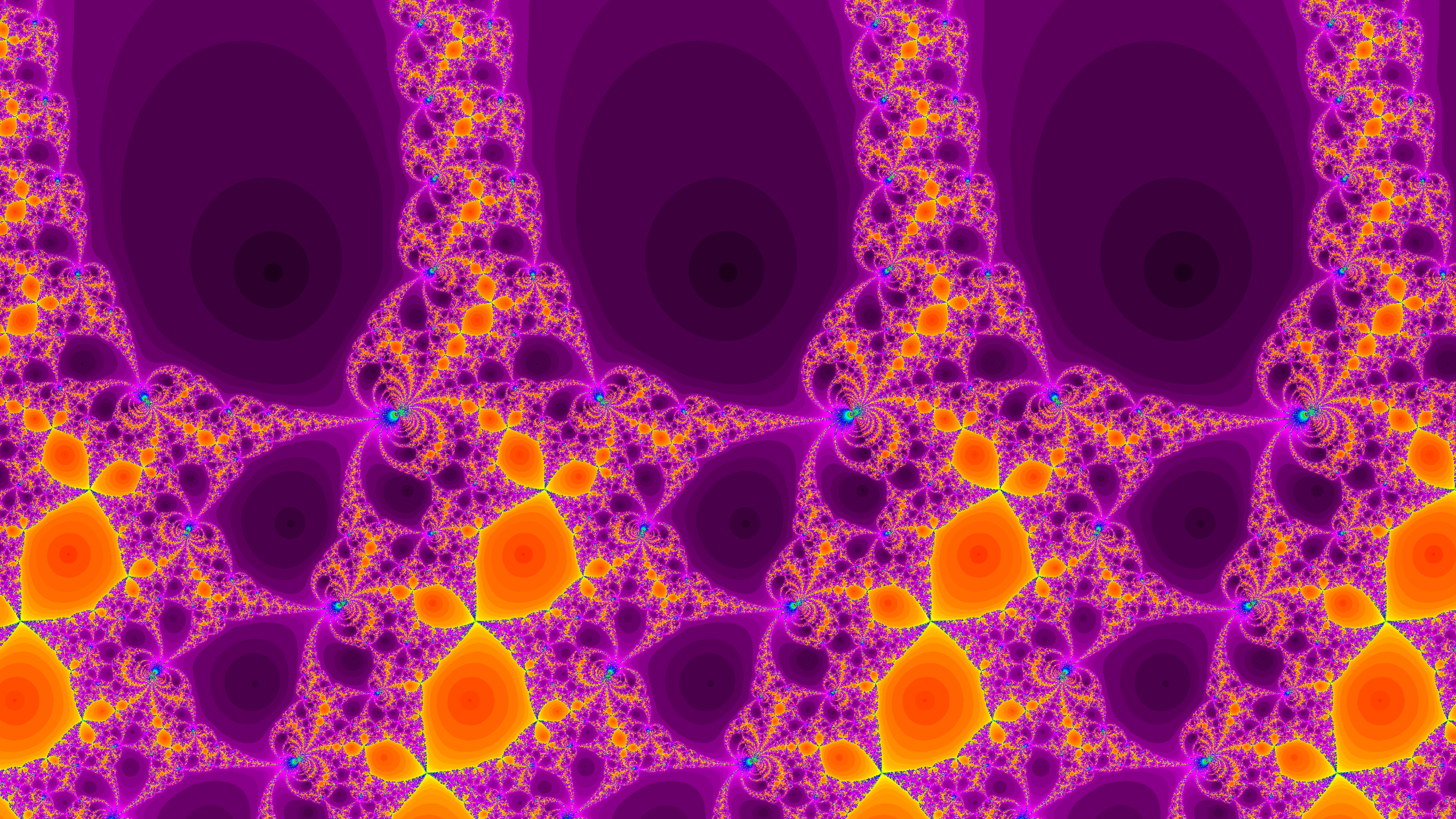}};
    \begin{scope}[x={(image.south east)},y={(image.north west)}]
        \coordinate (a) at ($ ({(0+1.6)/3.2},{(0+1.2)/1.8}) $);
        \coordinate (a2) at ($ ({(1+1.6)/3.2},{(0+1.2)/1.8}) $);
        \coordinate (a3) at ($ ({(-1+1.6)/3.2},{(0+1.2)/1.8}) $);
        \coordinate (b1) at ($ ({(-1.450175+1.6)/3.2},{(-0.619038+1.2)/1.8}) $);
        \coordinate (b2) at ($ ({(-0.328572+1.6)/3.2},{(-0.448625+1.2)/1.8}) $);
        \coordinate (b3) at ($ ({(0.541862+1.6)/3.2},{(-0.399623+1.2)/1.8}) $);
        \coordinate (b4) at ($ ({(1.549825+1.6)/3.2},{(-0.619038+1.2)/1.8}) $);
        \fill[white!80] (a) circle (0.9pt) node[above]{\tiny$0$};
        \fill[white!80] (a2) circle (0.9pt) node[above]{\tiny$1$};
        \fill[white!80] (a3) circle (0.9pt) node[above]{\tiny\hspace{-0.55em}$\shortminus 1$};
        \draw[line width=0.5pt,color=white,-stealth] (b1) to[bend right] (b2);
        \draw[line width=0.5pt,color=white,-stealth] (b2) to[bend right] (b3);
        \draw[line width=0.5pt,color=white,-stealth] (b3) to[bend right] (b4);
        \draw (b1) node[color=black,scale=1] {\tiny$\star$} node[color=black,below]{\tiny$\widetilde{z}\hspace{0.1mm} \ \ $};
        \draw (b2) node[color=black,scale=1] {\tiny$\star$};
        \draw (b3) node[color=black,scale=1] {\tiny$\star$};
        \draw (b4) node[color=black,scale=1] {\tiny$\star$} node[color=black,below]{\tiny$\widetilde{z}\shortplus 3 \quad $};
        \end{scope}
    \end{tikzpicture}
    \hfill
    \begin{tikzpicture}
    \node[anchor=south west,inner sep=0] (image) at (0,0) {\includegraphics[width = 0.48\textwidth]{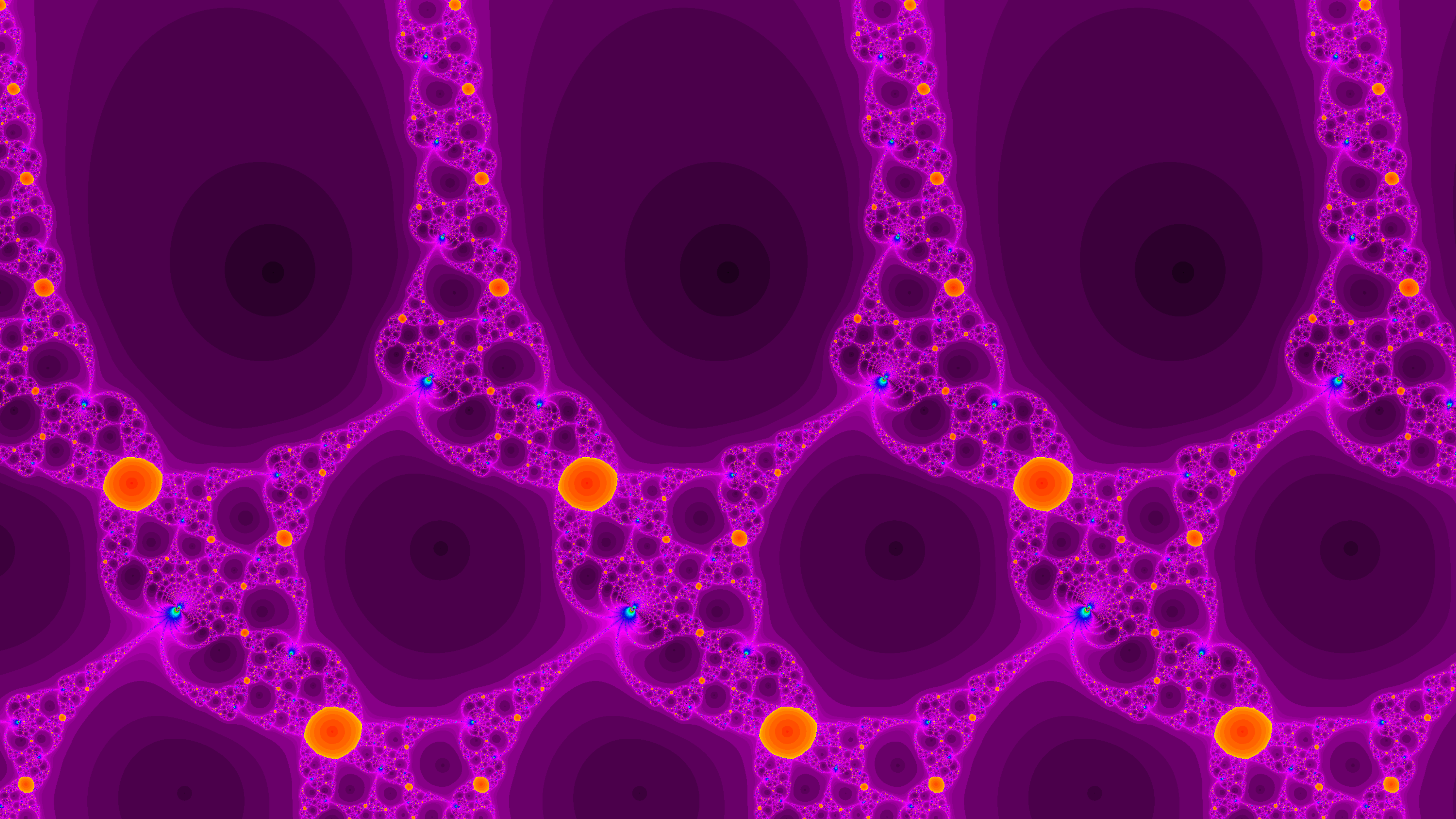}};
    \begin{scope}[x={(image.south east)},y={(image.north west)}]
        \coordinate (a) at ($ ({(0+1.6)/3.2},{(0+1.2)/1.8}) $);
        \coordinate (a2) at ($ ({(1+1.6)/3.2},{(0+1.2)/1.8}) $);
        \coordinate (a3) at ($ ({(-1+1.6)/3.2},{(0+1.2)/1.8}) $);
        \coordinate (b1) at ($ ({(-0.31074+1.6)/3.2},{(-0.462754+1.2)/1.8}) $);
        \coordinate (c1) at ($ ({(0.284204+1.6)/3.2},{(-0.101985+1.2)/1.8}) $);
        \coordinate (bp1) at ($ ({(-0.31074+1+1.6)/3.2},{(-0.462754+1.2)/1.8}) $);
        \coordinate (cp1) at ($ ({(0.284204+1+1.6)/3.2},{(-0.101985+1.2)/1.8}) $);
        \coordinate (bm1) at ($ ({(-0.31074-1+1.6)/3.2},{(-0.462754+1.2)/1.8}) $);
        \coordinate (cm1) at ($ ({(0.284204-1+1.6)/3.2},{(-0.101985+1.2)/1.8}) $);
        \fill[white!80] (a) circle (0.9pt) node[above]{\tiny$0$};
        \fill[white!80] (a2) circle (0.9pt) node[above]{\tiny$1$};
        \fill[white!80] (a3) circle (0.9pt) node[above]{\tiny\hspace{-0.55em}$\shortminus 1$};
        \draw[line width=0.5pt,color=white,-stealth] (b1) to[bend right] (c1);
        \draw[line width=0.5pt,color=white,-stealth] (c1) to[bend right] (b1);
        \draw[line width=0.5pt,color=white,-stealth] (bp1) to[bend right] (cp1);
        \draw[line width=0.5pt,color=white,-stealth] (cp1) to[bend right] (bp1);
        \draw[line width=0.5pt,color=white,-stealth] (bm1) to[bend right] (cm1);
        \draw[line width=0.5pt,color=white,-stealth] (cm1) to[bend right] (bm1);
        \fill[black] (b1) circle (0.6pt);
        \fill[black] (bp1) circle (0.6pt);
        \fill[black] (bm1) circle (0.6pt);
        \fill[black] (c1) circle (0.2pt);
        \fill[black] (cp1) circle (0.2pt);
        \fill[black] (cm1) circle (0.2pt);
        \end{scope}
    \end{tikzpicture}
    \hfill\vspace{-0.5mm}
    \caption{\textit{Left (dynamical plane of $N_\lambda$ for $\lambda\approx -0.833-2.889 i$)}: The basins of attraction (in purple) of the fixed points of $N_\lambda$ at the integers coexist with a chain of simply-connected wandering domains (in orange), containing a $(3,3)$-pseudoperiodic point $\widetilde{z}$ of $N_\lambda$ which projects to some superattracting $3$-periodic point of $g_\lambda$; see \pref{rem:SubhypWD}. 
    \textit{Right (dynamical plane of $N_\lambda$ for $\lambda\approx -0.924-2.256 i$)}: The basins of the fixed points of $N_\lambda$ coexist now with infinitely many $2$-cycles of immediate superattracting basins (in orange). Ranges: $[-1.5,1.5] \hspace{-0.3mm}\times\hspace{-0.3mm} [-1.2,0.6]$. The color palettes refer to the speed of convergence to the periodic points for $g_{\lambda}$. The values of $\lambda$ are, respectively, at the center of the yellow (satellite) component and the orange (primitive) component of $\widetilde{\mathcal{M}}$ inside the cyan squares shown in \pref{fig:MandelFrot2}.}
    \label{fig:furtherNf}
\end{figure}
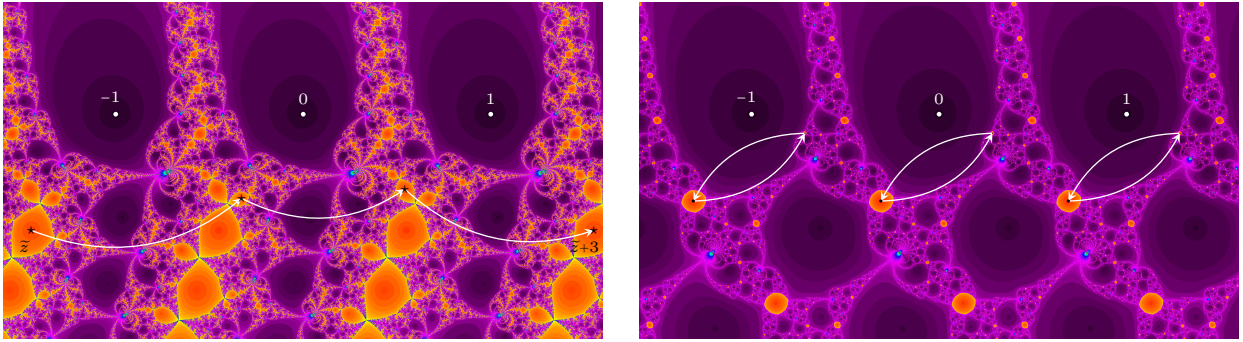\vspace{-3.5mm}

This kind of bifurcation phenomena can be also identified for subhyperbolic components of higher period. Following the previous remarks, we denote by $\Omega_{\sfrac{r}{p},k}^\pm$ the component of $\widetilde{\mathcal{M}}$ of period $p\geq 2$ (see \pref{fig:MandelFrot2}) which emerges from the ray of rational argument $\theta=\sfrac{r}{p}$ (with $1\leq r<p$, and $r$ coprime to $p$) landing at
\begin{equation}
    \label{eq:rationalRayP}
    \mathcal{R}_{\sfrac{r}{p},k}^{\Omega^\pm}(1) = \pm\bigg(\pi i - \frac{p}{r+k p}\bigg).
\end{equation}
Then, as $\lambda$ goes from $\Omega^\pm$ to $\Omega_{\sfrac{r}{p},k}^\pm$ through this value, a Baker domain of the Newton's method $N_\lambda$ turns into a chain of wandering domains, coexisting with the infinitely many basins of its fixed points. Our observations indicate that the bulb $\Omega_{\sfrac{r}{p},k}^\pm$ gives rise to a wandering domain $U$ of $N_\lambda$ such that $N_\lambda^p(U)\subset U\mp (r+pk)$ via the lifting method, which shrinks as it undergoes successive period $q$-tupling bifurcations from there. 

In contrast to subhyperbolic components of period $1$ in $\widetilde{\mathcal{M}}$, which always generate wandering domains for the Newton's method of the entire function $F_\lambda$, the components of higher period $p$ may lead to $p$-cycles of immediate attracting basins, alongside the unbounded invariant basins of the roots of $F_\lambda$, as illustrated in \pref{fig:furtherNf} (\textit{right}) for $p=2$. Our numerical inspection suggests that this occurs for subhyperbolic components of $\widetilde{\mathcal{M}}$ of primitive type whose root (a cusp not lying on the boundary of another component) happens to be accessible from the central capture component $\Omega^\dagger$ (the one containing $\lambda=0$).

We believe that these analogies and observations on $\widetilde{\mathcal{M}}$ are worth further exploration, which is nevertheless out of the scope of this paper.
%
\printbibliography
%
\end{document}
\typeout{get arXiv to do 4 passes: Label(s) may have changed. Rerun}